\documentclass[a4paper,11pt,english]{article}
\usepackage{babel}
\usepackage{authblk} 
\usepackage[latin1]{inputenc}
\usepackage[pagewise]{lineno}

\author[1]{Pablo Cubillos\thanks{pacubill@ucm.es}}
\author[1]{Julián López-Gómez\thanks{jlopezgo@ucm.es}}
\author[2]{Andrea Tellini\thanks{Corresponding author: andrea.tellini@upm.es}}
\affil[1]{\small Universidad Complutense de Madrid\\ Departamento de Análisis Matemático y Matemática Aplicada\\ Plaza de las Ciencias 3\\ 28040 Madrid, Spain}
\affil[2]{\small Universidad Polit\'ecnica de Madrid\\ E.T.S.I.D.I.\\ Departamento de Matem\'atica Aplicada a la Ingenier\'ia Industrial\\ Ronda de Valencia 3\\ 28012 Madrid, Spain}

\title{\textbf{High multiplicity of positive solutions in a superlinear problem of Moore-Nehari type}\thanks{This work has been supported by the Research Project PID2021-123343NB-I00 of the Ministry of Science and Innovation of Spain, by the Institute of Interdisciplinary Mathematics of Complutense University of Madrid. A. T. has also been supported by the INdAM-GNAMPA project \lq\lq Analisi qualitativa di problemi differenziali non lineari" and the Ramón y Cajal program RYC2022-038091-I, funded by MCIN/AEI/10.13039/501100011033 and by the FSE+.}}
\date{\today}

\usepackage{amsmath,amsfonts,amssymb}
\usepackage{yhmath,mathrsfs,yfonts,arcs}
\usepackage{amsthm}
\usepackage{multirow}
\usepackage{booktabs,tabularx}
\usepackage{graphicx}
\graphicspath{{./FiguresMN/},{./FiguresMN_new/1/},{./FiguresMN_new/2/},{./FiguresMN_new/3/},{./FiguresMN_new/4/},{./FiguresMN_new/5/}}
\usepackage[percent]{overpic} 
\usepackage{paralist}
\usepackage[titletoc,title]{appendix}
\usepackage{cite}
\usepackage{leftidx}
\usepackage{soul}
\usepackage{color}
\usepackage{bbm}
\usepackage{tabu} 
\usepackage{tikz,pgfplots,pgf}
\definecolor{fst-blue}{RGB}{0,128,255}

\usepackage{hyperref}

\theoremstyle{plain}
\newtheorem{theorem}{Theorem}[section]

\theoremstyle{definition}

\newtheorem{remark}[theorem]{Remark}

\theoremstyle{remark}
\numberwithin{equation}{section}

\usepackage[margin=0.80in]{geometry}

\newcommand\Item[1][]{%
	\ifx\relax#1\relax  \item \else \item[#1] \fi
	\abovedisplayskip=0pt\abovedisplayshortskip=0pt~\vspace*{-\baselineskip}}

\newcommand{\field}[1] {\mathbb{#1}}

\newcommand{\R}{\field{R}}

\def\a{\alpha}
\def\b{\beta}
\def\e{\varepsilon}
\def\D{\Delta}

\def\G{\Gamma}
\def\l{\lambda}

\def\s{\sigma}
\def\t{\theta}

\def\v{\varphi}

\def\ua{\uparrow}
\def\da{\downarrow}

\def\k{\kappa}

\newcommand{\mc}{\mathcal}

\makeatletter
\newcommand{\pushright}[1]{\ifmeasuring@#1\else\omit\hfill$\displaystyle#1$\fi\ignorespaces}
\newcommand{\pushleft}[1]{\ifmeasuring@#1\else\omit$\displaystyle#1$\hfill\fi\ignorespaces}
\makeatother

\begin{document}
\maketitle

\begin{abstract}
In this paper we consider a superlinear one-dimensional elliptic boundary value problem that generalizes the one studied by Moore and Nehari in \cite{MN}. Specifically, we deal with piecewise-constant weight functions in front of the nonlinearity with an arbitrary number $\k\geq 1$ of vanishing regions. We study, from an analytic and numerical point of view, the number of positive solutions, depending on the value of a parameter $\l$ and on $\k$.
\par
Our main results are twofold. On the one hand, we study analytically the behavior of the solutions, as $\l\da-\infty$, in the regions where the weight vanishes. Our result leads us to conjecture the existence of $2^{\k+1}-1$ solutions for sufficiently negative $\l$. On the other hand, we support such a conjecture with the results of numerical simulations which also shed light on the structure of the global bifurcation diagrams in $\l$ and the profiles of positive solutions.
\par
Finally, we give additional numerical results suggesting that the same high multiplicity result holds true for a much larger class of weights, also arbitrarily close to situations where there is uniqueness of positive solutions.
 \end{abstract}



\section{Introduction}

\label{sec:1}
\noindent In this paper we study the structure of the set of positive solutions for the one-dimensional boundary value problem
\begin{equation}
	\label{1.1}
	\left \{ \begin{array}{ll}
		-u''=\lambda u+a(x) u^3\qquad \hbox{in} \;\; (0,1),\\
		u(0)=u(1)=0, \end{array} \right.
\end{equation}
where $\l\in\R$ is treated as a parameter, and the weight $a\colon [0,1]\to [0,+\infty)$ is a piece-wise continuous function
of the type
\begin{equation}
	\label{1.2}
	a(x)=\left\{ \begin{array}{ll} 0 & \quad \hbox{if}\;\; x\in J_h,\\
		1 & \quad \hbox{if}\;\; x\in [0,1]\setminus J_h,\end{array}\right.
\end{equation}
where $J_h$ is a union of a certain number of intervals, all of them with length $h\in (0,1)$, so that $a$ is symmetric about $0.5$ (see \eqref{1.5} for the precise definition of the weights that we are going to consider). Thus, $a\geq 0$ vanishes in $J_h$ and equals $1$ in $[0,1]\setminus J_h$.

\par
This problem with $\l=0$ and the particular choice $J_h=\left(\frac{1-h}{2},\frac{1+h}{2}\right)$ was introduced in 1959 by E.~A. Moore and Z.~Nehari in \cite{MN}, who established the existence of three positive solutions for some values of $h\in(0,1)$. This is in strong contrast with the classical case when $a\equiv 1$ in $[0,1]$, where \eqref{1.1} becomes the autonomous (conservative) second order boundary value problem
\begin{equation}
	\label{1.3}
	\left \{ \begin{array}{ll}
		-u''=\lambda u+ u^3\qquad \hbox{in} \;\; (0,1),\\
		u(0)=u(1)=0, \end{array} \right.
\end{equation}
which has a unique positive solution $u_\l$ for each $\l<\pi^2$, and does not admit any positive solution if $\l \geq \pi^2$ (see, e.g., L\'opez-G\'omez, Mu\~{n}oz-Hern\'andez and Zanolin \cite[Proposition 2.1]{LGMHZ}). In addition, $(\l,u_\l)$ is a continuous curve such that
\begin{equation*}
\lim_{\l\da -\infty} \|u_\l\|_2=+\infty,\qquad \lim_{\l\ua \pi^2}\|u_\l\|_2=0.
\end{equation*}
Figure \ref{Fig1} shows a plot of the bifurcation diagram of positive solutions of \eqref{1.3}.
In this and all subsequent bifurcation diagrams of this work, we plot the value of the parameter $\l$ in abscissas versus the value of the $L^2$-norm (or the associated discrete version in the case of numerical experiments) of the corresponding solution in ordinates.
Observe, in particular, that positive solutions emanate subcritically from $u=0$ at $\l=\pi^2$.

\begin{figure}[htb]
	\centering
	\begin{tikzpicture}[scale=0.81]
		\begin{axis}[
			tick label style={font=\scriptsize},
			axis y line=middle,
			axis x line=middle,
			xtick={0, 9.87},
			ytick=\empty,
			xticklabels={$0$,$\pi^2$},
			yticklabels={},
			xlabel={\small $\l$},
			ylabel={\small $\|u\|_2$},
			every axis x label/.style={
				at={(ticklabel* cs:1.0)},
				anchor=west,
			},
			every axis y label/.style={
				at={(ticklabel* cs:1.0)},
				anchor=south,
			},
			width=10cm,
			height=6cm,
			xmin=-20,
			xmax=12,
			ymin=-0.8,
			ymax=6]
			\addplot[domain=-20:9.87,color=black,line width=1pt,smooth,samples=200]{sqrt(9.87-x)};
		\end{axis}
	\end{tikzpicture}
	\vspace{-0.2cm}
	\caption{Bifurcation diagram of positive solutions of \eqref{1.3}.}
	\label{Fig1}
\end{figure}
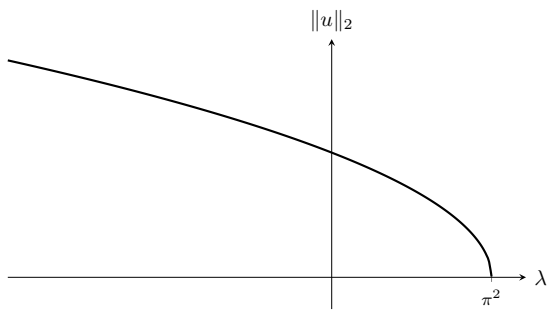

Actually, thanks to a recent result by L\'{o}pez-G\'{o}mez and Sampedro \cite{LGS}, the curve $(\l,u_\l)$ and the remaining continua that form the set of positive solutions of \eqref{1.1}, computed in this paper, consists of a set of analytic arcs plus a discrete set of branching (or bifurcation) points.
\par
Although in this work we focus on positive solutions, we point out that, more recently, the existence of nodal solutions for the following variant of the Moore--Nehari problem
\begin{equation}
	\label{1.4}
	\left \{ \begin{array}{ll}
		-u''=\lambda u+a(x) |u|^{p-1}u \qquad \hbox{in} \;\; (0,1),\\
		u(0)=u(1)=0, \end{array} \right.
\end{equation}
has been analyzed in Gritsans and Sadyrbaev  \cite{GS} for $\l=0$ and $p>1$, in Kajikiya \cite{Ka21} for
$\l=0$ and $0<p<1$, as well as in L\'{o}pez-G\'{o}mez, Mu\~{n}oz-Hern\'{a}ndez and Zanolin \cite{LGMHZ}, where the minimal global bifurcation diagrams of nodal solutions in terms of the parameter $\l$ were obtained by means of some techniques going back to Rabinowitz \cite{Rabgb, Rab71, Rab74} and L\'{o}pez-G\'{o}mez and Rabinowitz \cite{LGR15, LGR17, LGR20}.
\par
Weighted boundary value problems like \eqref{1.1} with $\mathrm{int\,}a^{-1}(0)\neq \emptyset$ are referred to as \emph{degenerate}, of \emph{superlinear} type if $a\gneq 0$, and of \emph{sublinear} type if $a\lneq 0$.
Degenerate problems have attracted a lot of attention during the last few decades starting with Br\'{e}zis and Oswald \cite{BO}, Ouyang \cite{Ou} and Fraile et al. \cite{FKLGM},  as they have been pivotal in developing the theory of metasolutions, which are fundamental for the study of the dynamics of positive solutions in their parabolic counterparts when $a\lneq 0$ (see the monograph  \cite{LGBM}, as well as the list of references therein). Sublinear, possibly degenerate, problems exhibit uniqueness of positive solutions. Instead, the study of nodal solutions for such problems, which has been recently initiated by L\'{o}pez-G\'{o}mez and Rabinowitz \cite{LGR15,LGR16,LGR17,LGR20}, L\'{o}pez-G\'{o}mez, Molina-Meyer and  Rabinowitz \cite{LGMMR}, and the authors \cite{CLGT,CLGTS}, has shown that the number of 1-node solutions increases as soon as the number of components of $a^{-1}(0)$ increases.

\par
One of the main goals of this work is to show that, in the superlinear case $a\gneq 0$, such a complexity already arises for positive solutions, again when one considers weight functions $a(x)$ where $a^{-1}(0)$ has several components. Precisely, we will consider a class of weight functions $a(x)=a_{\kappa,\e}(x)$, which are piece-wise constant, symmetric about $0.5$ and are constructed as follows. First of all, for any given integer $\kappa\geq 1$, we choose $\kappa$ open subintervals of $(0,1)$, $(\a_i,\b_i)$, $i\in\{1,\ldots,\kappa\}$, such that:
\begin{itemize}
	\item $0<\a_1<\b_1<\a_2<\b_2<\cdots<\a_\kappa<\b_\kappa<1$;
	\item $0<h=\b_i-\a_i$ for all $i\in\{1,\ldots,\kappa\}$, i.e., these $\kappa$ intervals have the same length, $h>0$;
	\item $(\a_i,\b_i)$ is the symmetric interval, with respect to $0.5$, of $(\a_{\kappa+1-i},\b_{\kappa+1-i})$ for all $i\in \{1,...,\kappa\}$.
\end{itemize}
Observe that, when $\kappa$ is odd, 0.5 is the midpoint of the central interval $(\a_{\frac{\k+1}{2}},\b_{\frac{\k+1}{2}})$.

Then, once the choice of these intervals has been made, we define the weight functions
\begin{equation}
	\label{1.5}
	a_{\kappa,\e}(x):=\left\{ \begin{array}{ll} 1 &\quad \hbox{if}\;\; x\in[0,1]\setminus \bigcup_{i=1}^\kappa
		(\a_i,\b_i),\\[1ex] \e &\quad \hbox{if}\;\; x\in\bigcup_{i=1}^\kappa (\a_i,\b_i), \end{array}\right.
\end{equation}
where $\e\in[0,1]$ is regarded as a secondary bifurcation parameter for our problem. Using the terminology mentioned above, notice that \eqref{1.1} with $a=a_{\k,\e}$ is superlinear degenerate for $\e=0$, while it is (purely) superlinear if $\e>0$. Moreover,  as $\e$ changes from $0$ to $1$, this parameter establishes a continuous deformation between the degenerate problem and  the classical autonomous problem \eqref{1.3}, with $a\equiv 1$ in $[0,1]$.

The first main result of this work (see Theorem \ref{th2.1}) is a theoretical result which shows that any possible positive solution $u_\lambda$ of problem \eqref{1.1} with $a(x)=a_{\k,0}(x)$, $\k\geq 1$, converges, as $\l\da-\infty$, to $0$ in the region where $a=0$; this allows us to conjecture a multiplicity result for \eqref{1.1} (see Remark \ref{re2.2}).

\par
Then, we will focus on the structure of the global bifurcation diagrams of positive solutions of \eqref{1.1} with $a=a_{\k,\e}$. Ascertaining such a structure analytically is impossible with the existing mathematical tools. For this reason, a systematic numerical study is required both to get a first complete picture of the behavior of positive solutions of \eqref{1.1}, and to give some hints on the kind of results one can expect to prove theoretically, on the base of new techniques. The second main point of this work is the presentation of the results of these numerical experiments (see Sections \ref{sec:3}--\ref{sec:6}).

\par
To perform such a numerical study, we have used some sophisticated numerical path-following techniques, like those pioneered and developed, e.g., by Keller and Yang \cite{KY} and Keller \cite{Ke}, in order to determine the global bifurcation diagrams of positive solutions of \eqref{1.1} using $\lambda$ as the main bifurcation parameter (see Section \ref{sec:7} for the details on the numerical methods used in this work).

\par
As we adopt the point of view of global bifurcation theory, we regard a solution of \eqref{1.1} as a  pair $(\l,u)$ consisting of a value of $\l$ and a piece-wise  $\mc{C}^2$ function satisfying the boundary value problem \eqref{1.1}. Since $(\l,0)$ solves \eqref{1.1} for all $\l\in\R$, it is referred to as the trivial solution. Dealing with symmetric weights $a_{\k,\e}$ simplifies the analysis, since, otherwise, the secondary bifurcations on the bifurcation diagrams will be lost, giving rise to more components and leading to a much more intricate situation. As a theoretical reference to this phenomenon in a problem similar to \eqref{1.1} but with a different weight and different boundary conditions, the reader is referred to L\'opez-G\'omez and Tellini \cite{LGT} and Tellini \cite{T15}, where it has been shown that breaking the symmetry of the weight function produces the split of the bifurcation curves and may give rise to an arbitrarily large number of components in the bifurcation diagrams.
\par
Subsequently, we show that \eqref{1.1} cannot admit a positive solution
if $\l\geq \pi^2$. Indeed suppose that $a\gneq 0$ and that \eqref{1.1} has some positive solution, say $u\gneq 0$. Then,
\begin{equation*}
	(-D^2 -a(x) u^2)u =\l u
\end{equation*}
in $(0,1)$, where  $D^2 u = u''$, and $u(0)=u(1)=0$. Thus, by the uniqueness of the principal eigenvalue,
\begin{equation}
	\label{1.6}
	\l = \s[-D^2-a(x)u^2],
\end{equation}
where, for any given $V\in\mathcal{C}[0,1]$,  $\s[-D^2+V]$ denotes the smallest eigenvalue of the linear eigenvalue problem
\begin{equation*}
\left\{\begin{array}{ll} -\v '' + V \v =\sigma \v & \quad\hbox{in}\;\; (0,1), \\
	\v(0)=\v(1)=0.\end{array}\right.
\end{equation*}
The reader is sent, e.g., to \cite[Ch. 7]{LGBE} for the existence and uniqueness, in a very general setting, of the principal eigenvalue, and to \cite{LG96} and \cite[Ch. 8]{LGBE} for its main properties.  Since $u\gneq 0$ is a principal eigenfunction associated to $\s[-D^2-a(x)u^2]$, $u$ is strongly positive, i.e., $u(x)>0$ for all $x\in (0,1)$, $u'(0)>0$, and $u'(1)<0$. Otherwise, $u=0$ in $[0,1]$ by the fundamental existence and uniqueness theorem for linear second order equations. Therefore, since $au^2\gneq 0$ in $(0,1)$, by the monotonicity of the principal eigenvalue with respect to $V$, we find from \eqref{1.6} that
\begin{equation*}
\l = \s[-D^2-a(x)u^2] < \s[-D^2]=\pi^2,
\end{equation*}
as claimed. The local bifurcation of positive solutions from the trivial solution $(\l,0)$ at $\l=\pi^2$ follows from the main result of Crandall and Rabinowitz \cite{CR1}, while the existence of a global unilateral component follows from  L\'{o}pez-G\'{o}mez \cite[Theorem 6.4.3]{LGBM}, and the existence of a priori bounds for positive solutions of \eqref{1.1}, that are uniform on compact subintervals of $\l\in (-\infty,\pi^2)$, was established in L\'{o}pez-G\'{o}mez, Mu\~{n}oz-Hern\'{a}ndez and Zanolin \cite{LGMHZ} for a general class of weight functions including \eqref{1.5}, by adapting the blowing-up techniques of Amann and L\'{o}pez-G\'{o}mez \cite{ALG}. As a consequence of these results, there exists a component, $\mathscr{C}_0^+$, of positive solutions of
\eqref{1.1} sub-critically bifurcating from $(\l,u)=(\pi^2,0)$
such that
\begin{equation*}
\mathcal{P}_\l\left(\overline{\mathscr{C}_0^+}\right)=(-\infty,\pi^2],
\end{equation*}
where $\mathcal{P}_\l$ denotes the $\l$--projection operator $\mathcal{P}_\l(\l,u)=\l$. This component $\mathscr{C}_0^+$ will be referred to as the  \emph{main component} of the set of positive solutions of \eqref{1.1}. As mentioned above, it contains all positive solutions of \eqref{1.1} when $a\equiv 1$.
\par
To conclude this introduction, we summarize the general structure of this paper. Section \ref{sec:2} is devoted to the proof of the  fact that  any possible positive solution $u_\lambda$ of problem \eqref{1.1} with $a(x)=a_{\k,0}(x)$, $\k\geq 1$, converges almost everywhere to $0$ in $J_h=a^{-1}(0)$ as $\l\da-\infty$. From this result, we are lead to the  conjecture that, for sufficiently negative $\l$, problem \eqref{1.1} with $a(x)=a_{\k,0}(x)$, $\k\geq 1$, possesses $2^{\kappa+1}-1$ positive solutions.
\par
In Sections \ref{sec:3}, \ref{sec:4} and \ref{sec:5} we present the numerical results for $a(x)=a_{\k,0}(x)$, with $\k=1$, $\k=2$ and $\k=3$, respectively. Such simulations shed some light on the global structure of the set of positive solutions of \eqref{1.1}, showing how such a structure strongly depends on the integer $\k$. Moreover, we investigate the role of $h$, the length of the vanishing intervals of the weight, on the global structure of the solution set. All the numerical experiments support the validity of the conjecture of Section \ref{sec:2}. 
\par
In Section \ref{sec:6} we analyze how the global bifurcation diagrams vary as $\e$ increases from 0 to 1, establishing an homotopy from the case $a(x)=a_{1,0}(x)$ analyzed in Section \ref{sec:3} and the classical (autonomous) case 
\[
  a= a_{1,1} \equiv 1.
\]
In particular, we construct some examples of weight functions $a_{1,\e}$, with $\e\in(0,1)$, that exhibit, at least, 5 positive solutions for sufficiently negative $\lambda$. These examples leads us to conjecture that such a multiplicity holds true for all $\e\in(0,1)$. This is surprising for two reasons: first, because the multiplicity conjectured in Section \ref{sec:2}, and observed in Section \ref{sec:3}, only gives the existence of 3 solutions for $\e=0$; second, because for the special choice $a=1$ the model has a unique positive solution for every $\l<\pi^2$. Actually, on the base of our numerical experiments, we conjecture that for $a=a_{\kappa,\e}$, $\kappa\geq 1$, the model has an arbitrarily large number of positive solutions (by taking $\kappa$ sufficiently large) even for  $\e<1$ arbitrarily close to $1$. This illustrates how spatial heterogeneities in superlinear models can substantially increase the complexity of the solution set of \eqref{1.1}.

\par
Finally, in Section \ref{sec:7} we present the numerical methods that have been used in the simulations of this paper, commenting the main technical aspects and issues  that we have faced in their implementation.

\section{Point-wise behavior of positive solutions as $\l \da -\infty$}
\label{sec:2}
In this section, we consider the general case $a=a_{\kappa,0}$, $\kappa\geq 1$, and, for such a weight, we study the limiting behavior of the positive solutions of \eqref{1.1} as $\l\da -\infty$.  Our main result complements Theorem 4.10 of L\'{o}pez-G\'{o}mez \cite{LGBM} and Theorem 3.1 of Fencl and L\'{o}pez-G\'{o}mez \cite{FLGN} by  considering the case when $a(x)$ is given by \eqref{1.5} with $\e=0$. Essentially, it establishes that, regardless the value of $h>0$, as $\l\da -\infty$ the mass of any positive solution $u_{\lambda}$ of \eqref{1.1} must be concentrated in one, or several,  of the intervals where $a(x)=1$, as illustrated by Figure \ref{Fig3}.  It can be stated as follows.

\begin{theorem}
	\label{th2.1}
	Suppose that $a=a_{\kappa,0}$ (see \eqref{1.5}) for some integer $\kappa\geq 1$, and, for  $\l<\pi^2$, let $u_\l$ be a positive solution of \eqref{1.1}. Then,
	\begin{equation*}
		\lim_{\l\da -\infty} u_\l =0 \;\; \hbox{almost everywhere in }\;\; \bigcup_{i=1}^\kappa \left[\a_i,\b_i\right].
	\end{equation*}
\end{theorem}

\begin{proof} Since we are interested in the limit as $\lambda \da-\infty$, throughout this proof we will assume, without loss of generality, that $\lambda<0$. In addition, for  notational convenience, we will set $\b_0:=0$ and $\a_{\k+1}:=1$.  We  divide the proof into two steps: the first one will analyze the problem in the superlinear part, i.e., when $a=1$, with the goal of establishing the existence of a constant $c\geq 0$ such that
\begin{equation}
\label{2.1}
		u_\l(\a_i)\leq\sqrt{-2\lambda+c} \;\; \hbox{and} \;\; u_\l(\b_i)\leq\sqrt{-2\lambda+c} \;\;
     \hbox{for all} \;\; i\in\{1,\dots,\k\}\;\;\hbox{and}\;\;\l<0.
\end{equation}
Then, the second step will analyze the problem in the linear part, i.e., when $a=0$.
\par
\vspace{0.2cm}
\noindent	\emph{Step 1.}  Observe that, in order to prove \eqref{2.1} it suffices to show that, for every  $i\in\{0,\dots,\k\}$, there exists a constant $c_i\geq 0$  such that
\begin{equation}
		\label{2.2}
		u_\l(\a_{i+1})\leq\sqrt{-2\lambda+c_i} \;\; \hbox{and} \;\; u_\l(\b_i)\leq\sqrt{-2\lambda+c_i}\;\;
\hbox{for all}\;\; \l<0,
\end{equation}
and then take
\[
c:=\max_{i\in\{0,\dots,\k\}}c_i.
\]
Thus, we fix $i\in\{0,\dots,\k\}$ and analyze the problem in $[\b_i,\a_{i+1}]$.   Note that
\[
  u_\l(\a_{\kappa+1})=u_\l(1)=0=u_\l(0)=u_\l(\b_\kappa).
\]
Hence, \eqref{2.2} is satisfied at these particular points.
\par
Since $a=1$ in  the interval $[\b_i,\a_{i+1}]$,  $u_\l$ is a positive solution of the autonomous equation
\begin{equation}
		\label{2.3}
		-u'' = \l u + u^3.
\end{equation}
As the potential energy of \eqref{2.3} is given by
\begin{equation*}
	P(u):= \frac{\l u^2}{2}+\frac{u^4}{4},\qquad u\in\R,
\end{equation*}
the phase portrait for the positive solutions of \eqref{2.3} is the one sketched in Figure \ref{Fig2}, where the two non-negative equilibria of the associated first order system have been represented: the saddle point $(0,0)$ and the nonlinear center $(\sqrt{-\l},0)$. The integral curve
	\begin{equation*}
	\frac{v^2}{2}+P(u)=0
	\end{equation*}
	provides us with a homoclinic connection of $(0,0)$ that intersects the $u$-axis at the equilibrium $u=0$ and at  $u=\sqrt{-2\l}$. This homoclinic connection separates closed orbits, which are located inside it and surround the center $(\sqrt{-\l},0)$,  from the exterior trajectories of Figure \ref{Fig2}, i.e. those starting on the $v$-axis at $(0,v_0)$, with $v_0>0$, and reaching their maximal amplitude at some $(u_0,0)$ with $u_0>\sqrt{-2\l}$.
\par
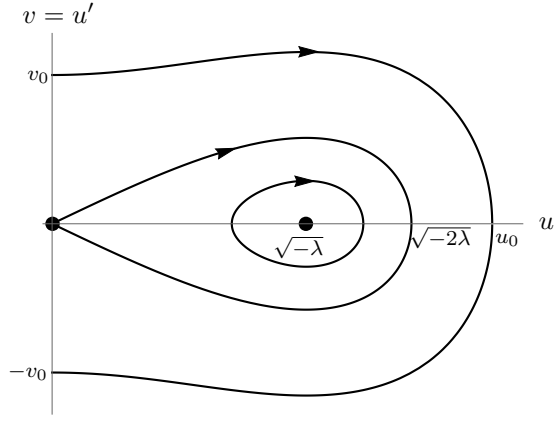
\begin{figure}[ht!]
		\centering
		\begin{overpic}[scale=0.5]{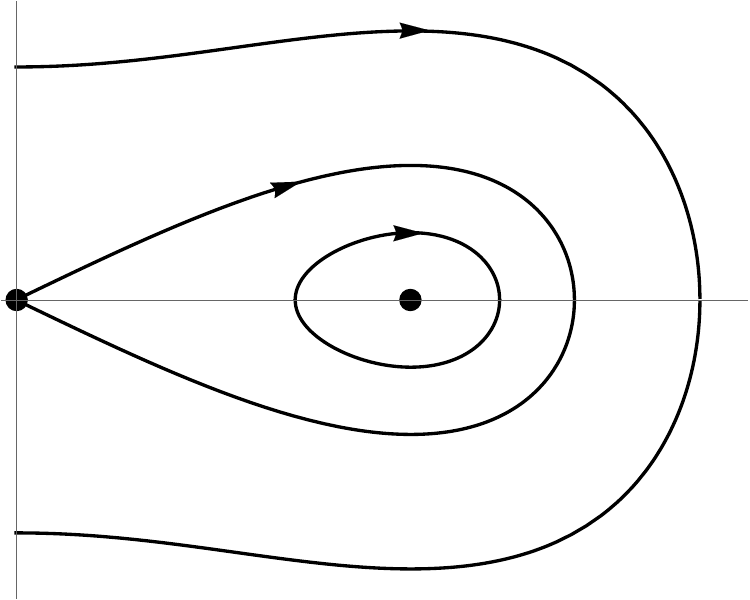}
			\put (103,39) {\small$u$}
			\put (94,36) {\scriptsize$u_0$}
			\put (76.2,35.5) {\scriptsize$\sqrt{-2\l}$}
			\put (48,33.5) {\scriptsize$\sqrt{-\l}$}
			\put (-4,82) {\small$v=u'$}
			\put (-3,70) {\scriptsize$v_0$}
			\put (-7,8) {\scriptsize$-v_0$}
		\end{overpic}
		\vspace{-0.1cm}
		\caption{Phase portrait of positive solutions of \eqref{2.3}.}
		\label{Fig2}
	\end{figure}

Thanks to the conservation of total energy for equation \eqref{2.3}, any solution of \eqref{1.1} lies, in the interval $[\b_i,\a_{i+1}]$, on one of the orbits sketched in Figure \ref{Fig2}. Thus, we can distinguish two possible cases: either (a)
the solution $(u_\lambda(x),u'_\lambda(x))$ lies on the homoclinic connection or inside it, or (b) it lies  outside the homoclinic.
\par
	In case (a), the bound \eqref{2.2}  holds by simply taking $c_i=0$, since the trajectory does not exceed the maximal amplitude of the homoclinic connection, at $u=\sqrt{-2\l}$.
\par
	In case (b), which occurs, for example, in the intervals $[0,\a_1]$ and $[\b_{\k},1]$, due to the boundary conditions $u_\l(0)=u_\l(1)=0$, we consider the external trajectory $\G$ where the solution $(u_\lambda(x),u'_\lambda(x))$ lies for all $x\in[\b_i,\a_{i+1}]$. By simply having a glance at the phase portrait it is apparent that
\begin{equation}
\label{2.4}
		\max_{x\in[\b_i,\a_{i+1}]} u_\l(x) \leq u_0.
\end{equation}
Now, with the notations introduced in Figure \ref{Fig2}, let
\[
   T= T(u_0,\l)>0
\]
denote the necessary time to reach $(u_0,0)$ from $(0,v_0)$ along the trajectory $\G$.
We claim that
\begin{equation}
\label{ii.5}
  \frac{\a_{i+1}-\b_i}{2} \leq T(u_0,\l).
\end{equation}
Arguing by contradiction, suppose that $T< \frac{\a_{i+1}-\b_i}{2}$, i.e.,
\begin{equation}
\label{ii.6}
   \b_i+2T\in(\b_i,\a_{i+1}).
\end{equation}
Then, by symmetry and the definition of $T$, the unique solution, $U$, of the Cauchy problem
	\begin{equation*}
	\left\{\begin{array}{ll} -u''=\l u+ u^3 & \hbox{in}\;\; (\b_i,\a_{i+1}),\\ u(\b_i)=0, \; u'(\b_i)=v_0,\end{array}\right.
	\end{equation*}
satisfies
\begin{equation}
\label{ii.7}
   U(\b_i+2T)=0,\quad U'(\b_i+2T)=-v_0.
\end{equation}
Next, let  $\tau\geq 0$ denote the time necessary to reach  $(u_\l(\b_i),u'_\l(\b_i))$ from $(0,v_0)$ along the trajectory $\G$. As the differential equation is autonomous in $(\b_i,\a_{i+1})$, the uniqueness of solution for the associated Cauchy problems entails that
\begin{equation}
\label{ii.8}
  u_\l(x)=U(x+\tau)\;\; \hbox{for all}\;\; x\in [\b_i,\a_{i+1}].
\end{equation}
By definition, $\tau<2T$. Thus,  \eqref{ii.6} implies that $\b_i+2T-\tau\in(\b_i,\a_{i+1})$.
Therefore, according to \eqref{ii.7} and \eqref{ii.8}, we find that
\[
	u_\l(\b_i+2T-\tau)=U(\b_i+2T)=0,	
\]
which is impossible, because we are assuming that $u_\l(x)>0$ for all $x\in (\b_i,\a_{i+1})$.
Consequently, \eqref{ii.5} holds.
\par
As the time $T$ can be computed from the identities
\begin{equation*}
		T  = \int_0^T \frac{u'(x)}{v(x)}\,dx=\int_0^T \frac{u'(x)}
		{\sqrt{\l\left(u_0^2-u^2(x)\right)+\frac{1}{2}\left(u_0^4-u^4(x)\right) }}\,dx,
\end{equation*}
by making the changes of variable $\xi = u(x)$ and $\xi = u_0\theta$, respectively, we are lead to the following estimate for $T$
\begin{align*}
		T= T(u_0,\l) & = \int_0^{1} \frac{d\theta}
		{\sqrt{\l\left(1-\t^2\right)+\frac{u_0^2}{2}\left(1-\theta^4\right)}}
\\ &  =  \int_0^{1} \frac{d\theta}{\sqrt{1-\t^2}\sqrt{\l+\frac{u_0^2}{2}(1+\theta^2)}} <
	\frac{1}{\sqrt{\l+\frac{u_0^2}{2}}}\int_0^{1} \frac{d\theta}{\sqrt{1-\t^2}}=
	\frac{\pi}{2\sqrt{\l+\frac{u_0^2}{2}}}.
	\end{align*}
Therefore, from \eqref{ii.5} we conclude that, for any positive solution $u_\l$ of \eqref{1.1},
	\begin{equation*}
	\frac{\a_{i+1}-\b_i}{2}  < 	\frac{\pi}{2\sqrt{\l+\frac{u_0^2}{2}}},
	\end{equation*}
	which entails
\begin{equation*}
	u_0< \sqrt{-2\lambda+2\left(\frac{\pi}{\a_{i+1}-\b_i}\right)^2}.
\end{equation*}
Combining this estimate with \eqref{2.4}, provides us \eqref{2.2} also in case (b).
\par
\vspace{0.2cm}

\noindent	\emph{Step 2.}
	Now we fix $i\in\{1,\dots,\k\}$ and consider problem \eqref{1.1} in $[\a_i,\b_i]$, where $u_\l$ satisfies
	\begin{equation*}
	-u''_\l=\l u_\l.
	\end{equation*}
Multiplying this differential equation by the eigenfunction $\v(x)=\sin \frac{\pi(x-\a_i)}{\b_i-\a_i}$, $x\in [\a_i,\b_i]$, and integrating 	in $[\a_i,\b_i]$, we obtain that
	\begin{equation*}
	-\int_{\a_i}^{\b_i} u''_\l\v = \l \int_{\a_i}^{\b_i} u_\l\v.
	\end{equation*}
	Thus,
	\begin{equation*}
	\l  \int_{\a_i}^{\b_i}\! u_\l\v = -\int_{\a_i}^{\b_i}\! (u'_\l\v)'+\int_{\a_i}^{\b_i}\! u'_\l\v' = u'_\l(\a_i)\v(\a_i)-u'_\l(\b_i)\v(\b_i) +
	\int_{\a_i}^{\b_i}\! u'_\l\v'= \int_{\a_i}^{\b_i}\! u'_\l\v',
	\end{equation*}
	because $\v(\a_i)=\v(\b_i)=0$. Similarly, multiplying the differential equation satisfied by $\v$,
	\begin{equation*}
	-\v''=\left( \frac{\pi}{h}\right)^2\v \quad \hbox{in}\;\; (\a_i,\b_i),
	\end{equation*}
	by $u_\l$ and integrating by parts in $(\a_i,\b_i)$, it follows that
	\begin{equation*}
	\left( \frac{\pi}{h}\right)^2 \int_{\a_i}^{\b_i} u_\l\v = -\int_{\a_i}^{\b_i} u_\l\v'' =
	-\int_{\a_i}^{\b_i} (u_\l \v')'+\int_{\a_i}^{\b_i} u_\l'\v'.
	\end{equation*}
	Hence,
	\begin{equation*}
	\l  \int_{\a_i}^{\b_i} u_\l\v=\int_{\a_i}^{\b_i} u'_\l\v'=  \left( \frac{\pi}{h}\right)^2 \int_{\a_i}^{\b_i} u_\l \v
	+u_\l(\b_i)\v'(\b_i)-u_\l(\a_i)\v'(\a_i)
	\end{equation*}
	and,  since $\l<\pi^2<\left(\frac{\pi}{h}\right)^2$, we conclude that
	\begin{equation}
		\label{ii.9}
		\int_{\a_i}^{\b_i} u_\l \v = \frac{ u_\l(\b_i)\v'(\b_i)-u_\l(\a_i)\v'(\a_i)}{\l-\left( \frac{\pi}{h}\right)^2 }.
	\end{equation}
	Letting $\l\da -\infty$ in \eqref{ii.9},  it becomes apparent from \eqref{2.1} that
	\begin{equation*}
	\lim_{\l\da -\infty}\int_{\a_i}^{\b_i} u_\l(x)\v(x)\,dx=0,
	\end{equation*}
	which entails
	\begin{equation*}
	\lim_{\l\da -\infty}u_\l =0 \;\; \hbox{almost everywhere in}\;\; [\a_i,\b_i]
	\end{equation*}
	and ends the proof.
\end{proof}

\begin{remark}
	\label{re2.2} \rm
	As a consequence of Theorem \ref{th2.1}, as $\l\da -\infty$ \eqref{1.1}  splits into $\kappa+1$ different problems, one in each of the intervals where $a=1$. This strongly supports the validity of the conjecture of G\'{o}mez-Re\~{n}asco and L\'{o}pez-G\'{o}mez \cite{GRLGJDE} not only for	general superlinear indefinite problems, but also for the  class of degenerate superlinear weights considered in Theorem \ref{th2.1}. According to this conjecture, \eqref{1.1} should have, at least, $2^{\kappa+1}-1$ positive solutions for sufficiently negative $\l<0$. Actually, since $a\equiv 1$ on the support of $a(x)$, \eqref{1.1} should have, exactly, $2^{\kappa+1}-1$ positive solutions. All the numerical simulations presented in this paper strongly confirm the validity of such a multiplicity result.
	\par
	Actually, we think that the combinatorial argument by Feltrin and Zanolin \cite{FeZa} involving the topological degree can be combined with some techniques in bifurcation theory to prove the validity of this conjecture (see Feltrin, L\'{o}pez-G\'{o}mez and Sampedro \cite{FLGS}).
\end{remark}

\begin{figure}[h!]
	\centering
	\begin{overpic}[scale=0.24,trim = 1cm 5cm 1cm 5cm, clip]{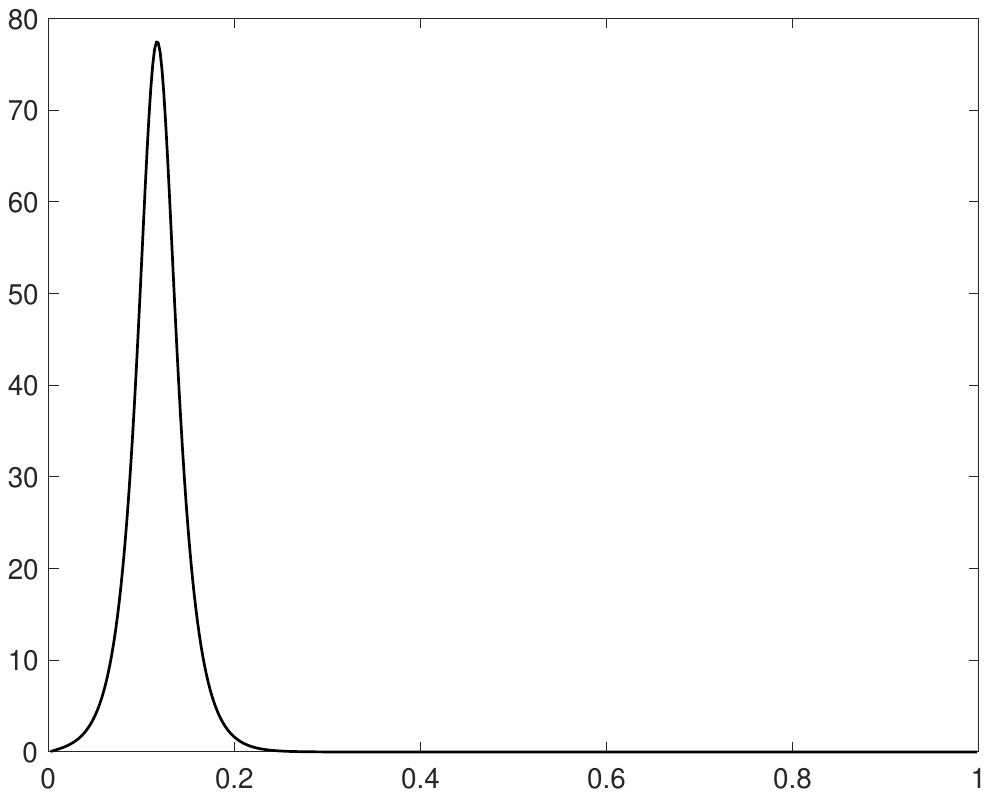}
		\put (11,81) {\tiny$u(x)$}
		\put (93,14) {\tiny$x$}
		\end{overpic}
	\begin{overpic}[scale=0.24,trim = 1cm 5cm 1cm 5cm, clip]{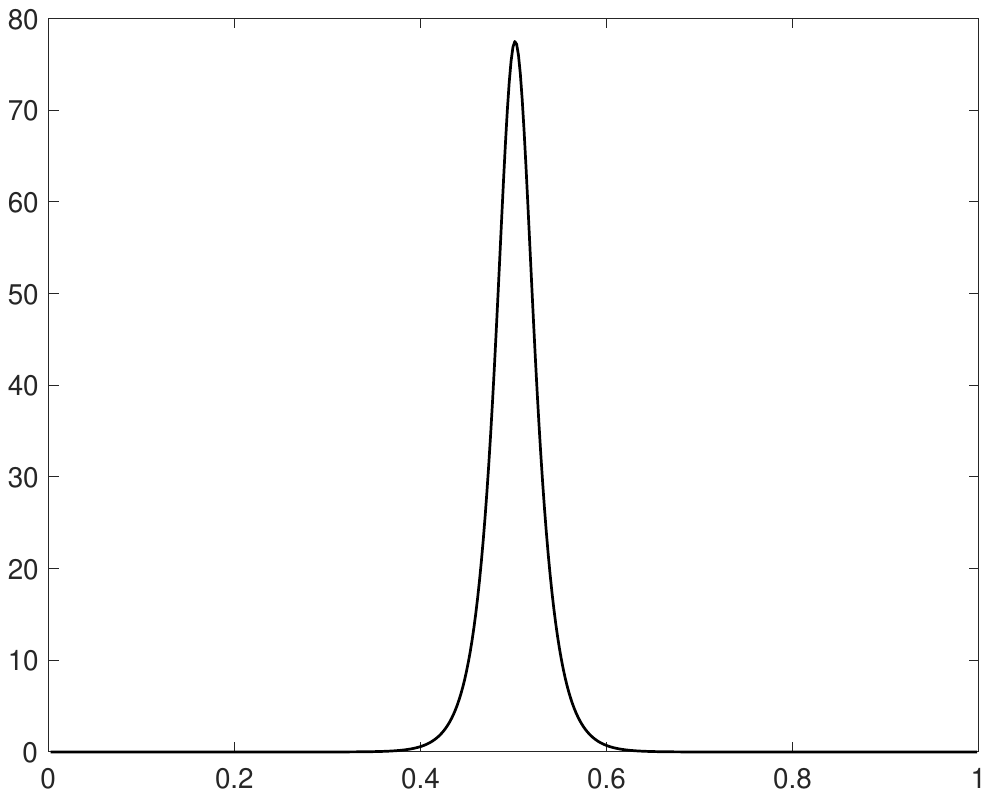}
		\put (11,81) {\tiny$u(x)$}
		\put (93,14) {\tiny$x$}
	\end{overpic}
	\begin{overpic}[scale=0.24,trim = 1cm 5cm 1cm 5cm, clip]{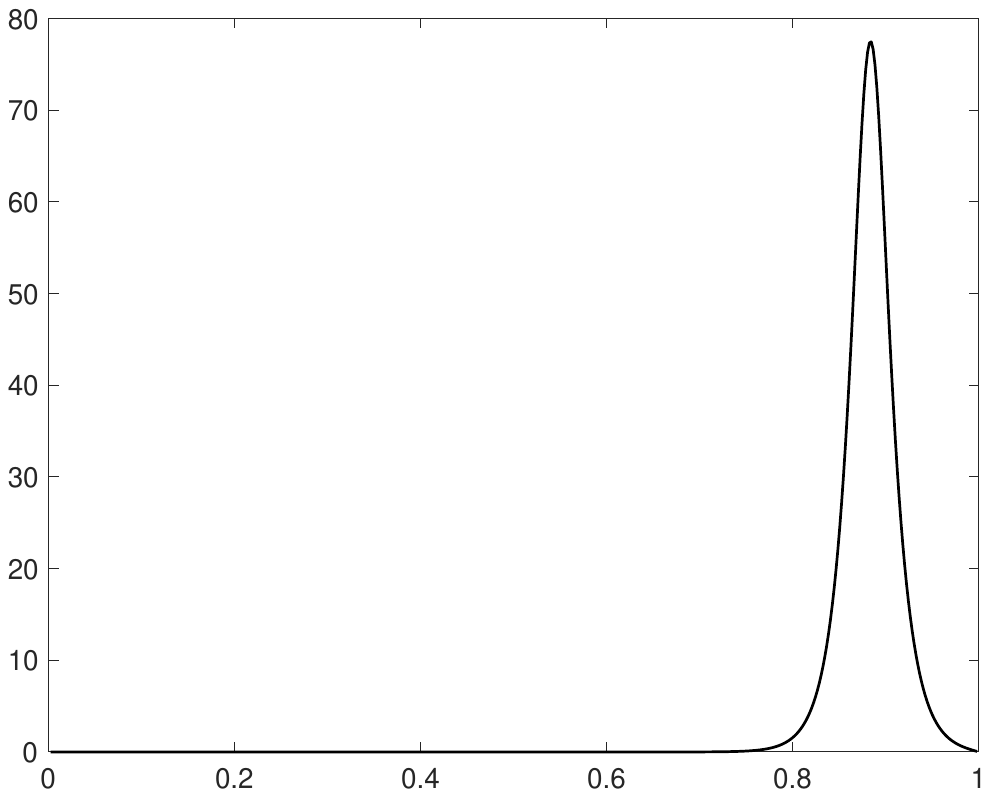}
	\put (11,81) {\tiny$u(x)$}
	\put (93,14) {\tiny$x$}
\end{overpic}
	\\[-2.5em]
	\begin{overpic}[scale=0.24,trim = 1cm 5cm 1cm 5cm, clip]{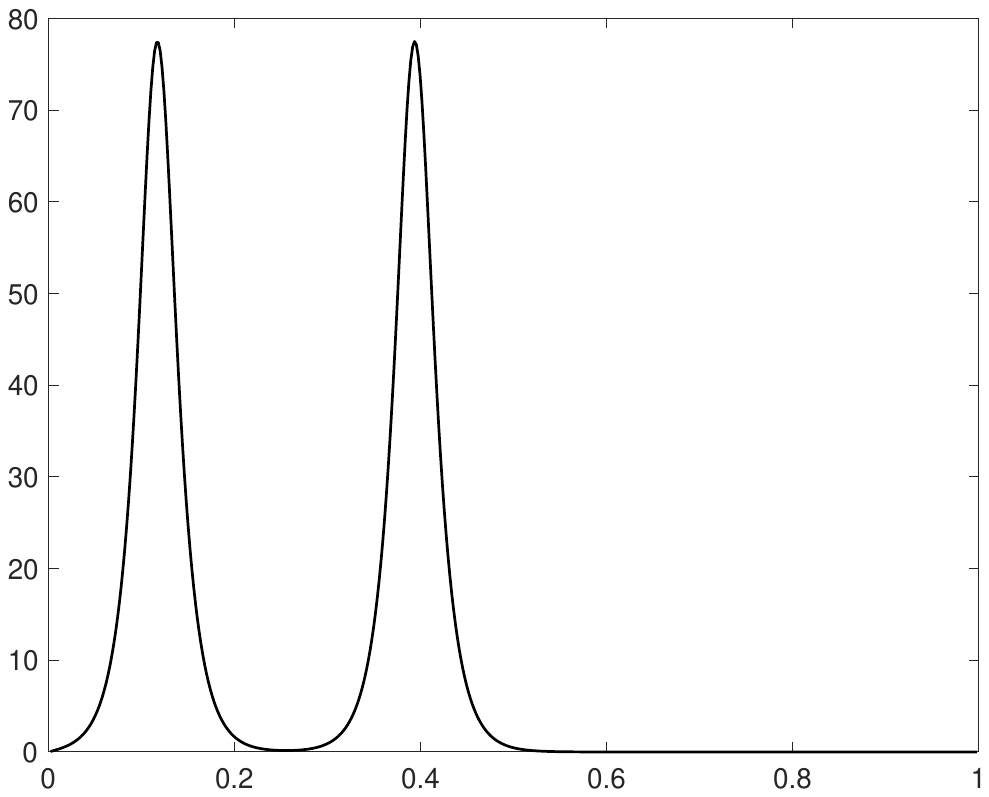}
	\put (11,81) {\tiny$u(x)$}
	\put (93,14) {\tiny$x$}
\end{overpic}
	\begin{overpic}[scale=0.24,trim = 1cm 5cm 1cm 5cm, clip]{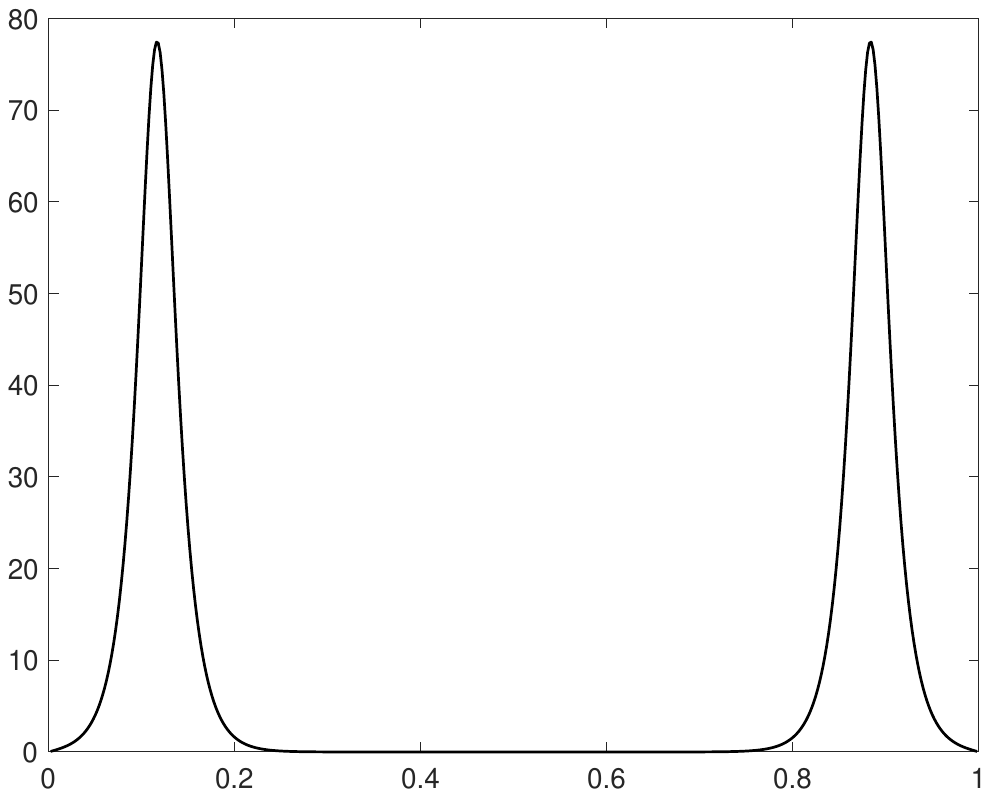}
	\put (11,81) {\tiny$u(x)$}
	\put (93,14) {\tiny$x$}
\end{overpic}
	\begin{overpic}[scale=0.24,trim = 1cm 5cm 1cm 5cm, clip]{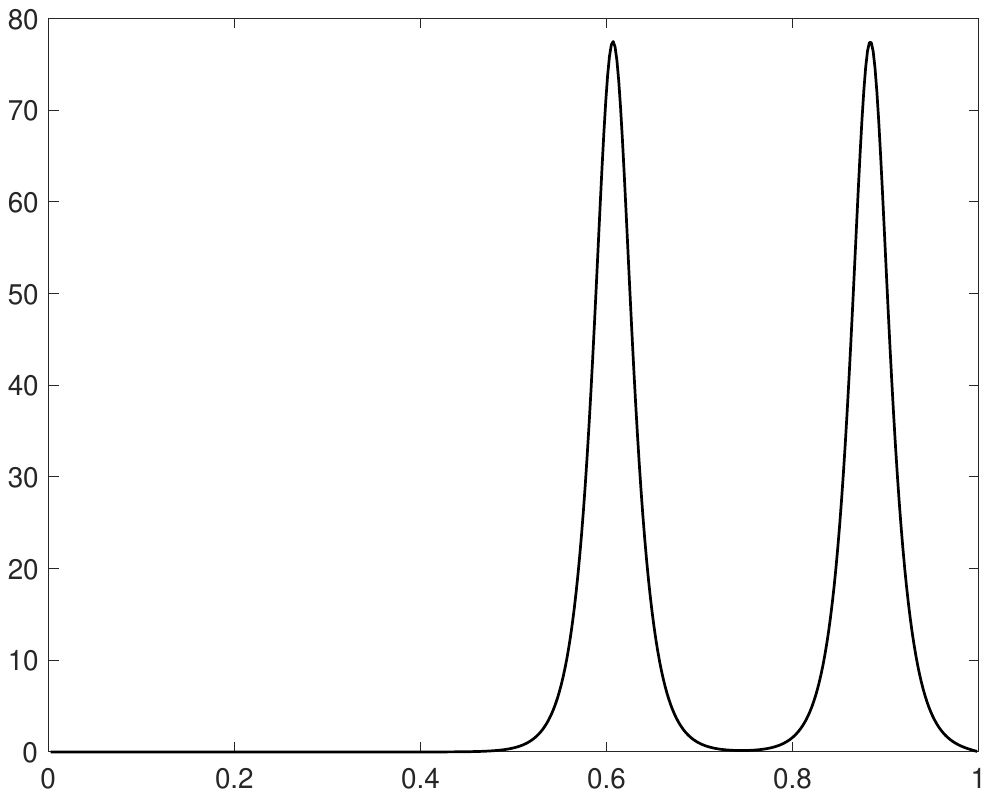}\put (11,81) {\tiny$u(x)$}
	\put (93,14) {\tiny$x$}
\end{overpic}\\[-2.5em]
	\begin{overpic}[scale=0.24,trim = 1cm 5cm 1cm 5cm, clip]{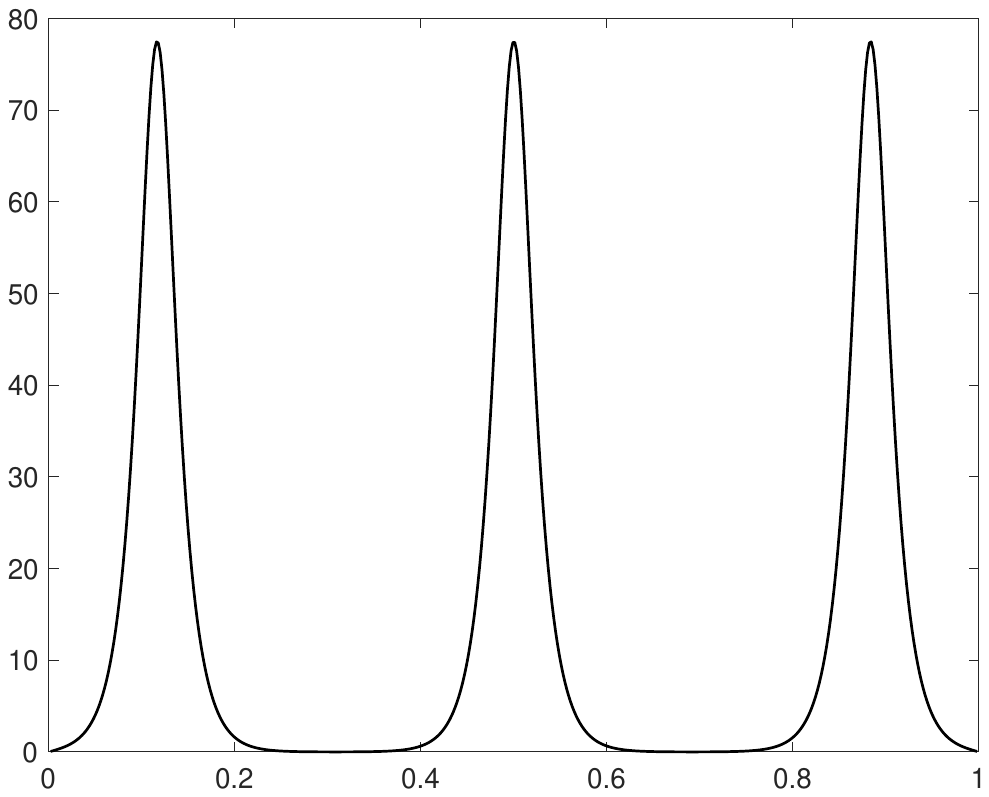}	\put (11,81) {\tiny$u(x)$}
	\put (93,14) {\tiny$x$}
\end{overpic}
	\vspace{-0.4cm}
	\caption{Profiles of the solutions of \eqref{1.1} with $a=a_{2,0}$ for $\l=-3000$. Here $\a_1=0.175$ and $\b_1=0.325$ (thus, $\a_2=0.675$ and $\b_2=0.825$).  These profiles show the  phenomenology described in Theorem \ref{th2.1}.}
	\label{Fig3}
\end{figure}

\section{The case $\kappa =1$ with $h\in (0,1)$}
\label{sec:3}
In this section we compute the global bifurcation diagram of positive solutions of \eqref{1.1} for a series of choices of $a=a_{1,0}$. In these examples, $a^{-1}(0)$ consists of one single interval $J_h=\left(\frac{1-h}{2},\frac{1+h}{2}\right)$. Figure \ref{Fig4} shows one of these weights, corresponding to $h=0.5$.

\par
One of our main goals is to analyze the effects of increasing $h$ on the structure of the set of positive solutions of \eqref{1.1} as well as on their profiles. Thus, we will compute the diagrams for different values of $h$ varying between $h=0.001$ and $h=0.95$.  Remark \ref{re2.2} conjectures the existence of (at least) three positive solutions for sufficiently negative $\l$. This has also been confirmed by the numerical experiments carried out in this section.

\begin{figure}[h!]
	\centering
	\begin{overpic}[scale=0.28,trim = 1cm 6cm 1cm 7cm, clip]{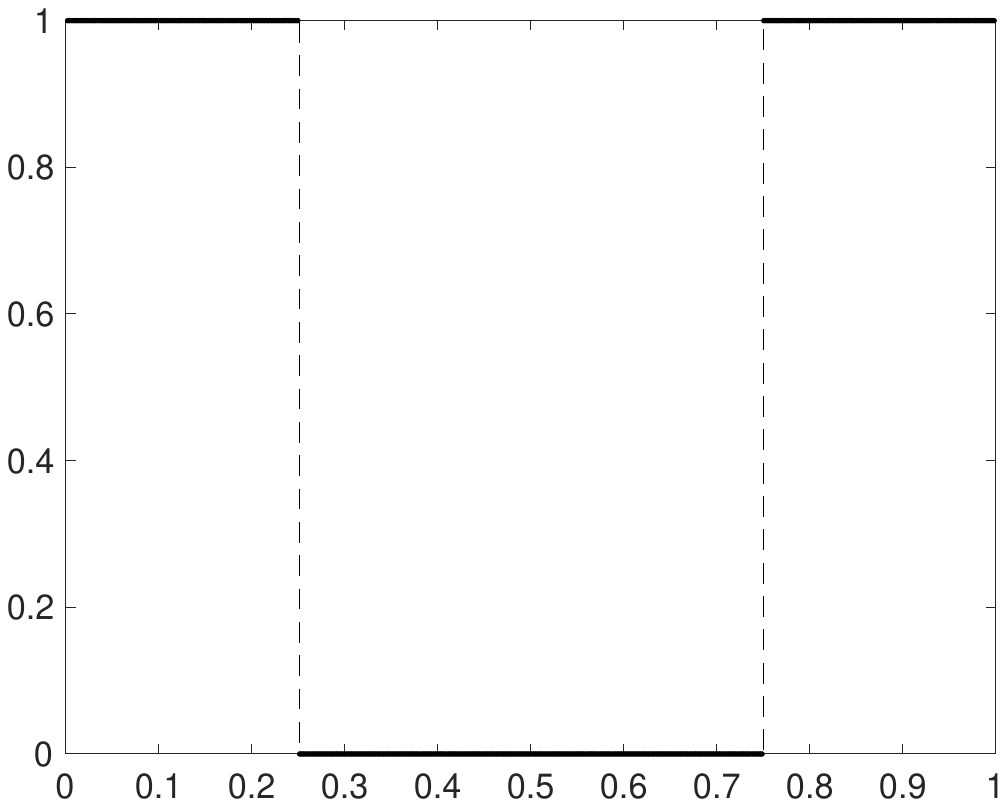} \put (12,78) {\tiny$a(x)$}
	\put (96,10) {\tiny$x$}
\end{overpic}
	\vspace{-0.4cm}
	\caption{An admissible weight function $a(x)$ with $\kappa=1$.}
	\label{Fig4}
\end{figure}

Indeed, using the notation of Section \ref{sec:1}, the component $\mathscr{C}_0^+$ that bifurcates from $(\pi^2,0)$  exhibits a secondary subcritical pitchfork bifurcation at some value of the parameter, $\l_b$, for all computed values of $h$, so that \eqref{1.1} has three positive solutions for each $\l\in (-\infty,\l_b)$. Observe that, throughout this work, bifurcation points in the bifurcation diagrams have been marked with a thick dot as, e.g., in Figure \ref{Fig5}. Table \ref{Tab1} collects the values of the computed  secondary bifurcation points $\l_b$  for a series of values of the parameter $h$. According to our numerical experiments,  $\l_b=\l_b(h)$ satisfies
\begin{equation}
\label{iii.1}
	\lim_{h \da 0} \l_b(h)=-\infty,\qquad \lim_{h\ua 1}\l_b(h)=\pi^2.
\end{equation}

\begin{table}[ht]
	\caption{Computed values of $\l_{b}$ for a series of $h$.}
	\begin{center}  \footnotesize
		\begin{tabular}{|ccccccccccc|}
			\hline
			$h$ & & $\l_b$ & \vline \qquad \qquad \vline & $h$ & & $\l_b$ & \vline \qquad \qquad \vline & $h$ & & $\l_b$\\\hline
			$0.001$  & & $-57.65679$ & \vline \qquad \qquad \vline & $0.007$  & & $-33.03349$ & \vline \qquad \qquad \vline & $0.30$  & & $+2.03964$ \\\hline
			$0.002$  & & $-48.29521$ & \vline \qquad \qquad \vline & $0.008$  & & $-31.52271$ & \vline \qquad \qquad \vline & $0.50$  & & $+5.34880$ \\\hline
			$0.003$  & & $-43.13164$ & \vline \qquad \qquad \vline & $0.009$  & & $-30.20749$ & \vline \qquad \qquad \vline & $0.70$  & & $+7.37795$ \\\hline
			$0.004$  & & $-39.60097$ & \vline \qquad \qquad \vline & $0.01$  & & $-27.07733$ & \vline \qquad \qquad \vline & $0.80$  & & $+8.21472$ \\\hline
			$0.005$  & & $-36.93564$ & \vline \qquad \qquad \vline & $0.05$  & & $-12.40637$ & \vline \qquad \qquad \vline & $0.89$  & & $+8.95476$ \\\hline
			$0.006$  & & $-34.80407$ & \vline \qquad \qquad \vline & $0.10$  & & $-6.55902$ & \vline \qquad \qquad \vline & $0.95$  & & $+9.44545$ \\\hline
		\end{tabular}
		\label{Tab1}
	\end{center}
\end{table}

Although we do not have an analytical proof of these features yet, the first relation of \eqref{iii.1} is supported by the fact that $a_{1,0}$  approximates $1$  as $h\da 0$, and that \eqref{1.3} has a unique positive solution for every $\l\in (-\infty,\pi^2)$. Thus, we expect the subcritical pitchfork secondary bifurcation at $\l_b$ to approach $-\infty$ as $h\da 0$. On the other hand, problem \eqref{1.1} converges towards a linear problem as $h\ua 1$, as  in such a case $a_{1,0}$ approximates the constant $0$,  which is compatible with the fact that $\lim_{h\ua 1}\l_b(h)=\pi^2$.
\par
Figure \ref{Fig5} shows the computed bifurcation diagram of positive solutions of \eqref{1.1} for $h=0.05$. In this case, by Table \ref{Tab1}, $\l_b=-12.40637$. According to our numerical experiments, the problem admits at least a solution for every $\l<\pi^2$ and at least three solutions if  $\l<\l_b$. It is important to observe that the discrete $L^2$-norm of $u$, which has been used for the vertical axes in all the bifurcation diagrams of this work, make the two secondary branches plotted in red to overlap, because the corresponding solutions are one the reflection of the other around $0.5$. Whenever this occurs, we use, here and in all the subsequent bifurcation diagrams, a thicker line.

\begin{figure}[h!]
	\centering
	\begin{overpic}[scale=0.28,trim = 1cm 5cm 1cm 7cm, clip]{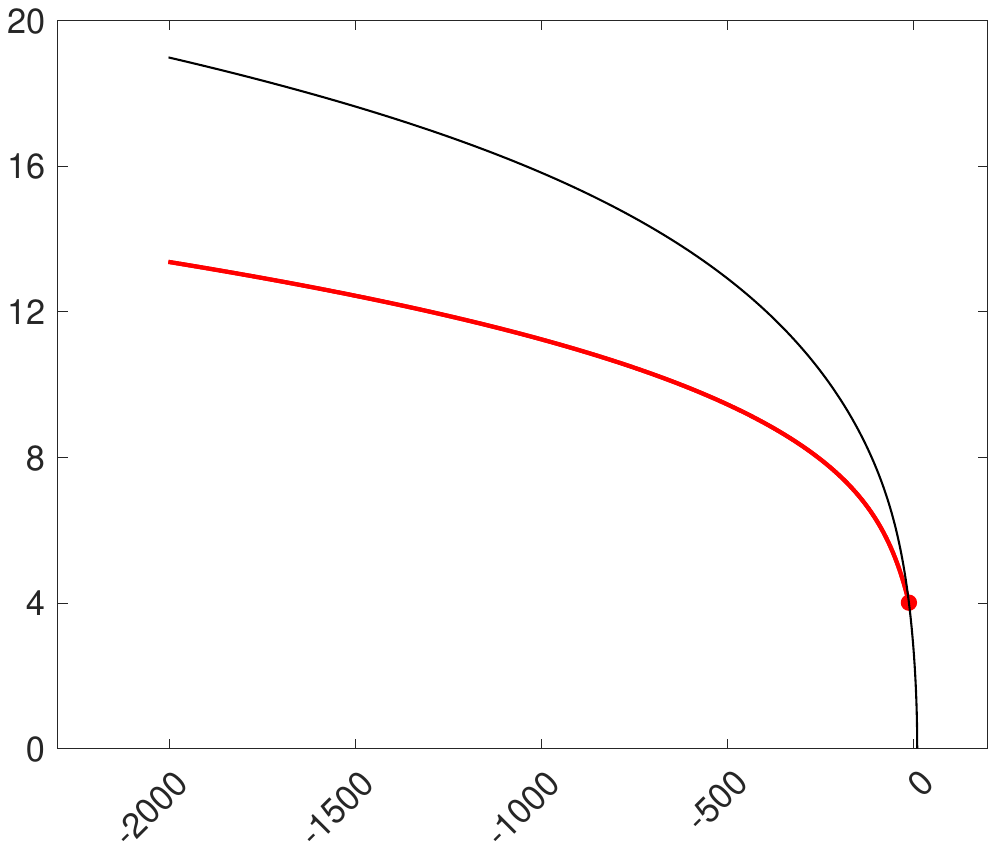} \put (12,84) {\tiny$\|u\|_2$}
	\put (96,15) {\tiny$\l$}
\end{overpic}
	\begin{overpic}[scale=0.28,trim = 1cm 5cm 1cm 7cm, clip]{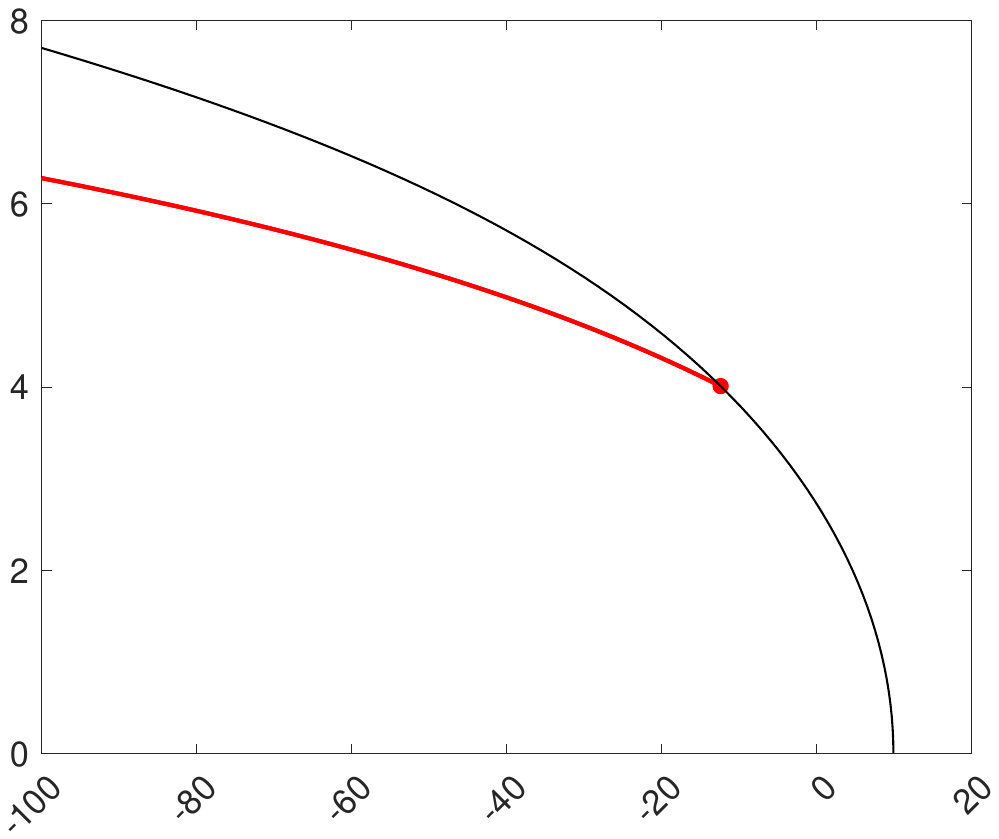} \put (12,84) {\tiny$\|u\|_2$}
	\put (96,15.5) {\tiny$\l$}
\end{overpic}
	\vspace{-0.4cm}
	\caption{Bifurcation diagram relative to \eqref{1.1} for $a(x)=a_{1,0}(x)$ with $h=0.05$ (left) and a zoom of it (right).}
	\label{Fig5}
\end{figure}

The first row of Figure \ref{Fig6} shows in black a series of positive solutions along the branch emanating from $(\pi^2,0)$, all symmetric about $0.5$. The left plot corresponds to profiles for positive values of $\lambda$ where all the solutions are concave and the right plot corresponds to profiles for negative values of $\lambda$, where solutions are not necessarily concave. It should be noted that the last plotted profile on the left coincides with the first profile plotted on the right. In the second row, a series of positive asymmetric solutions lying on each of the secondary branches bifurcating at $\l_b<0$ are plotted in red.

\begin{figure}[h!]
	\centering
	\begin{overpic}[scale=0.28,trim = 1cm 5cm 1cm 7cm, clip]{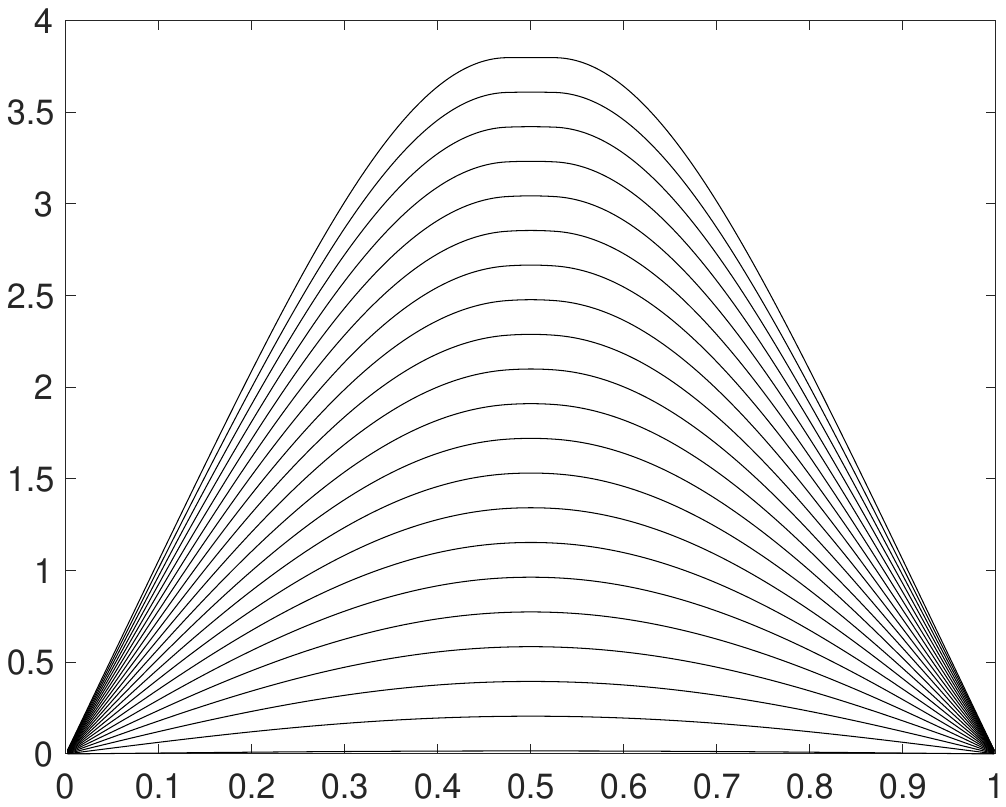} \put (12,82.5) {\tiny$u(x)$}
	\put (96,15) {\tiny$x$}
\end{overpic} \begin{overpic}[scale=0.28,trim = 1cm 5cm 1cm 7cm, clip]{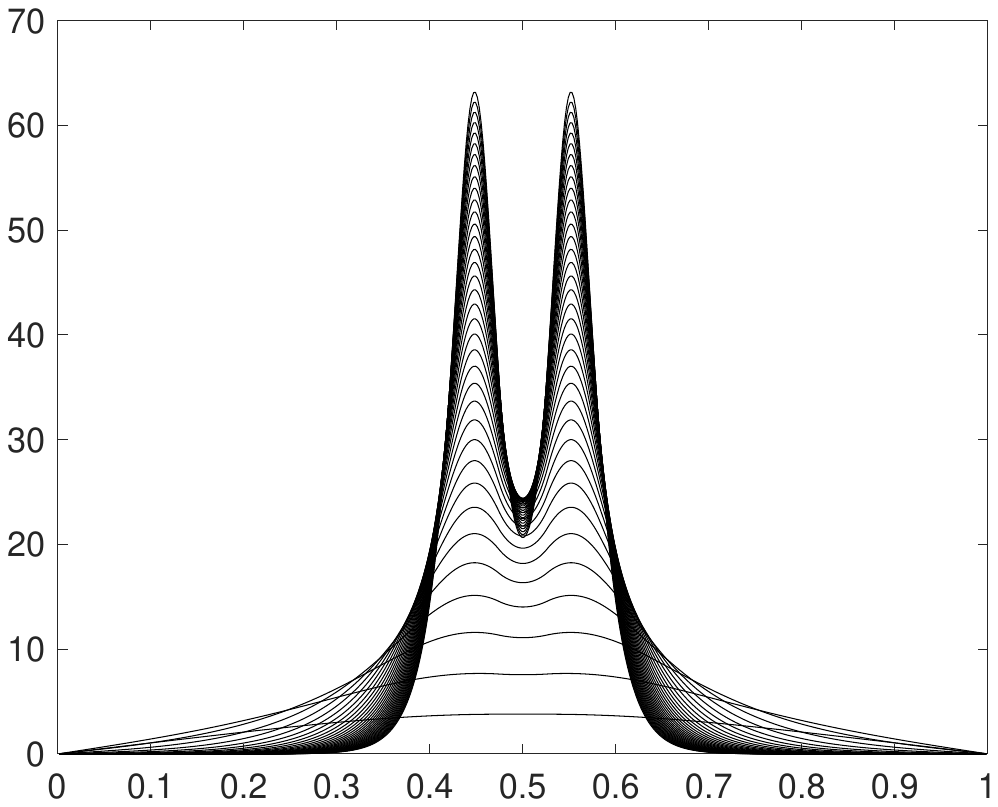} \put (12,82.5) {\tiny$u(x)$}
\put (96,15) {\tiny$x$}
\end{overpic} \\ [-2.5em]
	\begin{overpic}[scale=0.28,trim = 1cm 5cm 1cm 6cm, clip]{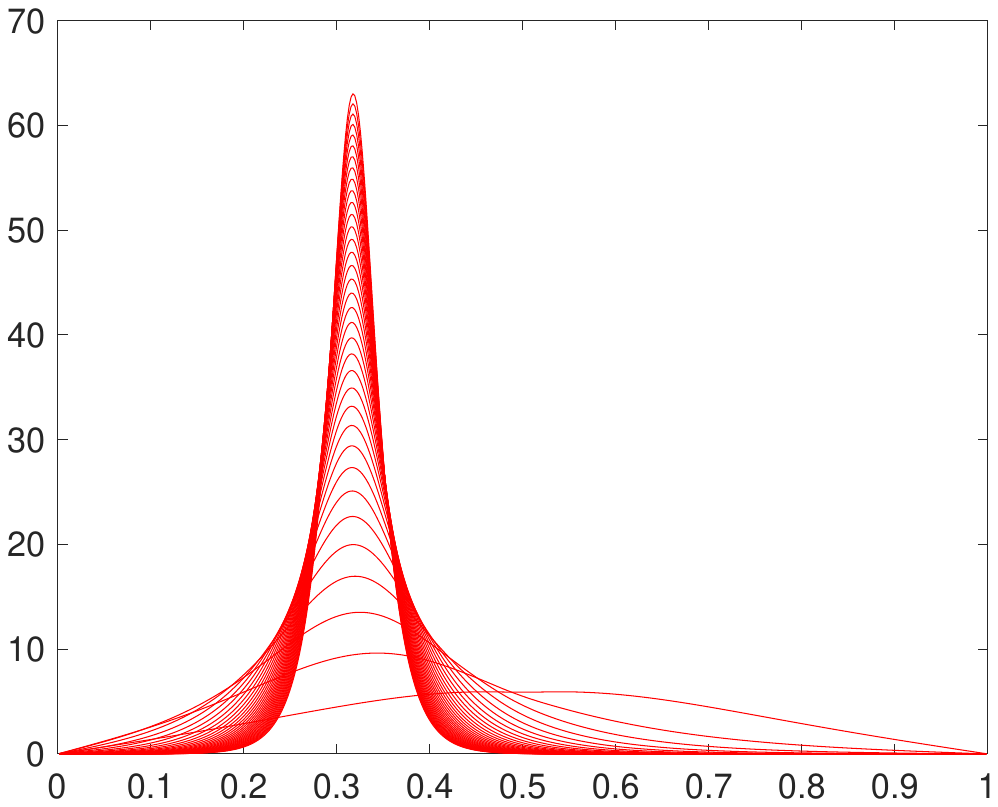} \put (12,82.5) {\tiny$u(x)$}
	\put (96,15) {\tiny$x$}\end{overpic} \begin{overpic}[scale=0.28,trim = 1cm 5cm 1cm 6cm, clip]{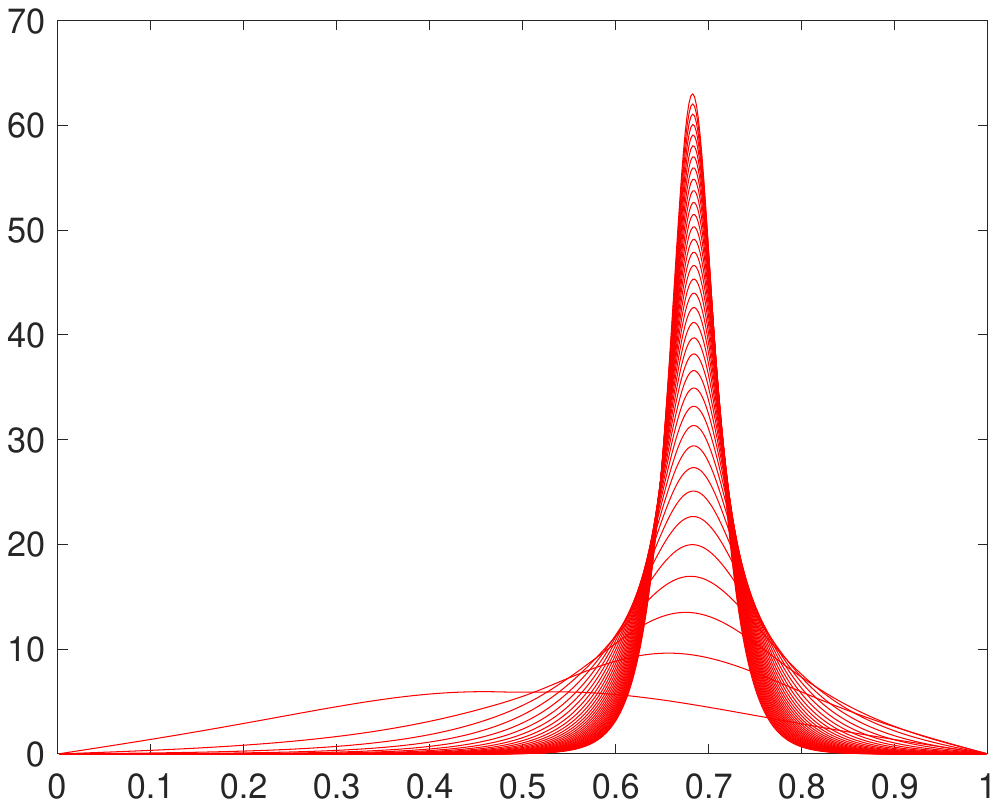} \put (12,82.5) {\tiny$u(x)$}
\put (96,15) {\tiny$x$}\end{overpic}
	\vspace{-0.4cm}
	\caption{A series of plots of positive solutions along the three branches corresponding to the bifurcation diagram of Figure \ref{Fig5}: symmetric branch (first row) for $\l>0$ (left) and $\l<0$ (right), and asymmetric branches for $\l<\l_b$ (second row).}
	\label{Fig6}
\end{figure}

\par
As Table \ref{Tab1} illustrates, our numerical experiments suggest the existence of $h_0 \in (0.1,0.3)$ such that $\lambda_b<0$ if $h\in(0,h_0)$ and $\lambda_b>0$ if $h\in(h_0,1)$. Thus, in Figures \ref{Fig8}, \ref{Fig10} and \ref{Fig12} below, corresponding to $h=0.8$, $h=0.89$ and $h=0.95$ respectively, the secondary bifurcation along the main symmetric branch occurs at a positive value $\l_b$; for this reason we have decided to plot the solutions' profiles along each of the secondary half-branches for $\l>0$ and $\l<0$ separately. We point out that the solutions have the same concavity patterns observed for symmetric solutions along the main bifurcated branch in the case $h=0.05$ (see Figure \ref{Fig6}). Observe that, in all the rows, the last plotted profile of the first column coincides with the first profile plotted in the second column. This allows us to show how the solutions emerge, either from $(\pi^2,0)$ or on the secondary bifurcation point at $\l=\l_b$, and then evolve on the corresponding branch.

\par We analyze now more in depth the solutions' profile for $\l<0$. Fix such a $\l$ and suppose that $x_m\in (0,1)$ is a local maximum of a positive solution $u$ of \eqref{1.1}. Then, $u'(x_m)=0$ and $u''(x_m)\leq 0$. Thus,
\begin{equation*}
0\leq -u''(x_m)=[\lambda + a(x_m) u^2(x_m)]u(x_m)
\end{equation*}
and hence,
\begin{equation*}
a(x_m) u^2(x_m)\geq -\l >0.
\end{equation*}
Therefore, $a(x_m)>0$, which entails $a(x_m)=1$. Consequently, for $\l<0$, the local maxima of the positive solutions can only be attained in the region where $a=1$, i.e., either in $[0,\frac{1-h}{2}]$ or
in  $[\frac{1+h}{2},1]$. As a byproduct, the symmetric solution on $\mathscr{C}_0^+$ must have, at least, two peaks: one on each of the intervals $[0,\frac{1-h}{2}]$ and $[\frac{1+h}{2},1]$, as it is apparent from the top right plots of Figure \ref{Fig6}. Moreover, our numerical experiments reveal that each of the asymmetric solutions has exactly one peak in each of these intervals.

\par
The main features of the bifurcation diagrams plotted in Figure \ref{Fig5} for $h=0.05$ are preserved until
$h$ reaches some intermediate value $h^*\in (0.4,0.5)$. Then, for every $h\in (h^*,1)$, a new phenomenon arises. Namely, the norm of the solutions along $\mathscr{C}_0^+$ increases very quickly in a short interval of $\l$'s as $\l$ separates from $\pi^2$, until it reaches a value denoted by $\l^*$ that lies slightly before $\l_b$. Then, for $\l<\l^*$, the norm decreases again very quickly, as illustrated in the bifurcation diagram of Figure \ref{Fig7}, corresponding to $h=0.8$.

\begin{figure}[h!]
	\centering
	\begin{overpic}[scale=0.28,trim = 1cm 5cm 1cm 7cm, clip]{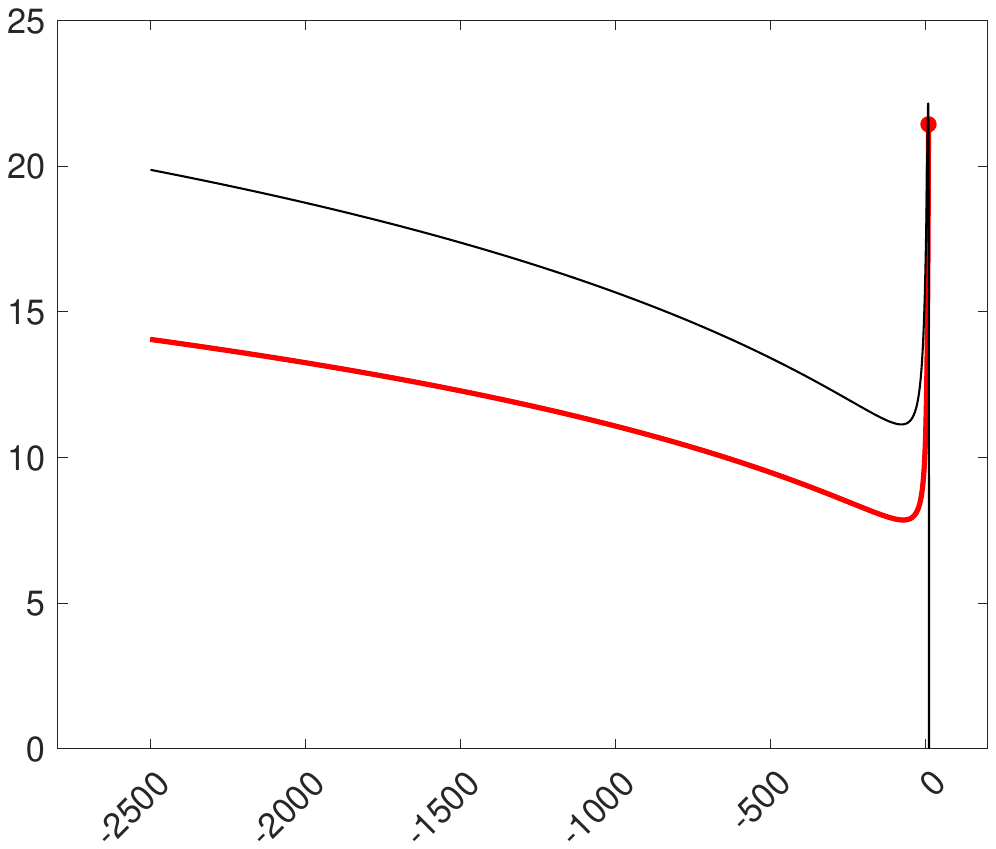} \put (12,84) {\tiny$\|u\|_2$}
		\put (97,15) {\tiny$\l$}
	\end{overpic}
	\begin{overpic}[scale=0.28,trim = 1cm 5cm 1cm 7cm, clip]{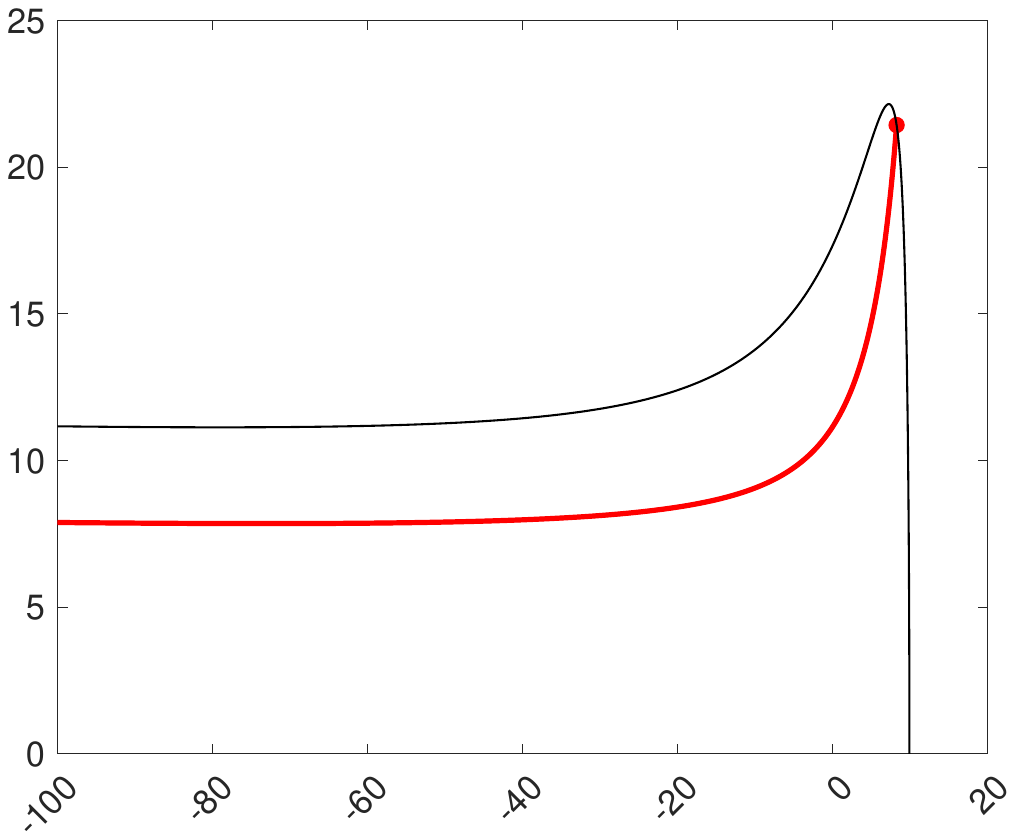} \put (12,84) {\tiny$\|u\|_2$}
		\put (97,15) {\tiny$\l$}
	\end{overpic}
	\vspace{-0.4cm}
	\caption{Bifurcation diagram relative to \eqref{1.1} for $a(x)=a_{1,0}(x)$ with $h=0.8$ (left) and a zoom of it (right).}
	\label{Fig7}
\end{figure}

\begin{figure}[h!]
	\centering
	\begin{overpic}[scale=0.28,trim = 1cm 5cm 1cm 9cm, clip]{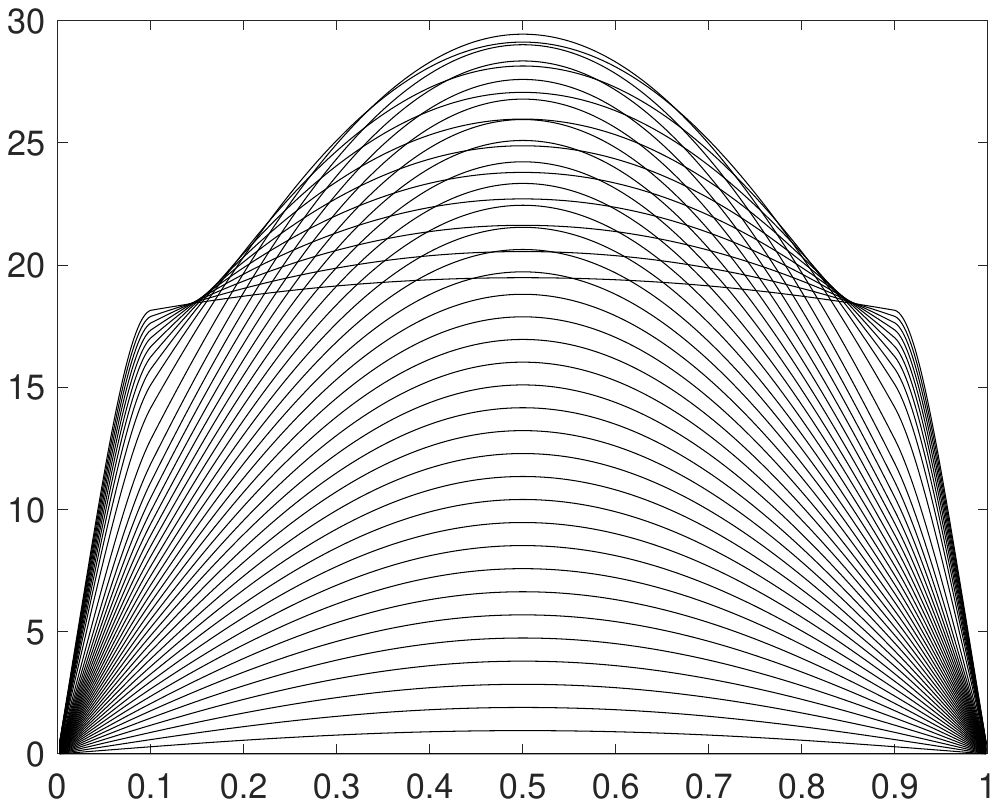} \put (12,83) {\tiny$u(x)$}
		\put (96,15) {\tiny$x$}\end{overpic} \begin{overpic}[scale=0.28,trim = 1cm 5cm 1cm 9cm, clip]{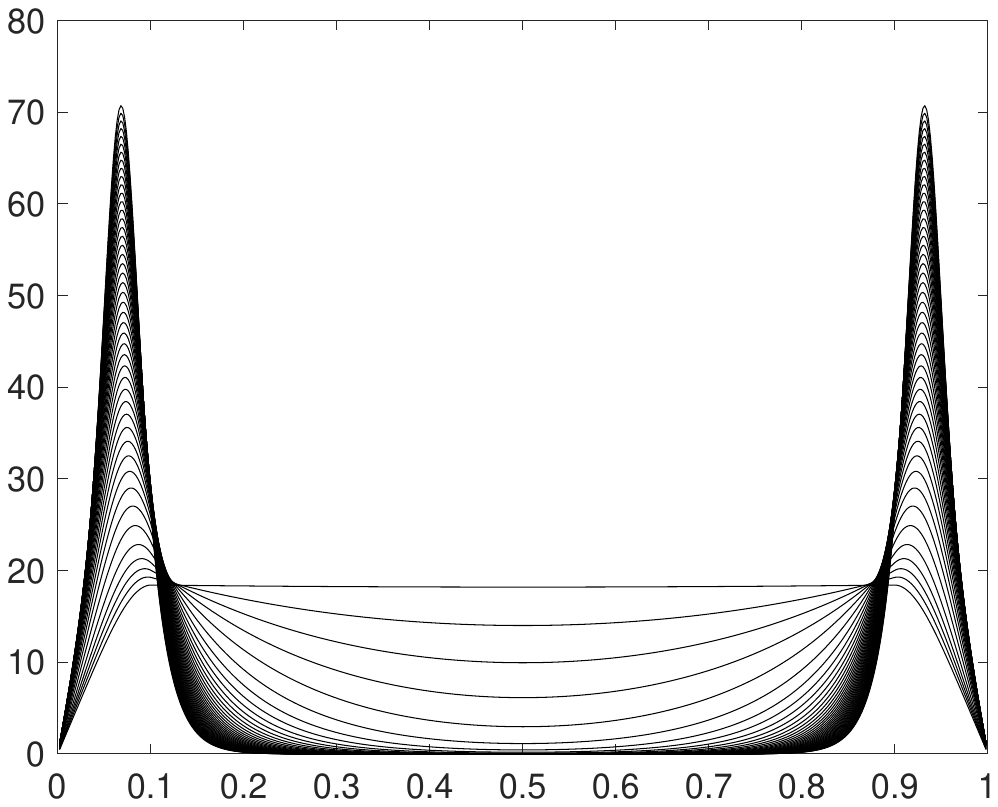} \put (12,83) {\tiny$u(x)$}
		\put (96,15) {\tiny$x$}\end{overpic} \\ [-2.5em]
	\begin{overpic}[scale=0.28,trim = 1cm 5cm 1cm 6cm, clip]{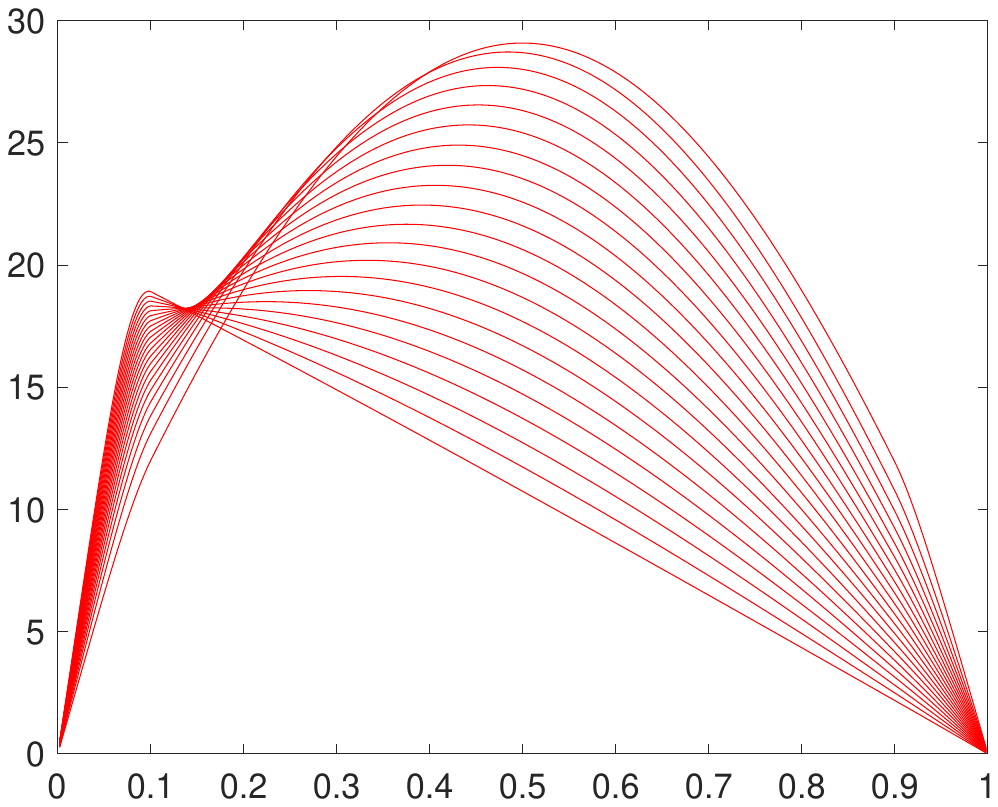} \put (12,83) {\tiny$u(x)$}
		\put (96,15) {\tiny$x$}\end{overpic} \begin{overpic}[scale=0.28,trim = 1cm 5cm 1cm 6cm, clip]{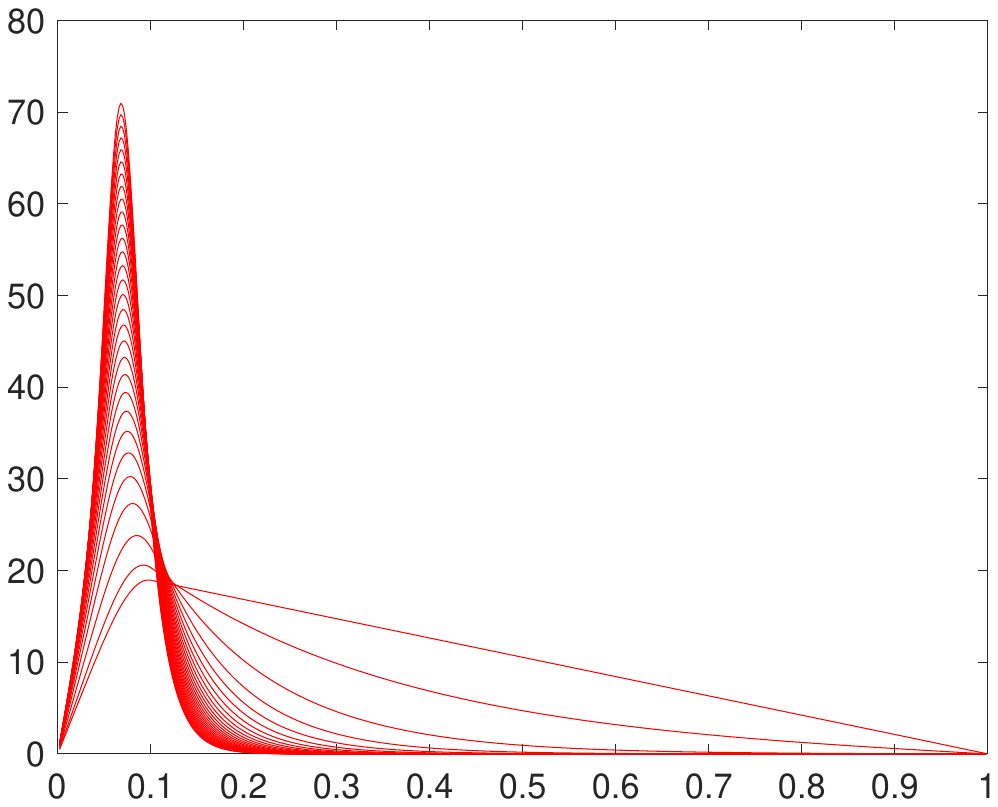} \put (12,83) {\tiny$u(x)$}
		\put (96,15) {\tiny$x$}\end{overpic}\\  [-2.5em]
	\begin{overpic}[scale=0.28,trim = 1cm 5cm 1cm 6cm, clip]{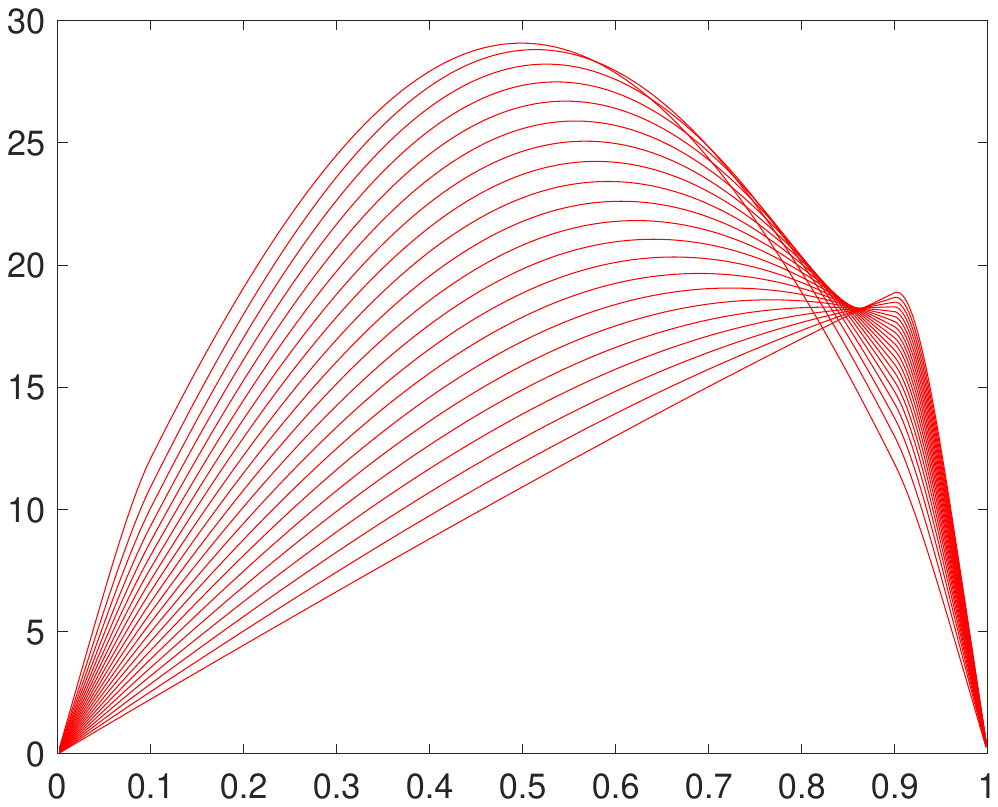} \put (12,83) {\tiny$u(x)$}
		\put (96,15) {\tiny$x$}\end{overpic} \begin{overpic}[scale=0.28,trim = 1cm 5cm 1cm 6cm, clip]{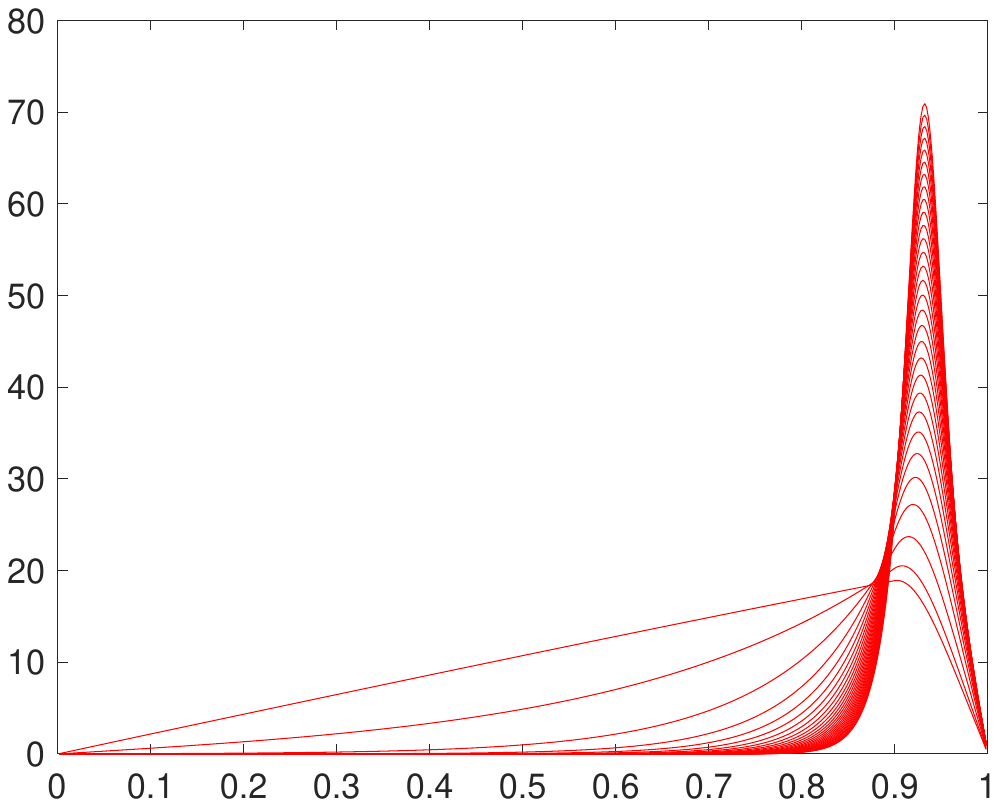} \put (12,83) {\tiny$u(x)$}
		\put (96,15) {\tiny$x$}\end{overpic}
	\vspace{-0.4cm}
	\caption{A series of plots of positive solutions along the three branches corresponding to the bifurcation diagram of Figure \ref{Fig7}: symmetric branch (first row) for $\l>0$ (left) and $\l<0$ (right), and asymmetric branches for $\l<\l_b\sim 8.21472$ (second and third row).}
	\label{Fig8}
\end{figure}

Since, according to Table \ref{Tab1}, $\l_b=8.21472$ for $h=0.8$, the secondary bifurcation value is relatively close to the bifurcation point
of $\mathscr{C}_0^+$ from $(\lambda,0)$, which arises  at $\l=\pi^2\sim 9.86965$. Thus, for $\l\sim\pi^2$, both symmetric and asymmetric solutions of the component $\mathscr{C}_0^+$ are concave and are close to a multiple of $\sin (\pi x)$, the principal eigenfunction of $-D^2$ in $[0,1]$ under Dirichlet boundary conditions. When $\l<0$, instead, one and two peaks arise in the region where $a>0$ for asymmetric solutions and for symmetric solutions, respectively, as shown above.
\par
The closer  $h$ is to $1$, the more emphasized is the rise and subsequent fall of the branches of the bifurcation diagram for $\l$ close to $\pi^2$, which is well illustrated in Figure \ref{Fig9} for the case $h=0.89$. The corresponding solutions' profile are shown in Figure \ref{Fig10}, where one can observe similar patterns as those described above. In this case, according to Table \ref{Tab1}, $\l_b\sim 8.95476$.

\begin{figure}[h!]
	\centering
	\begin{overpic}[scale=0.28,trim = 1cm 5cm 1cm 7cm, clip]{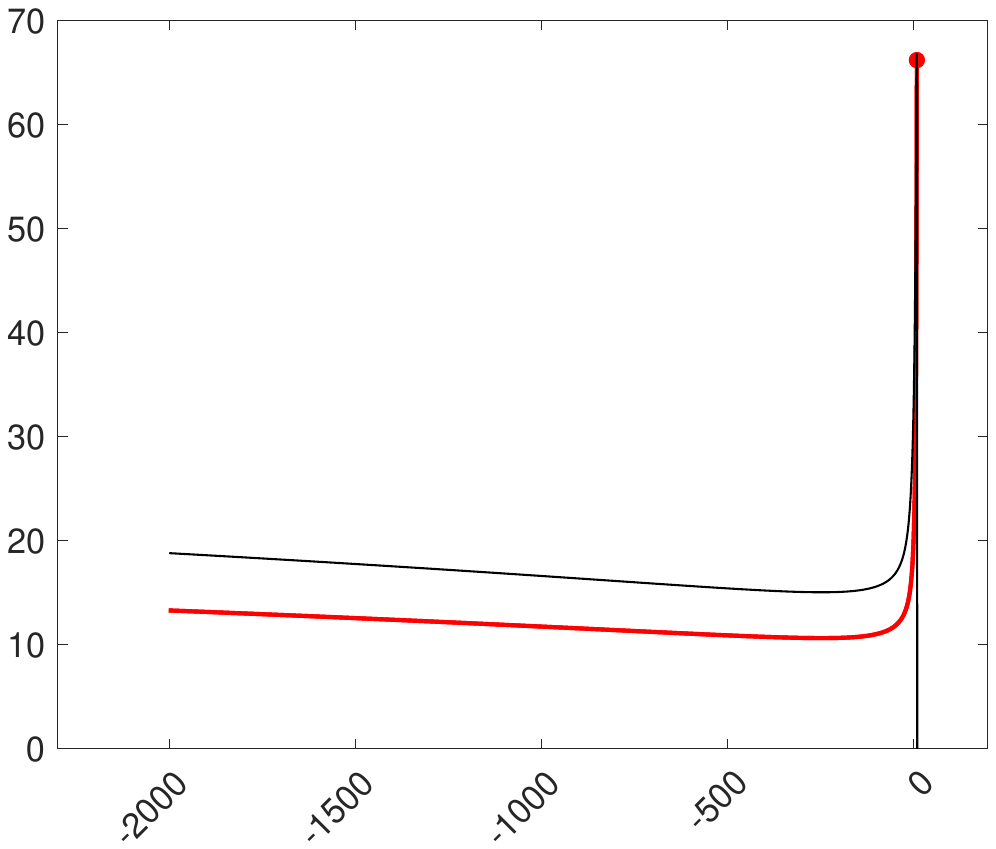} \put (12,84) {\tiny$\|u\|_2$}
	\put (97,15) {\tiny$\l$}
\end{overpic}
	\begin{overpic}[scale=0.28,trim = 1cm 5cm 1cm 7cm, clip]{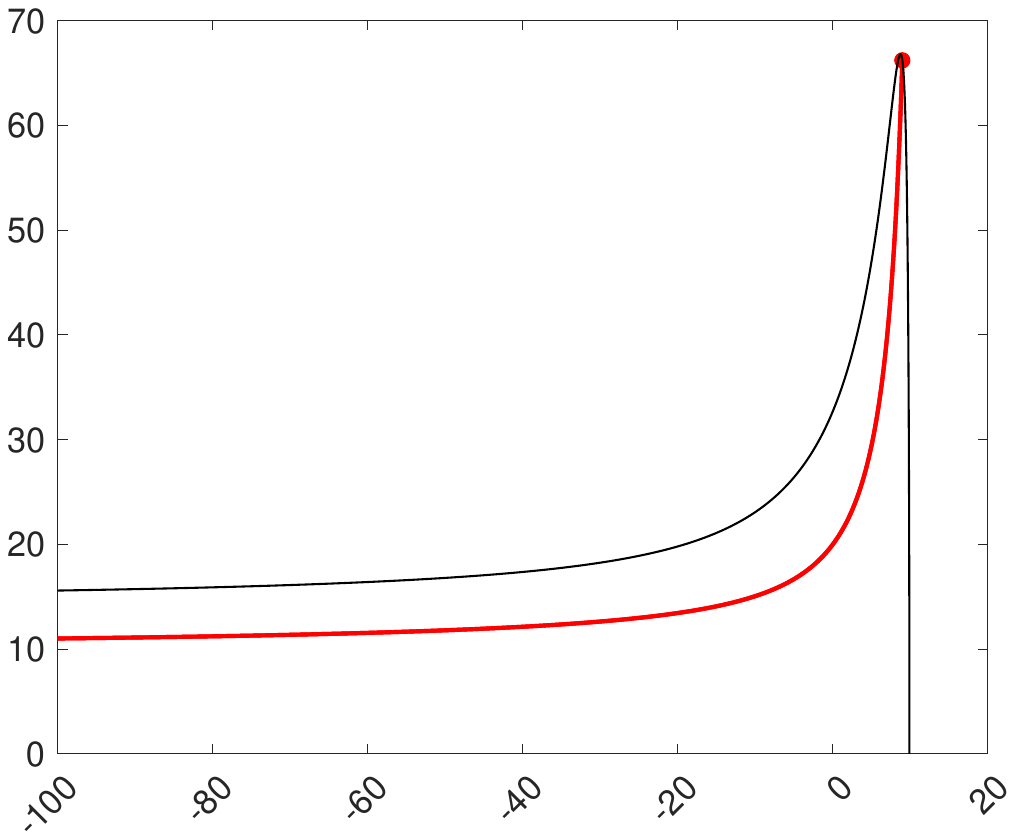} \put (12,84) {\tiny$\|u\|_2$}
	\put (97,15) {\tiny$\l$}
\end{overpic}
	\vspace{-0.4cm}
	\caption{Bifurcation diagram relative to \eqref{1.1} for $a(x)=a_{1,0}(x)$ with $h=0.89$ (left) and a zoom of it (right).}
	\label{Fig9}
\end{figure}

\begin{figure}[h!]
	\centering
	\begin{overpic}[scale=0.28,trim = 1cm 5cm 1cm 9cm, clip]{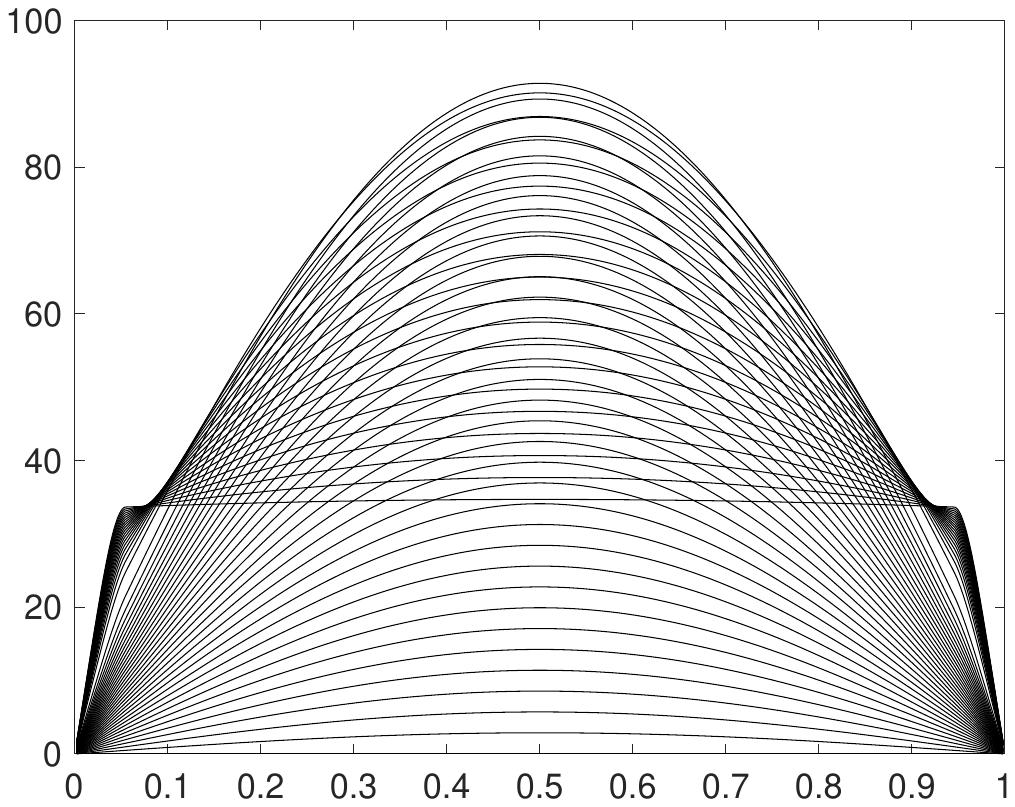}  \put (12,83) {\tiny$u(x)$}
	\put (96.5,15) {\tiny$x$}\end{overpic} \begin{overpic}[scale=0.28,trim = 1cm 5cm 1cm 9cm, clip]{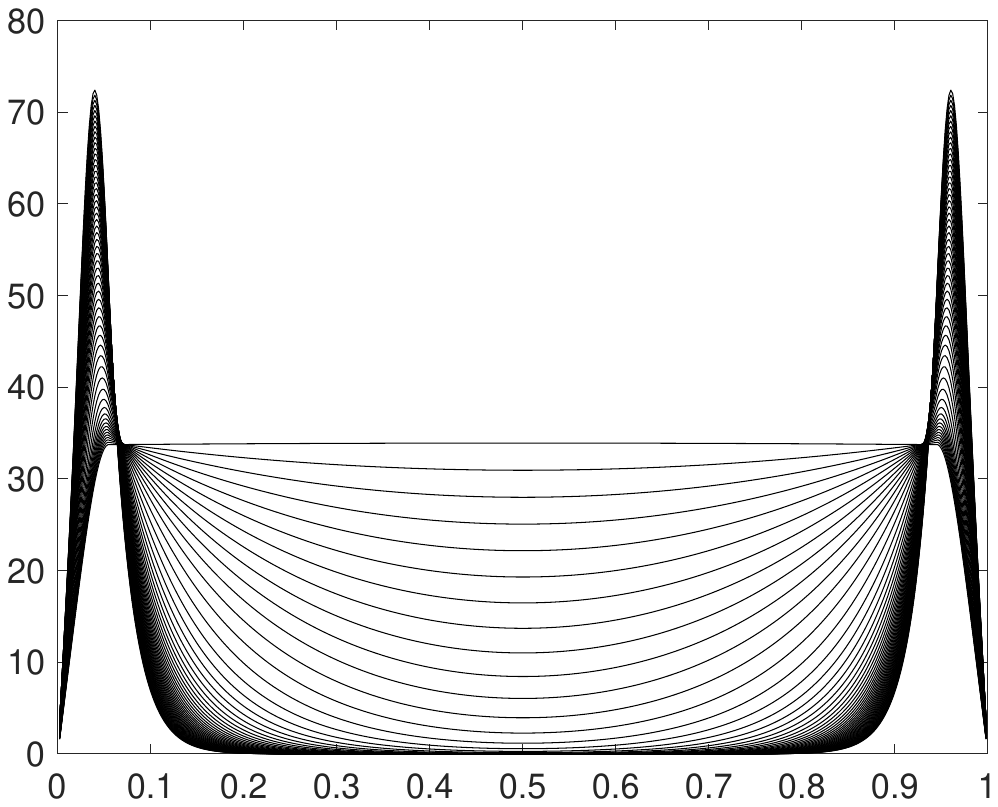}  \put (12,83) {\tiny$u(x)$}
\put (96.5,15) {\tiny$x$}\end{overpic}\\ [-2.5em]
	\begin{overpic}[scale=0.28,trim = 1cm 5cm 1cm 5cm, clip]{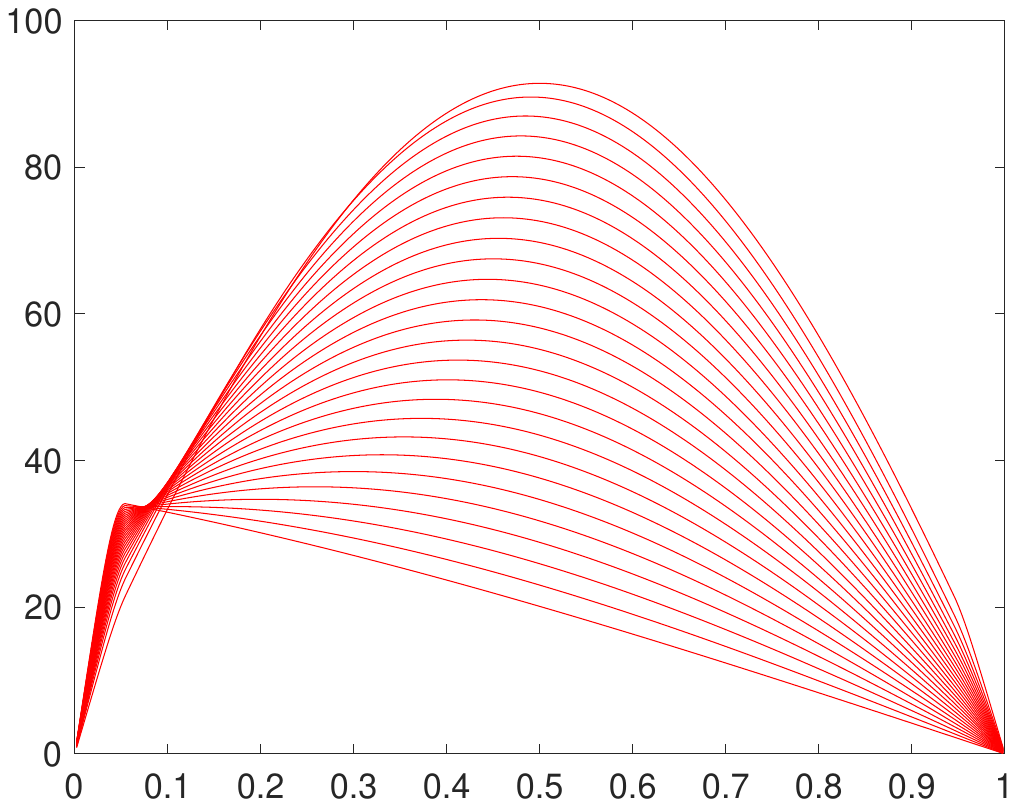}  \put (12,80) {\tiny$u(x)$}
	\put (93,15) {\tiny$x$}\end{overpic} \begin{overpic}[scale=0.28,trim = 1cm 5cm 1cm 5cm, clip]{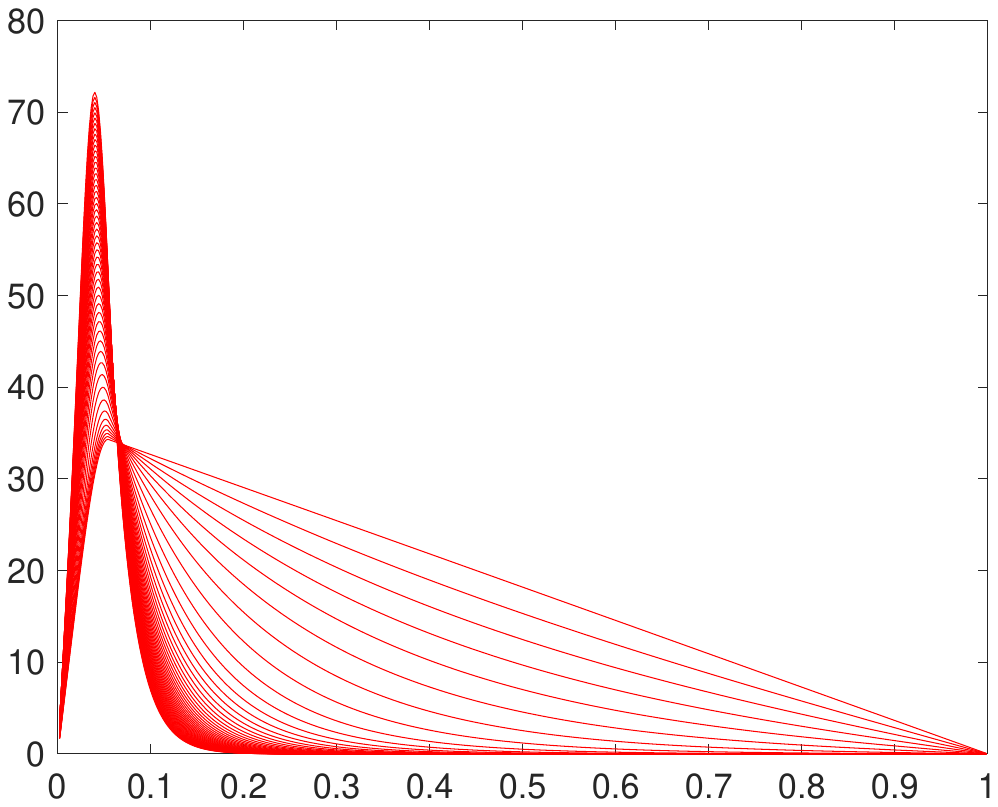}  \put (12,80) {\tiny$u(x)$}
\put (93,15) {\tiny$x$}\end{overpic}\\[-2.5em]
	\begin{overpic}[scale=0.28,trim = 1cm 5cm 1cm 5cm, clip]{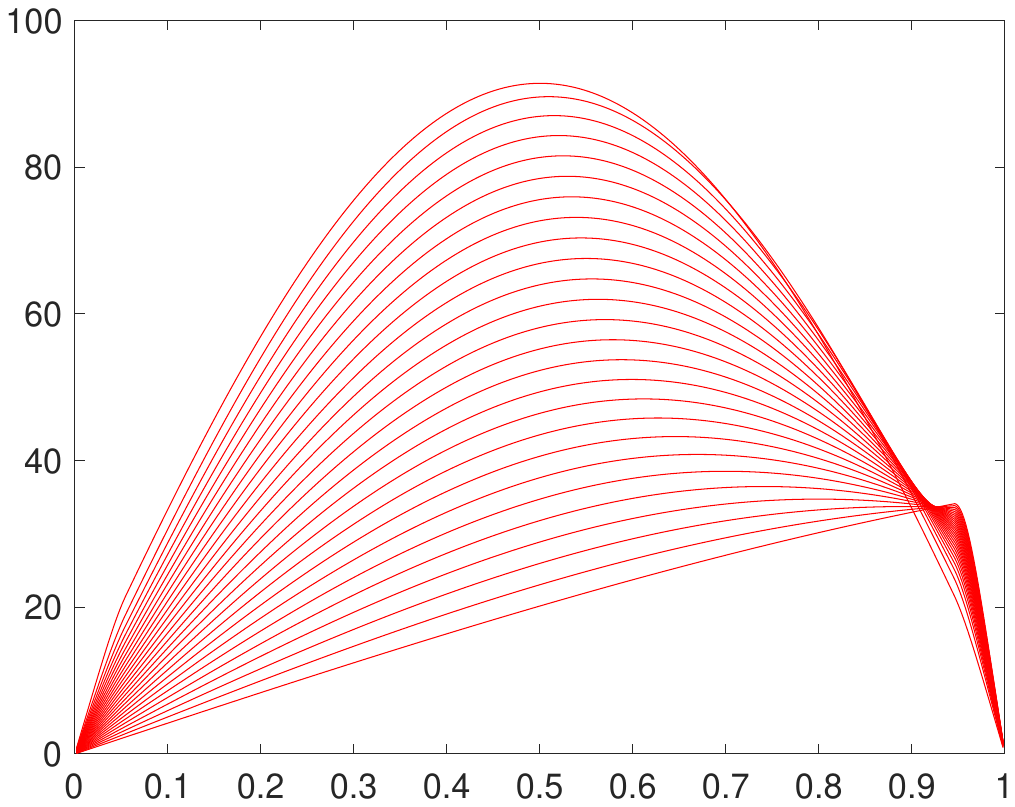}  \put (12,80) {\tiny$u(x)$}
		\put (93,15) {\tiny$x$}\end{overpic} \begin{overpic}[scale=0.28,trim = 1cm 5cm 1cm 5cm, clip]{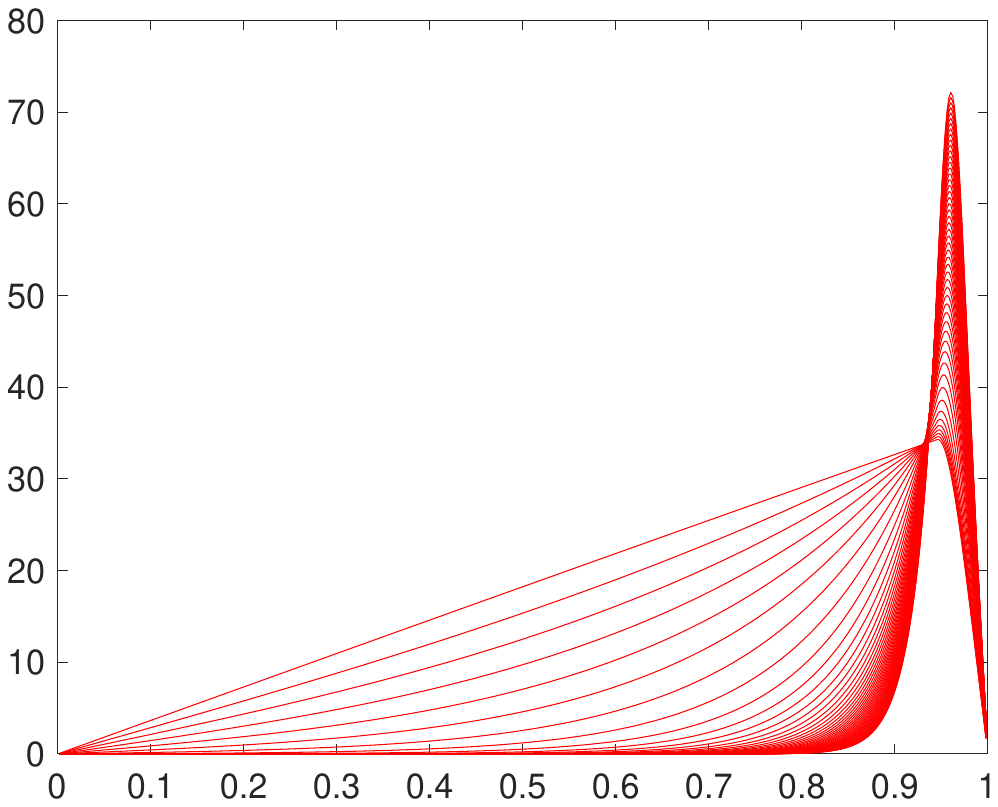}  \put (12,80) {\tiny$u(x)$}
		\put (93,15) {\tiny$x$}\end{overpic}
	\vspace{-0.4cm}
	\caption{A series of plots of positive solutions along the three branches corresponding to the bifurcation diagram of Figure \ref{Fig9}: symmetric branch (first row) for $\l>0$ (left) and $\l<0$ (right), and asymmetric branches for $\l<\l_b\sim 8.95476$ (second and third row).}
	\label{Fig10}
\end{figure}

Finally, Figures \ref{Fig11} and \ref{Fig12} show the global bifurcation diagram corresponding to $h=0.95$ and a series of plots of positive solutions on each of the three branches respectively. In this case, since $\l_b\sim 9.44545$, $\pi^2\sim 9.86965$, and the intervals where $a=1$ shorten to $[0,0.025]$ and $[0.975,1]$ while $a=0$ in $(0.025,0.975)$, the global bifurcation diagram is extremely steep in a neighborhood of $\pi^2$, adding a number of technical difficulties to its computation (see Section \ref{sec:7}). Thus, we have stopped our numerical experiments here.

\begin{figure}[h!]
	\centering
	\begin{overpic}[scale=0.28,trim = 1cm 5cm 1cm 9cm, clip]{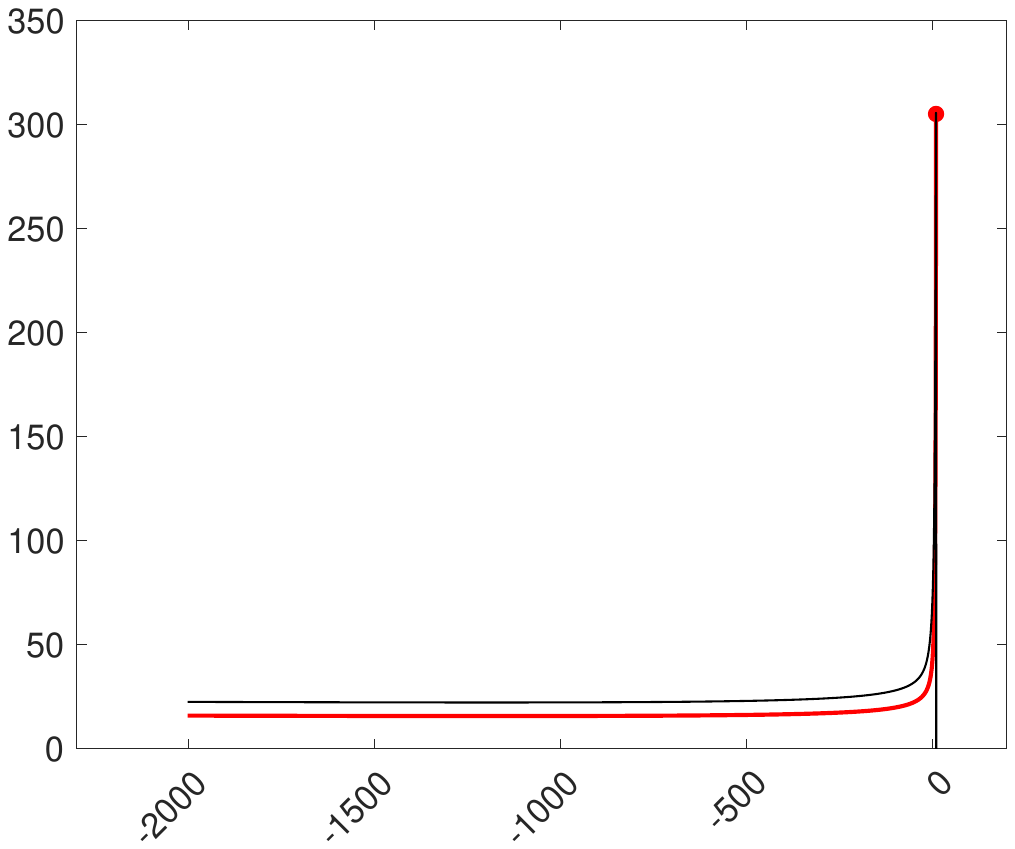} \put (12,84) {\tiny$\|u\|_2$}
	\put (97,15) {\tiny$\l$}
\end{overpic}
	\begin{overpic}[scale=0.28,trim = 1cm 5cm 1cm 9cm, clip]{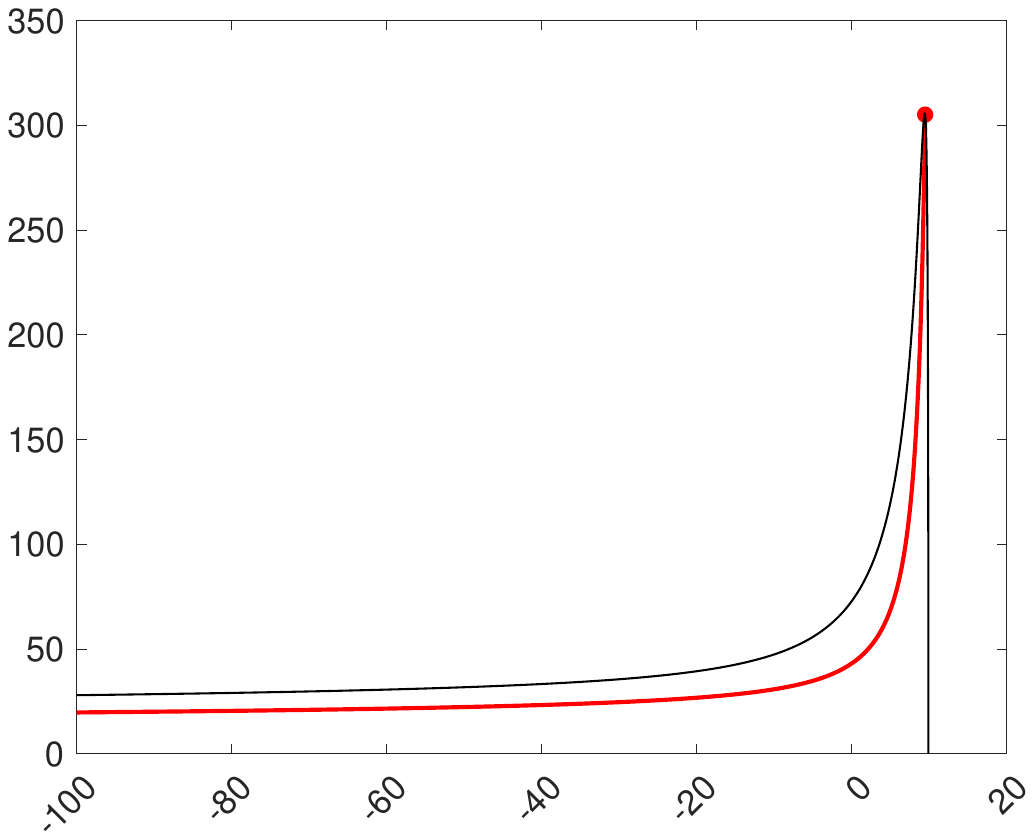} \put (12,84) {\tiny$\|u\|_2$}
	\put (97,15) {\tiny$\l$}
\end{overpic}
	\vspace{-0.4cm}
	\caption{Bifurcation diagram relative to \eqref{1.1} for $a(x)=a_{1,0}(x)$ with $h=0.95$ (left) and a zoom of it (right).}
	\label{Fig11}
\end{figure}

\begin{figure}[ht!]
	\centering
	\begin{overpic}[scale=0.28,trim = 1cm 5cm 1cm 9cm, clip]{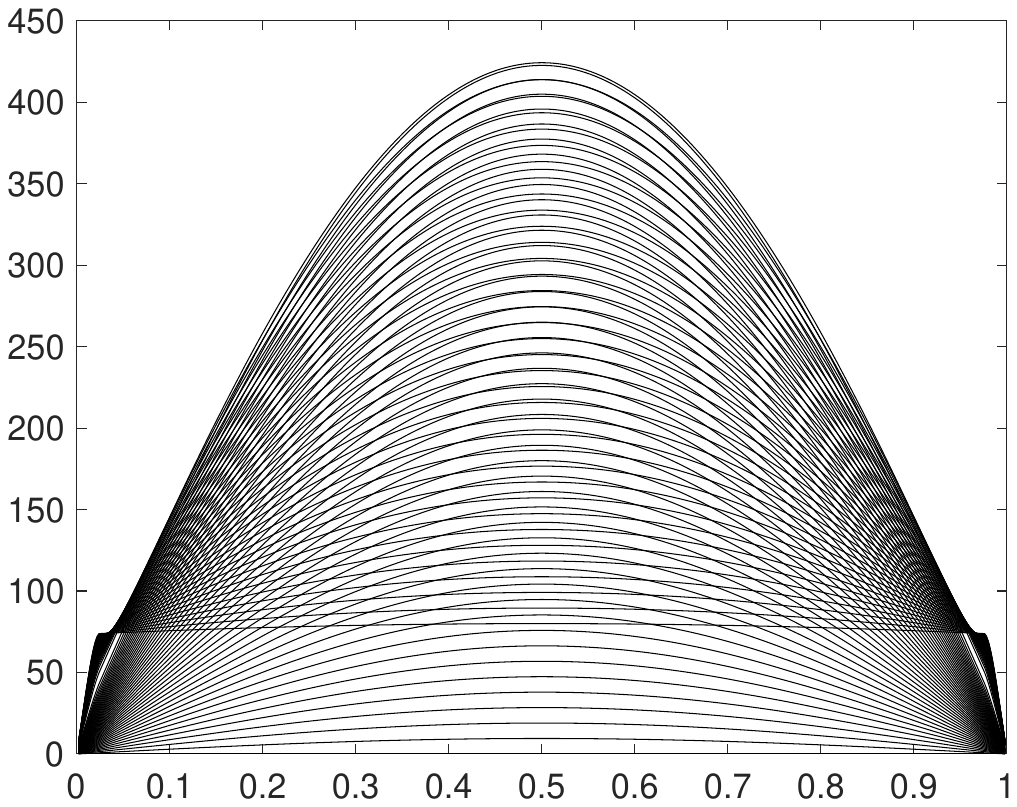} \put (12,83) {\tiny$u(x)$}
	\put (96.5,15) {\tiny$x$}\end{overpic} \begin{overpic}[scale=0.28,trim = 1cm 5cm 1cm 9cm, clip]{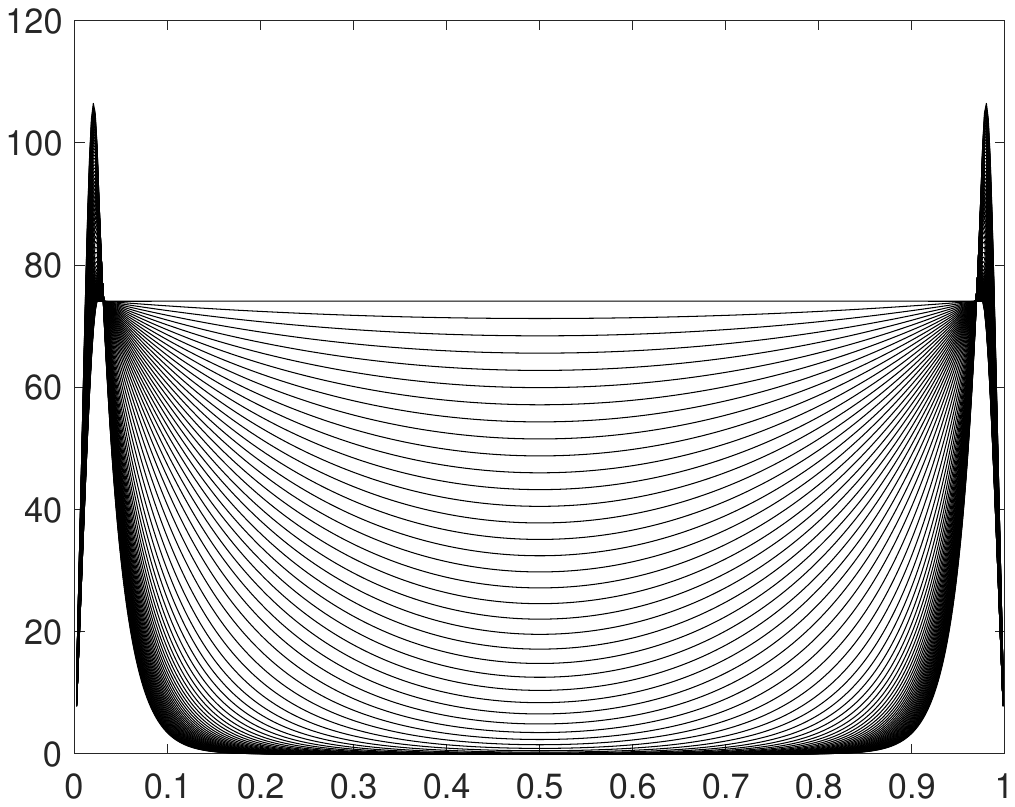} \put (12,83) {\tiny$u(x)$}
	\put (96.5,15) {\tiny$x$}\end{overpic} \\ [-2.5em]
	\begin{overpic}[scale=0.28,trim = 1cm 5cm 1cm 5cm, clip]{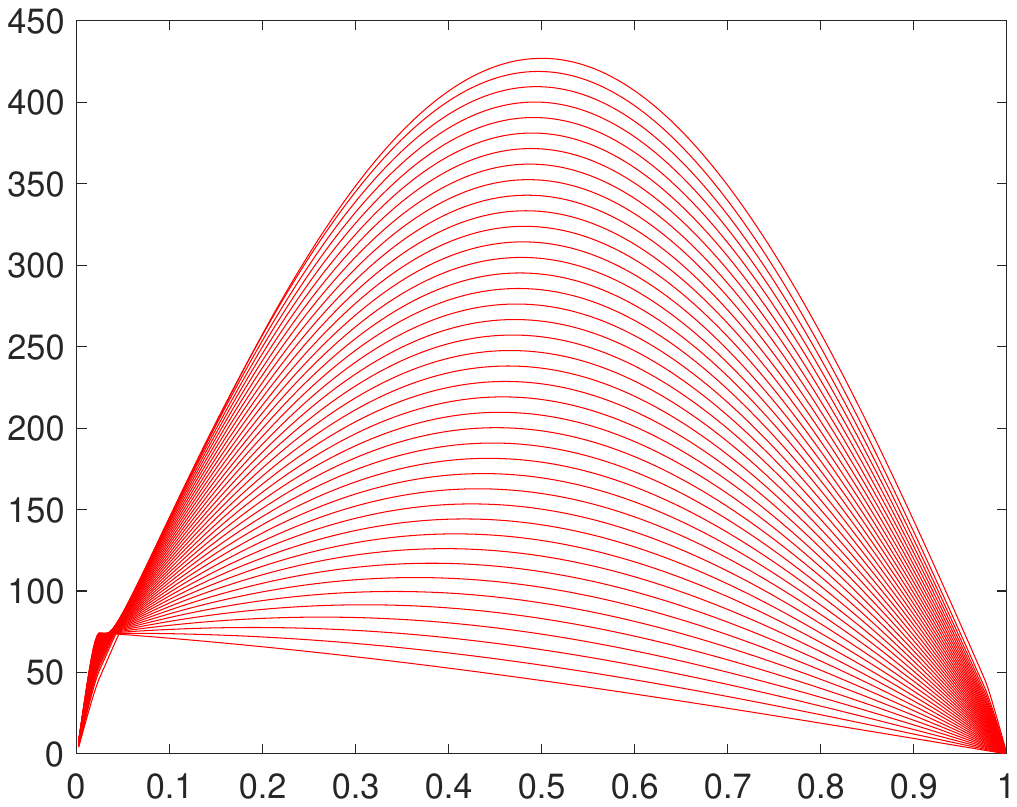} \put (12,80) {\tiny$u(x)$}
	\put (93,15) {\tiny$x$}\end{overpic} \begin{overpic}[scale=0.28,trim = 1cm 5cm 1cm 5cm, clip]{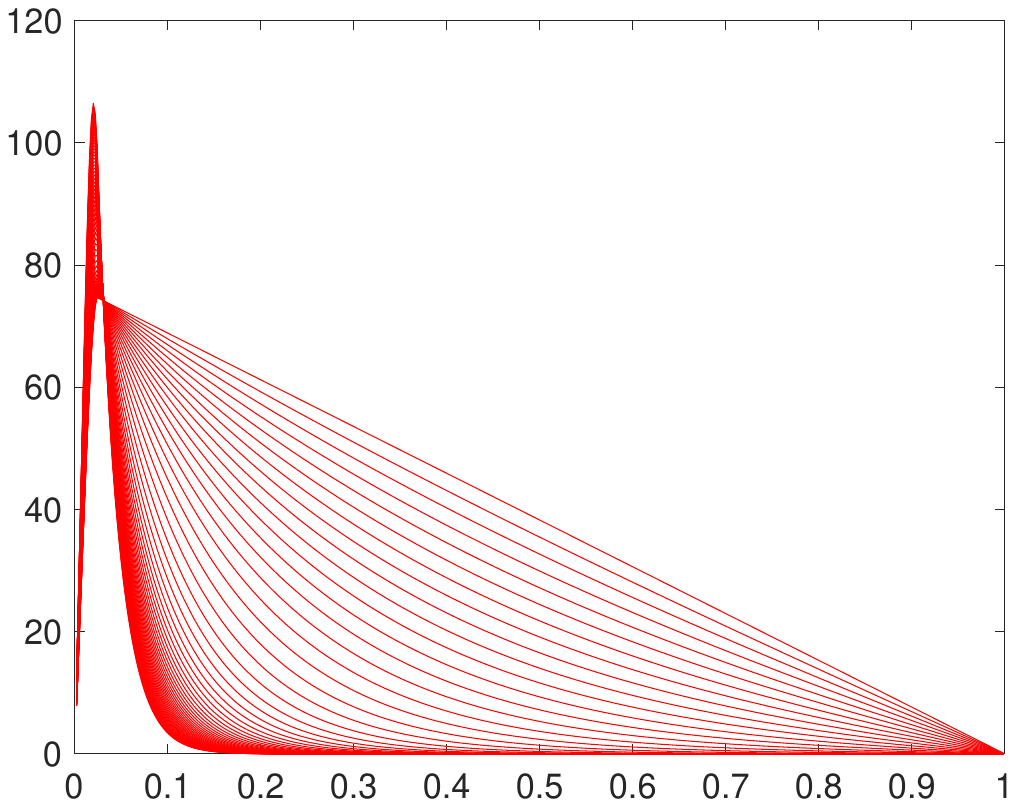} \put (12,80) {\tiny$u(x)$}
	\put (93,15) {\tiny$x$}\end{overpic}\\[-2.5em]
	\begin{overpic}[scale=0.28,trim = 1cm 5cm 1cm 5cm, clip]{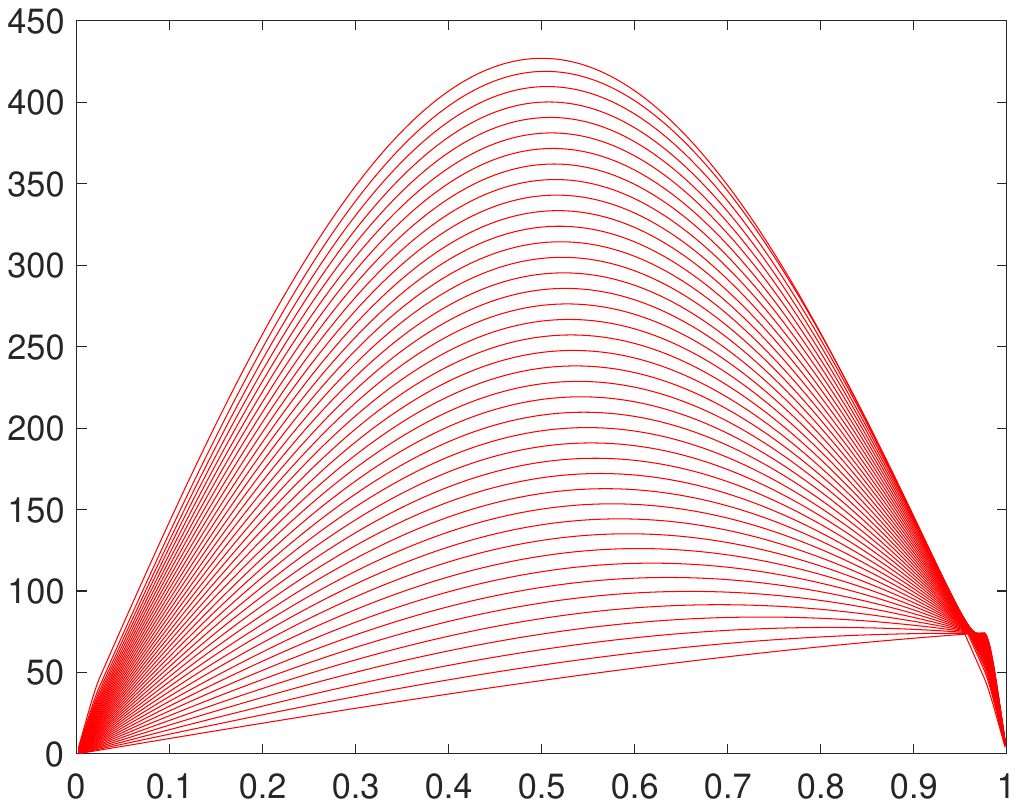} \put (12,80) {\tiny$u(x)$}
	\put (93,15) {\tiny$x$}\end{overpic} \begin{overpic}[scale=0.28,trim = 1cm 5cm 1cm 5cm, clip]{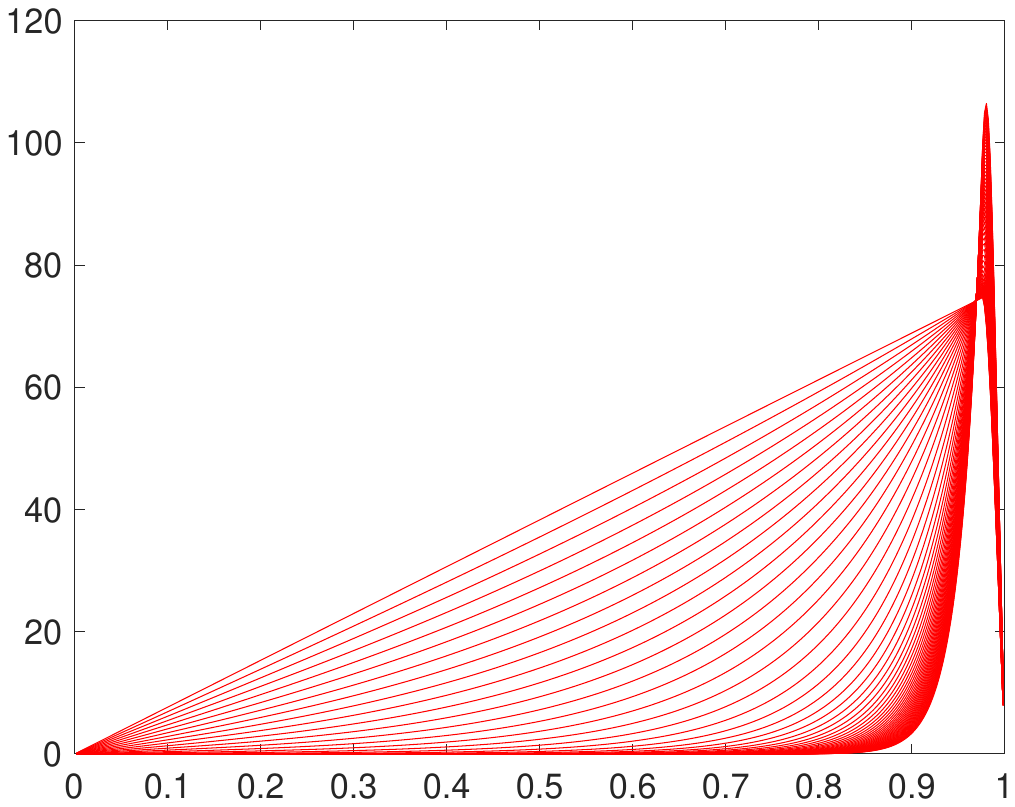} \put (12,80) {\tiny$u(x)$}
	\put (93,15) {\tiny$x$}\end{overpic}
	\vspace{-0.4cm}
	\caption{A series of plots of positive solutions along the three branches corresponding to the bifurcation diagram of Figure \ref{Fig11}: symmetric branch (first row) for $\l>0$ (left) and $\l<0$ (right), and asymmetric branches for $\l<\l_b\sim 9.44545$ (second and third row).}
	\label{Fig12}
\end{figure}

\section{The case $\kappa =2$ with $h\in (0,0.5)$}
\label{sec:4}

In this section,  we compute the global bifurcation diagrams of positive solutions of \eqref{1.1} for a series of choices of $a=a_{2,0}$ with $h\in (0,0.5)$. More precisely, we consider
\begin{equation}
	\label{4.1}
	a(x)=\begin{cases}
		0 & \text{if $x\in J_h:=\left(0.25-\frac{h}{2},0.25+\frac{h}{2}\right)\cup\left(0.75-\frac{h}{2},0.75+\frac{h}{2}\right)$,} \\
		1 & \text{if $x\in[0,1]\setminus J_h$.}
	\end{cases}
\end{equation}
Thus, in these examples, $a^{-1}(0)$ consists of two intervals. Figure \ref{Fig13} shows an example of this type of weight for $h=0.25$.

\begin{figure}[h!]
	\centering
	\begin{overpic}[scale=0.28,trim = 1cm 5cm 1cm 7cm, clip]{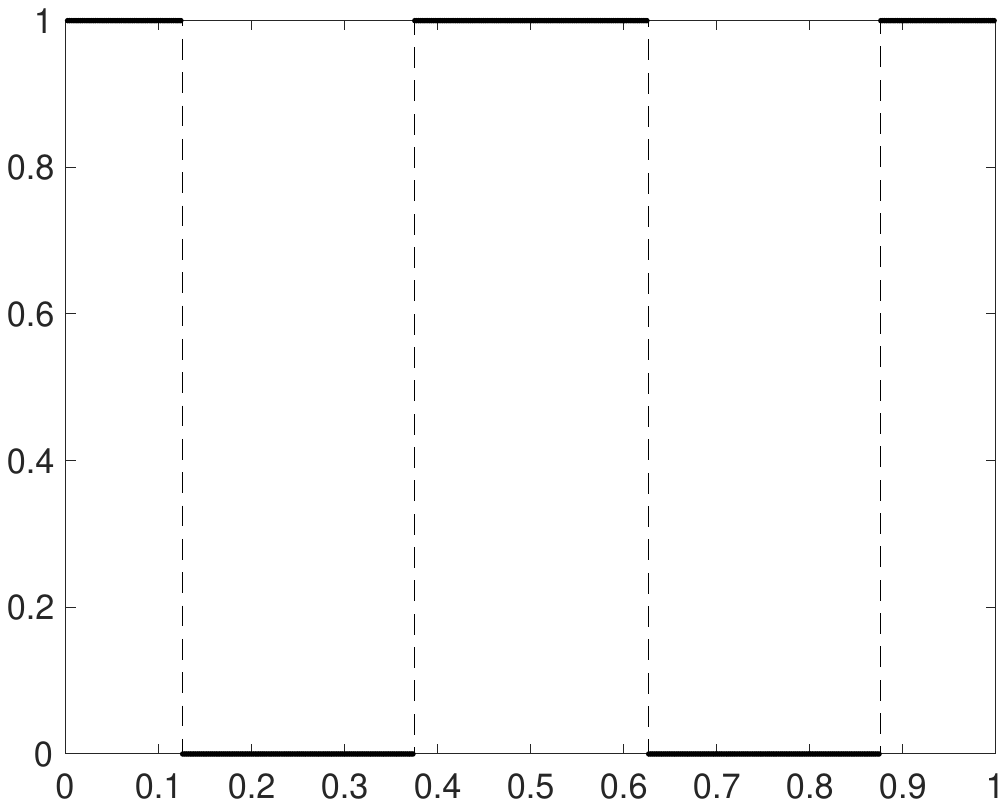} \put (13.5,83.5) {\tiny$a_{2,0}(x)$}
	\put (96.5,15) {\tiny$x$}
\end{overpic}
	\vspace{-0.8cm}
	\caption{An example of weight $a_{2,0}$ defined in \eqref{4.1} with $h=0.25$.}
	\label{Fig13}
\end{figure}

According to our numerical experiments, for all our choices of \eqref{4.1} with $h\in (0,0.5)$, the set of positive solutions of \eqref{1.1} consists of the main component $\mathscr{C}_0^+$, emanating from $u=0$ at $\l=\pi^2$, plotted in black, plus three subcritical global foliations separated away from the remaining components of the set of positive solutions. Observe that in Figure \ref{Fig14}, which corresponds to the weight of Figure \ref{Fig13}, one can distinguish only two of these additional components because two of them overlap due to the choice of the discrete $L^2$-norm on the vertical axis. Indeed, since $a(x)$ is symmetric about $0.5$, if $(\l,u(\cdot))$ is an asymmetric solution of \eqref{1.1}, then $(\lambda,u(1-\cdot))$ is another solution but has the same $L^2$-norm as $(\l,u(\cdot))$. These components, isolated from the remaining components of the diagram, are usually referred to as \emph{isolas} in the literature.
\par
In addition, observe that, according to the convention used in this paper to identify bifurcation points, the branches represented in the diagram of Figure \ref{Fig14} do not intersect, indicating the absence of bifurcation points. In particular, as opposed to the case of Section \ref{sec:3}, no secondary bifurcations have been detected on $\mathscr{C}_0^+$.

As it is apparent from the bifurcation diagram of Figure \ref{Fig14}, since the turning points of the isolas occur for $\l=-26.0214$ and $\l=-41.5460$, problem \eqref{1.1} with the weight $a(x)$ given in \eqref{4.1} with $h=0.25$ has at least one positive solution for every $\l<\pi^2$, at least five positive solutions for every $\l<-26.0214$ and at least seven positive solutions for $\l<-41.5460$.

\begin{figure}[h!]
	\centering
	\begin{overpic}[scale=0.28,trim = 1cm 5cm 1cm 7cm, clip]{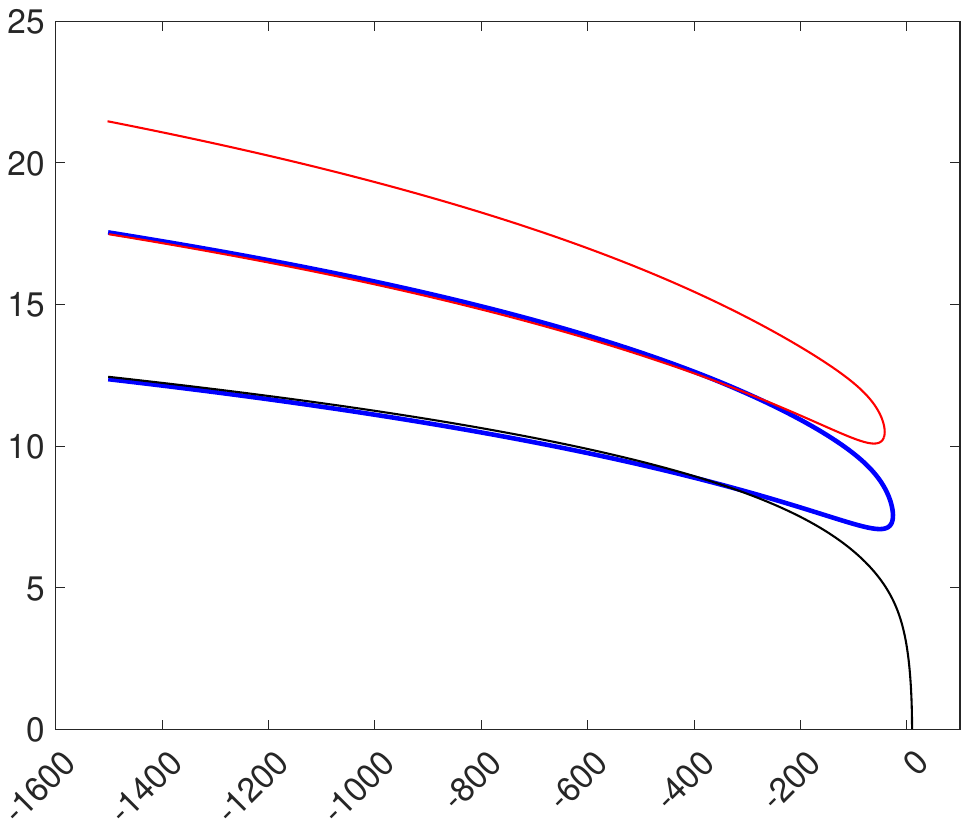}
		\put (12,87) {\tiny$\|u\|_2$}
		\put (93.5,20) {\tiny$\l$}
	\end{overpic}
	\vspace{-0.8cm}
	\caption{Computed bifurcation diagram relative to \eqref{1.1} for $a(x)=a_{2,0}(x)$ given by \eqref{4.1} with $h=0.25$.}
	\label{Fig14}
\end{figure}
Figure \ref{Fig15} collects a series of graphs of positive solutions along each of the components plotted in Figure \ref{Fig14}. In order to easily identify these solutions in the bifurcation diagram of Figure \ref{Fig14}, the profiles have been plotted using the same color of the component they belong to. Precisely, the first row of Figure \ref{Fig15} collects the plots of some symmetric solutions of the bifurcation diagram in Figure \ref{Fig14}. In particular, a series of positive solutions along the main branch $\mathscr{C}_0^+$ have been plotted in black and a series of positive solutions on the upper isola in red. The second and third rows collect the plots of a series of positive solutions along the upper and lower half branches of the two (superimposed) lower isolas of Figure \ref{Fig14}.

\par
In agreement with the conjecture presented in Remark \ref{re2.2} based on Theorem \ref{th2.1}, for sufficiently negative $\l<0$, problem \eqref{1.1} has $2^3-1=7$ positive solutions: one with three peaks, three with two peaks, and three with a single peak, as illustrated in Figure \ref{Fig15}.

\begin{figure}[h!]
	\centering
	\begin{tabularx}{\linewidth}{X}
		\hfill \begin{overpic}[scale=0.241,trim = 1cm 5cm 1cm 9cm, clip]{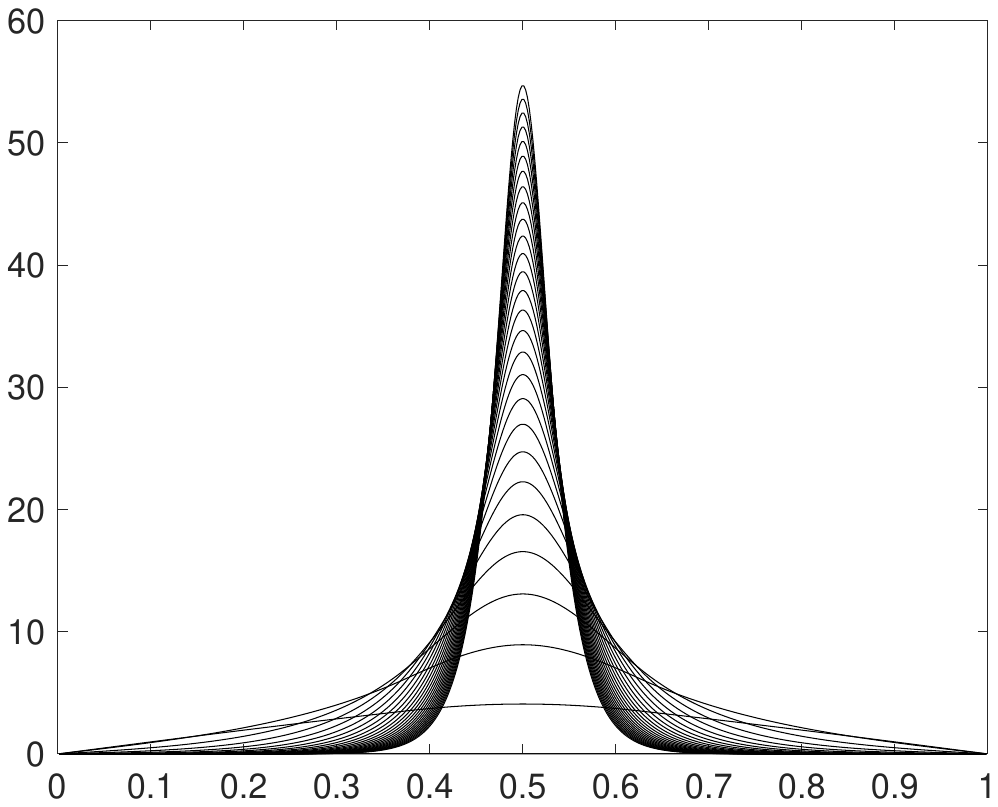} \put (12,83.5) {\tiny$u(x)$}
		\put (97,15) {\tiny$x$}\end{overpic}
		\hfill \begin{overpic}[scale=0.241,trim = 1cm 5cm 1cm 9cm, clip]{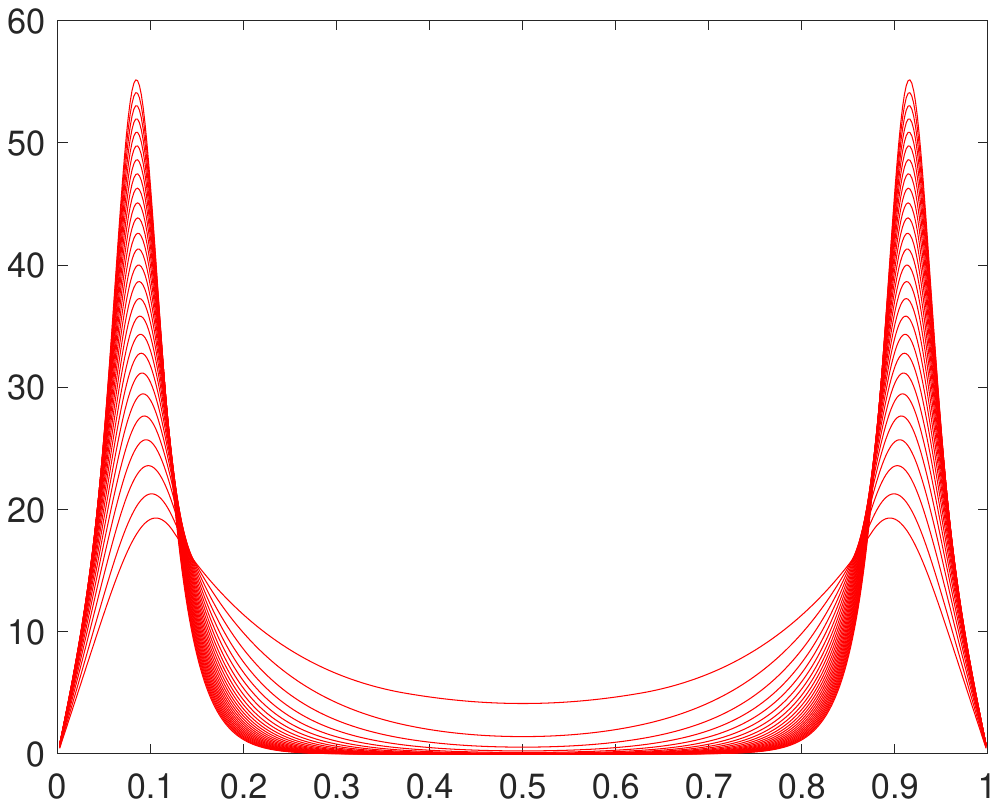} \put (12,83.5) {\tiny$u(x)$}
		\put (97,15) {\tiny$x$}\end{overpic}
		\hfill \begin{overpic}[scale=0.241,trim = 1cm 5cm 1cm 9cm, clip]{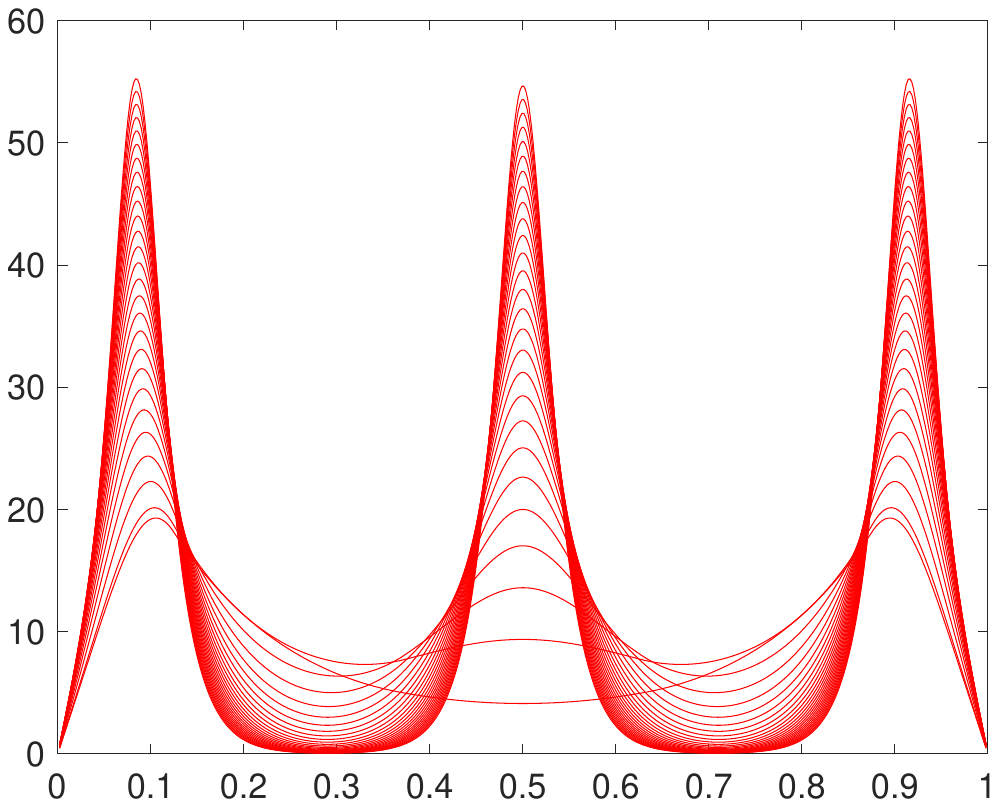} \put (12,83.5) {\tiny$u(x)$}
		\put (97,15) {\tiny$x$}\end{overpic} \hfill\null \\[-3em]
		\hfill \begin{overpic}[scale=0.241,trim = 1cm 5cm 1cm 5cm, clip]{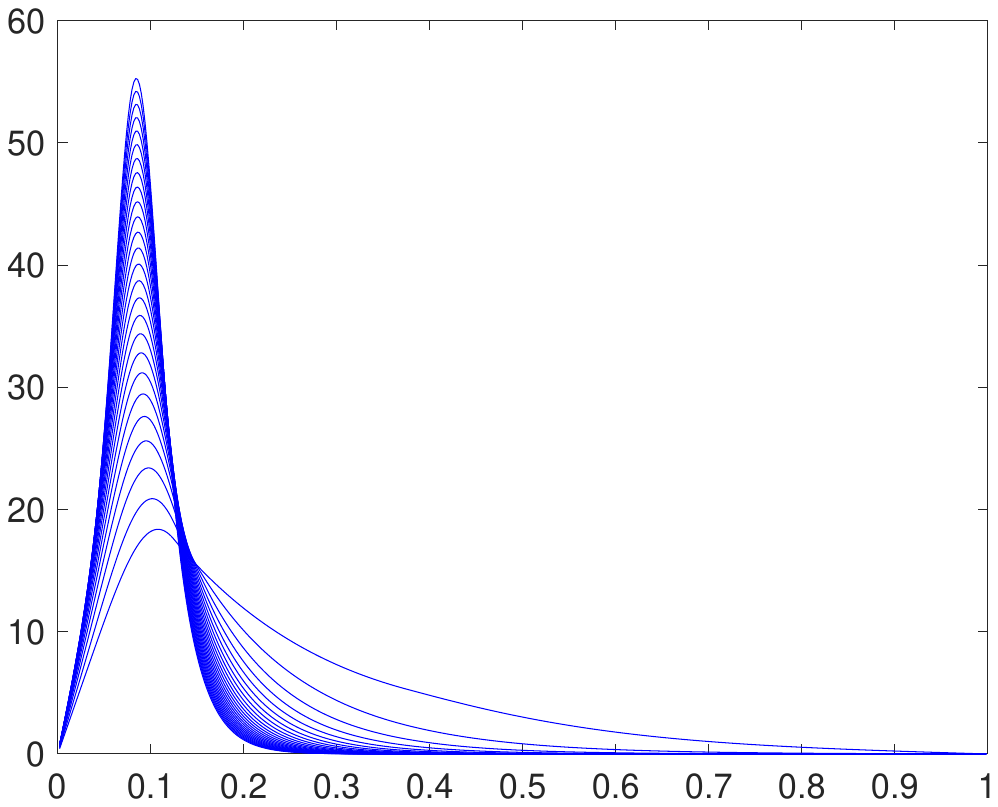} \put (12,80.5) {\tiny$u(x)$}
		\put (93.5,15) {\tiny$x$}\end{overpic}
		\hfill \begin{overpic}[scale=0.241,trim = 1cm 5cm 1cm 5cm, clip]{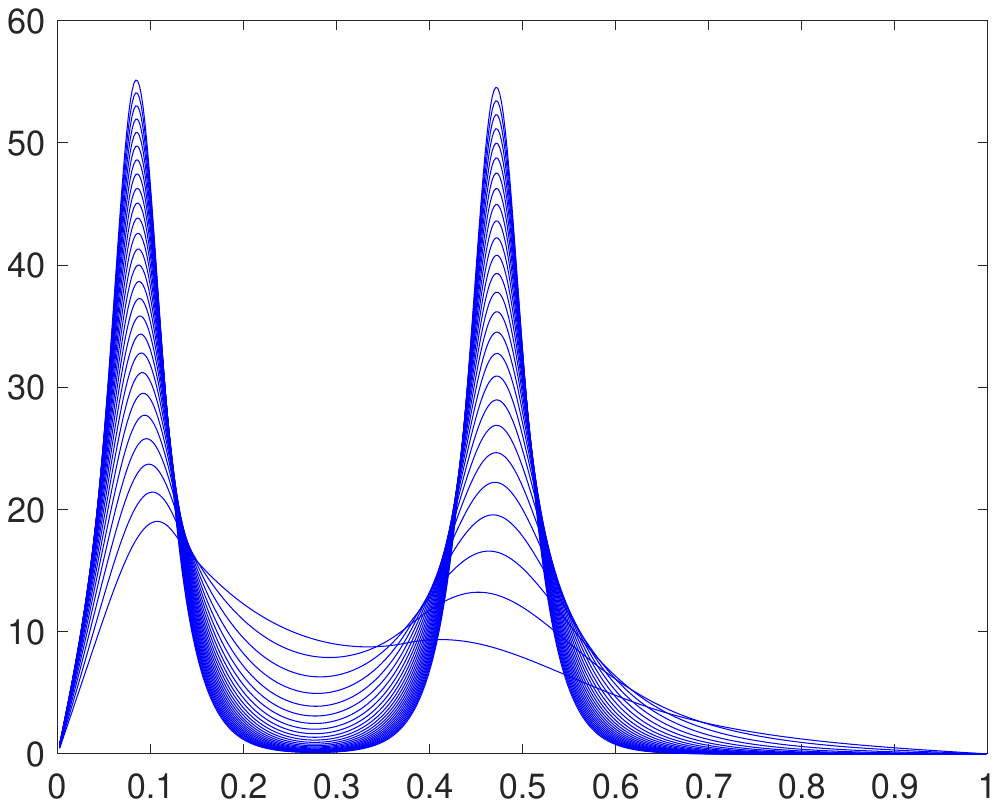} \put (12,80.5) {\tiny$u(x)$}
		\put (93.5,15) {\tiny$x$}\end{overpic} \hfill\null \\[-3em]
		\hfill \begin{overpic}[scale=0.241,trim = 1cm 5cm 1cm 5cm, clip]{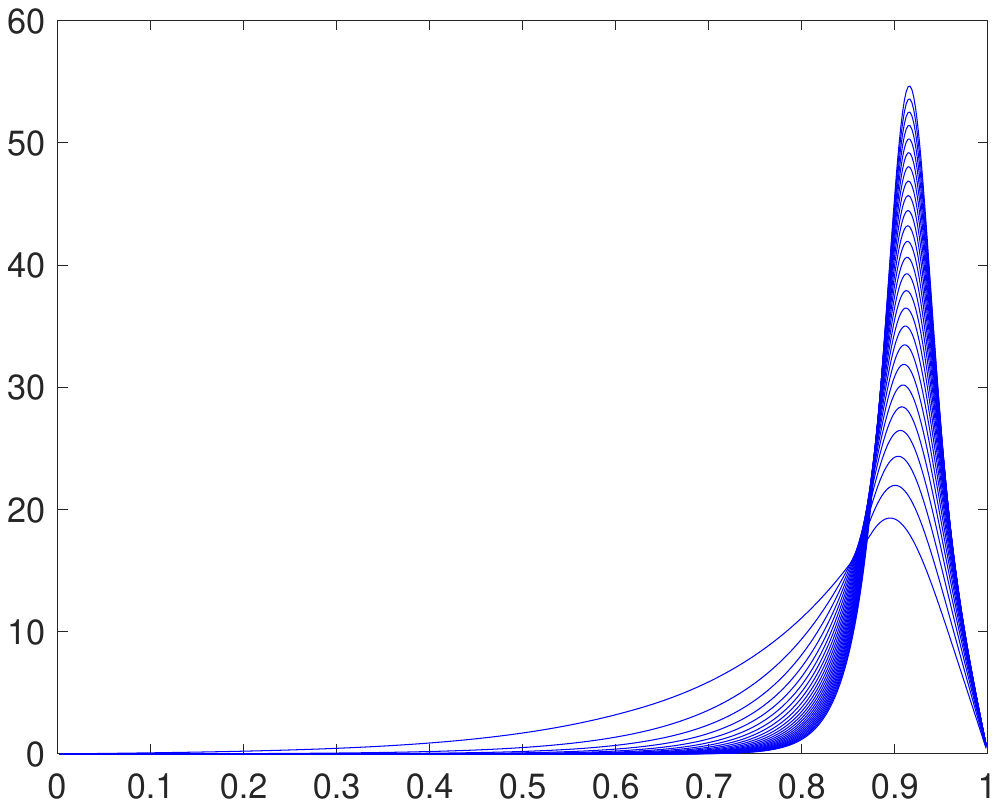} \put (12,80.5) {\tiny$u(x)$}
		\put (93.5,15) {\tiny$x$}\end{overpic}
		\hfill \begin{overpic}[scale=0.241,trim = 1cm 5cm 1cm 5cm, clip]{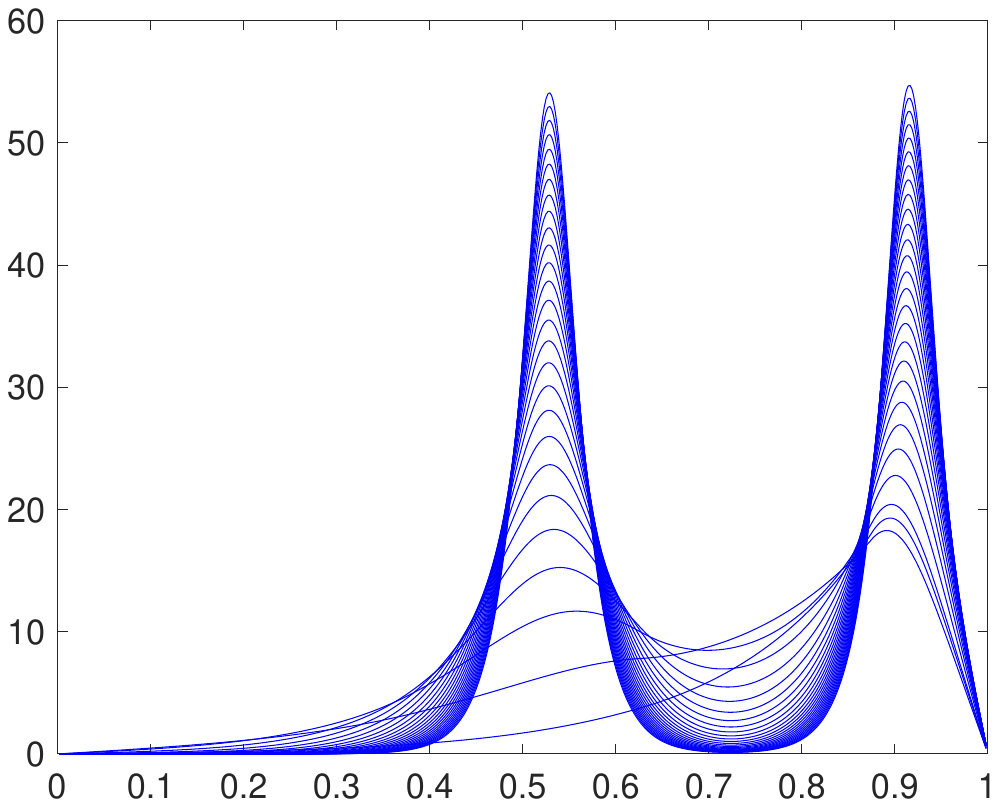} \put (12,80.5) {\tiny$u(x)$}
		\put (93.5,15) {\tiny$x$}\end{overpic} \hfill\null
	\end{tabularx}
	\vspace{-0.8cm}
	\caption{A series of positive solutions along the different branches of the bifurcation diagram of Figure \ref{Fig14}.}
	\label{Fig15}
\end{figure}

Figure \ref{Fig16} shows the global bifurcation diagrams computed for the weight \eqref{4.1} with $h=0.35$, $h=0.45$ and $h=0.49$. Note that, as $h$ increases, the subcritical turning points of the three isolas of positive solutions approximate $\pi^2$ from the left, and the right parts of the isolas approach the vertical line $\l=\pi^2$ as $\l\ua \pi^2$. Up to the best of our knowledge, this phenomenon has not been observed before.

\begin{figure}[h!]
	\centering
	\begin{overpic}[scale=0.255,trim = 1cm 5cm 1cm 7cm, clip]{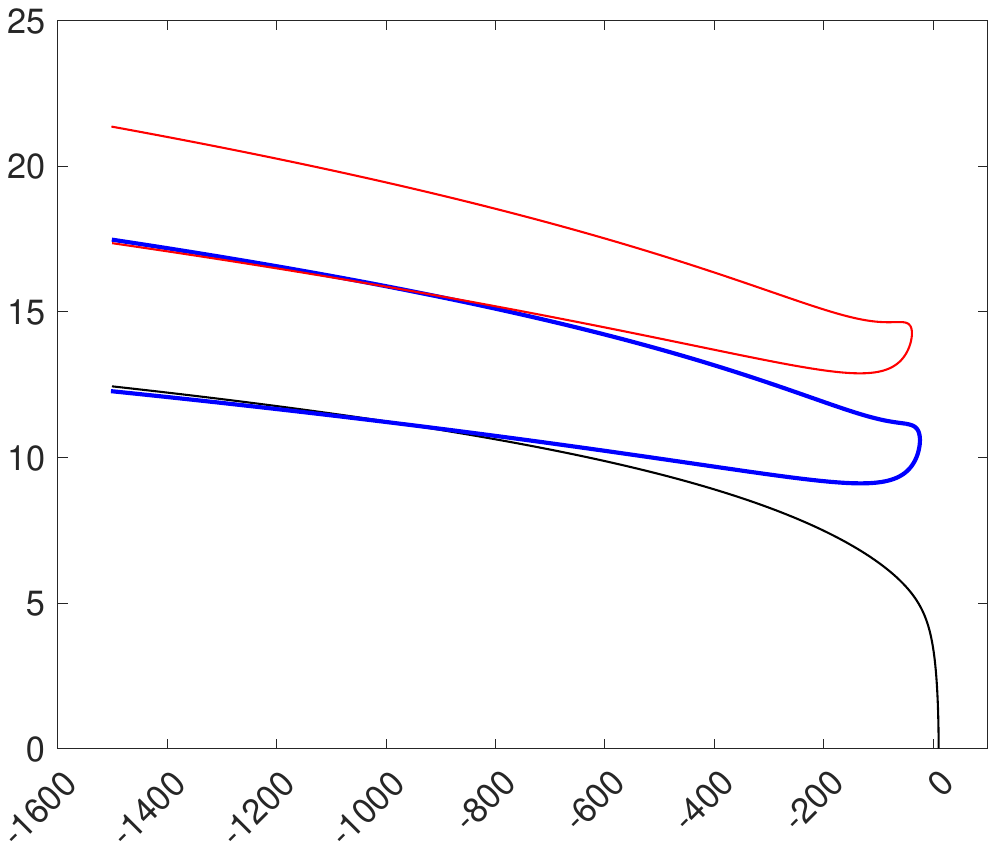} \put (12,84.5) {\tiny$\|u\|_2$}
	\put (96.5,15) {\tiny$\l$}
\end{overpic}
	\begin{overpic}[scale=0.255,trim = 1cm 5cm 1cm 7cm, clip]{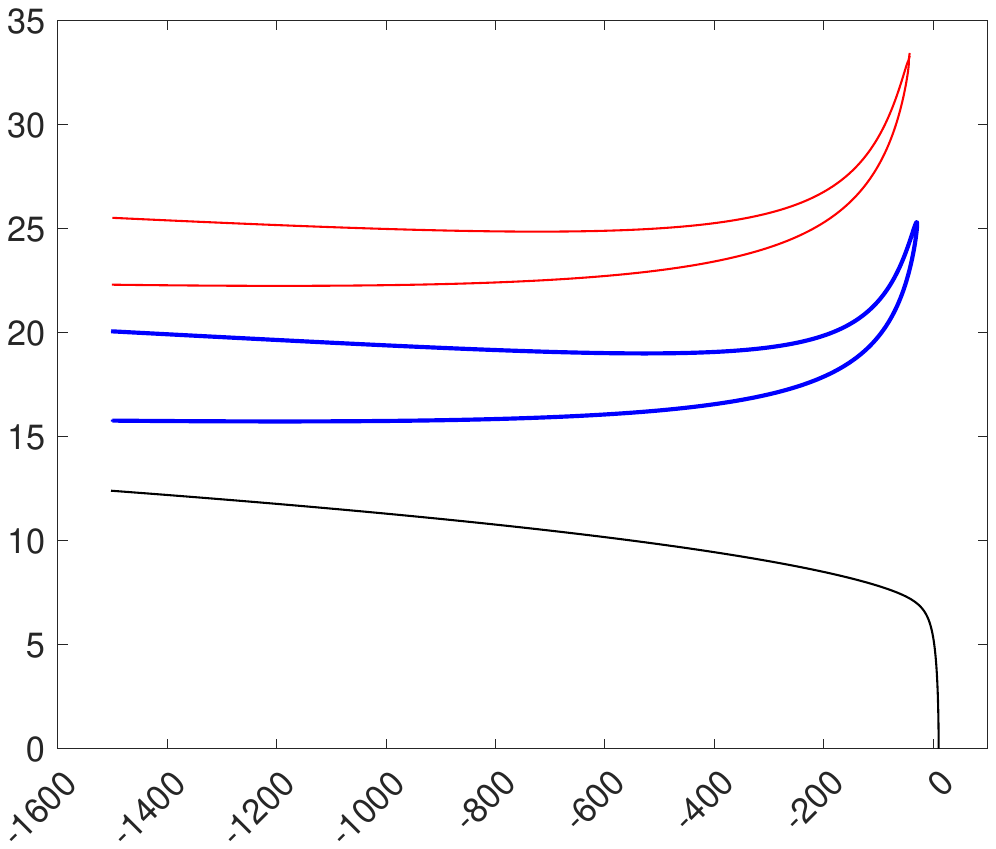} \put (12,84.5) {\tiny$\|u\|_2$}
		\put (96.5,15) {\tiny$\l$}
\end{overpic}
	\begin{overpic}[scale=0.255,trim = 1cm 5cm 1cm 7cm, clip]{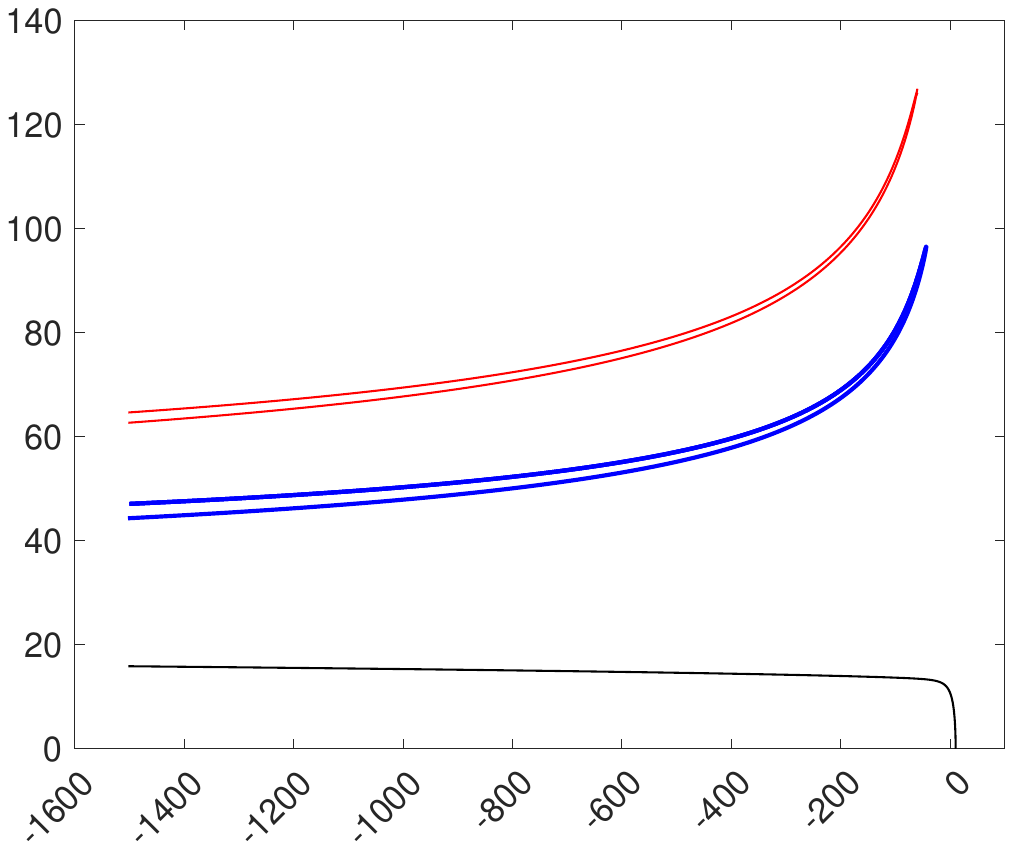} \put (12,84.5) {\tiny$\|u\|_2$}
		\put (96.5,15) {\tiny$\l$}
\end{overpic}
	\vspace{-0.4cm}
	\caption{Bifurcation diagrams of \eqref{1.1} for the weight \eqref{4.1}  with $h=0.35$ (left), $h=0.45$ (center) and $h=0.49$ (right).}
	\label{Fig16}
\end{figure}

\section{The case $\kappa =3$ with $h=0.3$}
\label{sec:5}

In this section we consider \eqref{1.1} with $a=a_{3,0}$, $h=0.3$, and $\frac{\a_1+\b_1}{2}=\frac{1}{6}$. Observe that, due its symmetry, this information is sufficient to completely identify the weight function $a(x)$, plotted in Figure \ref{Fig17}.
\begin{figure}[h!]
	\centering
	\begin{overpic}[scale=0.28,trim = 1cm 5cm 1cm 7cm, clip]{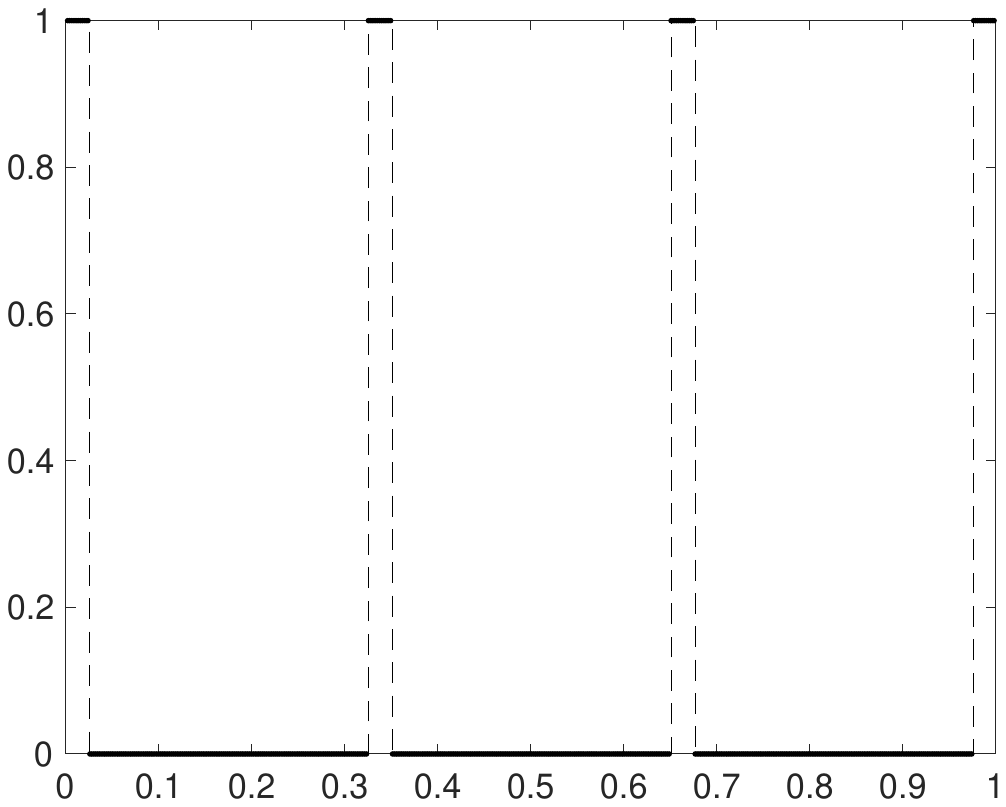} \put (13.5,83.5) {\tiny$a_{3,0}(x)$}
	\put (96.5,15) {\tiny$x$}
\end{overpic}
	\vspace{-0.8cm}
	\caption{Plot of $a_{3,0}$ for $h=0.3$.}
	\label{Fig17}
\end{figure}
According to Theorem \ref{th2.1} and Remark \ref{re2.2}, for such a choice, \eqref{1.1} should have $2^4-1=15$ positive solutions for sufficiently negative $\l<0$, which is confirmed by the numerical experiments (see Figures \ref{Fig18} and \ref{Fig19}).

\par
We proceed now with a brief description of the main features of the bifurcation diagram computed for this particular choice of $a(x)$ (see Figure \ref{Fig18}).  The component $\mathscr{C}_0^+$ bifurcating from $(\l,u)=(\pi^2,0)$ has a secondary subcritical pitchfork bifurcation at some $\l=\l_{b,1}\sim -0.7296$, so that, for every $\l<\l_{b,1}$, \eqref{1.1} has, at least, three positive solutions on $\mathscr{C}_0^+$.
The first row of Figure \ref{Fig19} shows, from left to right, a series of profiles of solutions along the main symmetric branch for $\l>0$ and for $\l<0$ and, then, a series of profiles of solutions along the two secondary half-branches bifurcating at $\l=\l_{b,1}$. In the case of symmetric solutions, for sufficiently negative $\l$, two peaks develop in the innermost intervals where $a=1$ while, for solutions along the secondary branches, only one peak located on either innermost interval where $a=1$ appears. As usual, observe that the secondary branches are superimposed in Figure \ref{Fig18}.

\begin{figure}[h!]
	\centering
	\begin{overpic}[scale=0.28,trim = 1cm 5cm 1cm 7cm, clip]{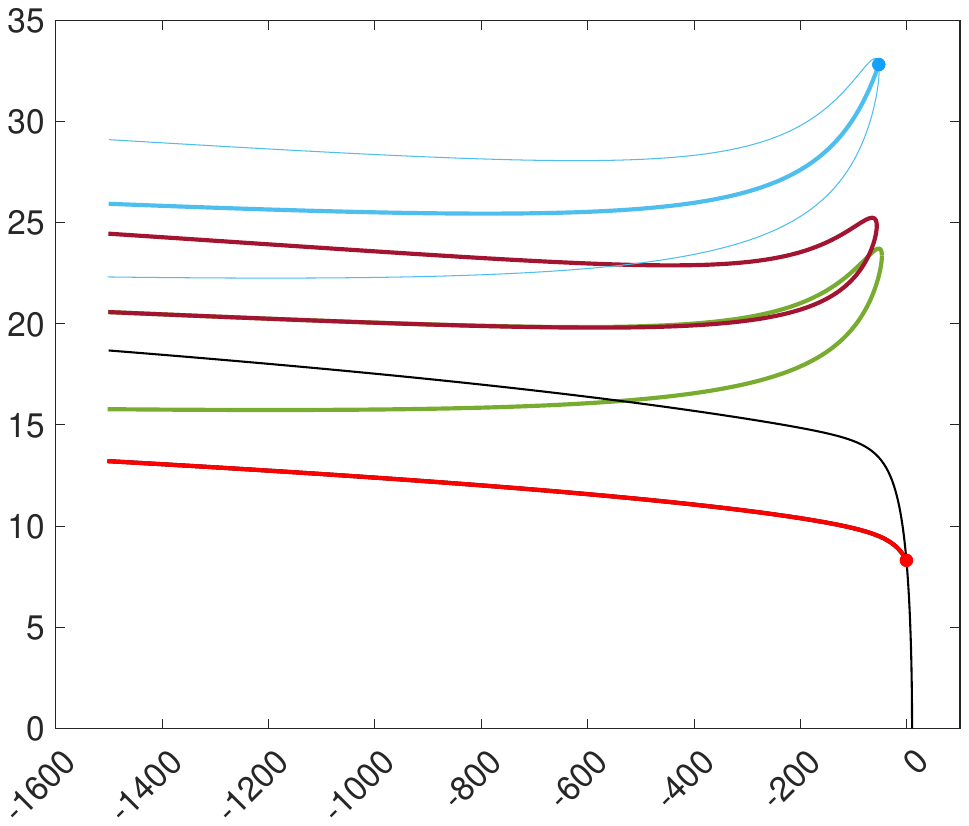} \put (12,87) {\tiny$\|u\|_2$}
	\put (94,20) {\tiny$\l$}
\end{overpic}
	\vspace{-0.8cm}
	\caption{Bifurcation diagram of \eqref{1.1} for $a=a_{3,0}$ with $h=0.3$ considered in Section \ref{sec:5}.}
	\label{Fig18}
\end{figure}

\par
Solutions along the lower two (superimposed) isolas, plotted in olive green in the second row of Figure \ref{Fig18} have, for sufficiently negative $\l$, either one peak located on the outermost intervals where $a=1$ or two peaks located in both left or right intervals where $a=1$. Solutions along the two mid (superimposed) isolas, plotted in brown in the third row of Figure \ref{Fig18} have, for sufficiently negative $\l$, either two peaks located at alternate intervals where $a=1$ or three peaks located in consecutive intervals where $a=1$.

\par
Finally, the set of positive solutions contains an upper isola with a subcritical turning point at $\l_t=-52.0364$, together with two global  branches emanating from it, subcritically, at some value of the parameter $\l=\l_{b,2}\sim -52.6275$. This component has been plotted in light blue in Figure \ref{Fig18}. Observe that, in the bifurcation diagram of Figure \ref{Fig18}, only three from a total of four branches of this component can be distinguished, because the solutions on two of these bifurcated half-branches are symmetric about $0.5$ of each other.
The last row of Figure \ref{Fig19} collects a series of plots of positive solutions of \eqref{1.1} on each of these four branches.
\par
Overall, for sufficiently negative $\l$, \eqref{1.1} has $4$ solutions with a single peak, $6$ solutions with two peaks, $4$ solutions with  three peaks, and $1$ solution with four peaks,  making a total of $15$ positive solutions, as predicted by Theorem \ref{th2.1} and Remark \ref{re2.2}.
\par
On each of the components represented in Figure \ref{Fig18}, the  degree of instability of the solutions, measured by the dimension of their unstable manifolds, changes by one every time a turning or a bifurcation point is crossed, much like in L\'{o}pez-G\'{o}mez, Molina-Meyer and Tellini \cite{LGMMT,LGMMT2}, and Fencl and L\'{o}pez-G\'{o}mez \cite{FLGN}.

\begin{figure}[h!]
	\begin{tabular}{cccc}
		\begin{overpic}[scale=0.2,trim = 1cm 5cm 1cm 9cm, clip]{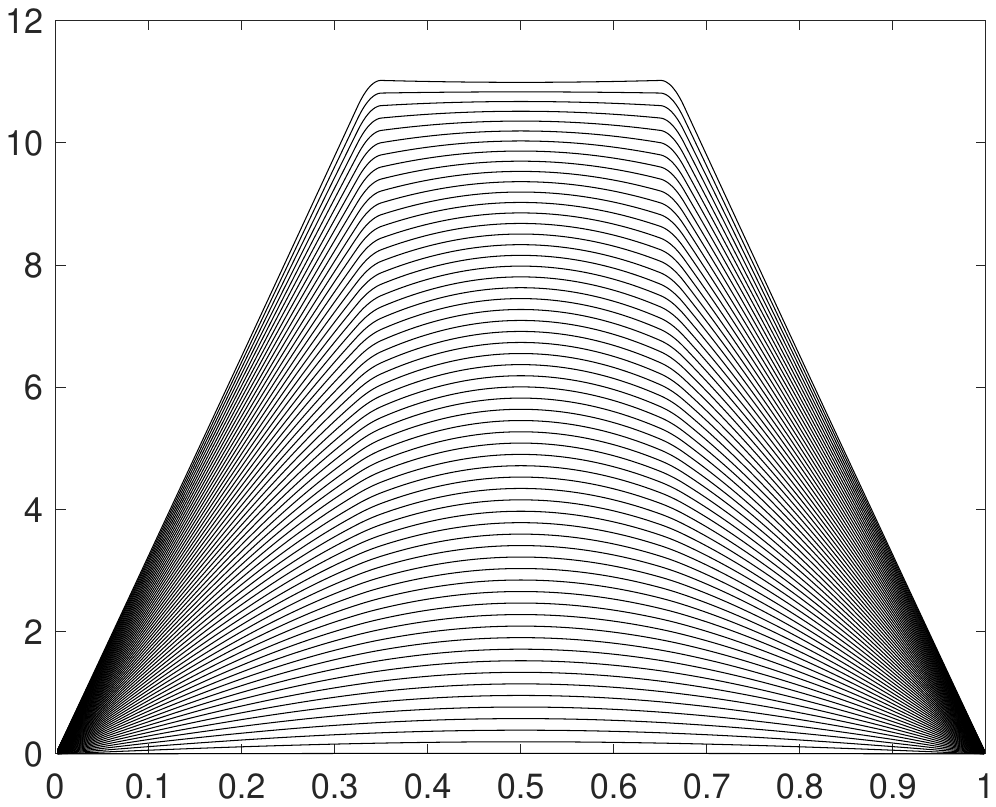}  \put (12,83.5) {\tiny$u(x)$}
		\put (96.5,15) {\tiny$x$}\end{overpic}  &
		\begin{overpic}[scale=0.2,trim = 1cm 5cm 1cm 9cm, clip]{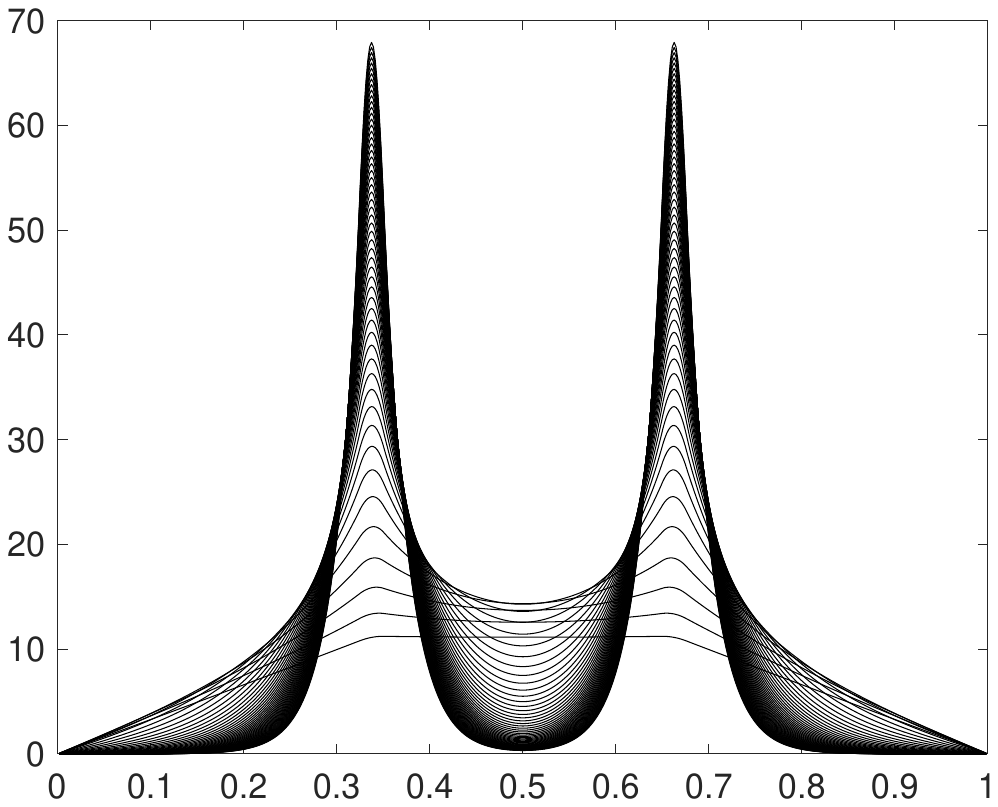}  \put (12,83.5) {\tiny$u(x)$}
		\put (96.5,15) {\tiny$x$}\end{overpic} &
		
		\begin{overpic}[scale=0.2,trim = 1cm 5cm 1cm 9cm, clip]{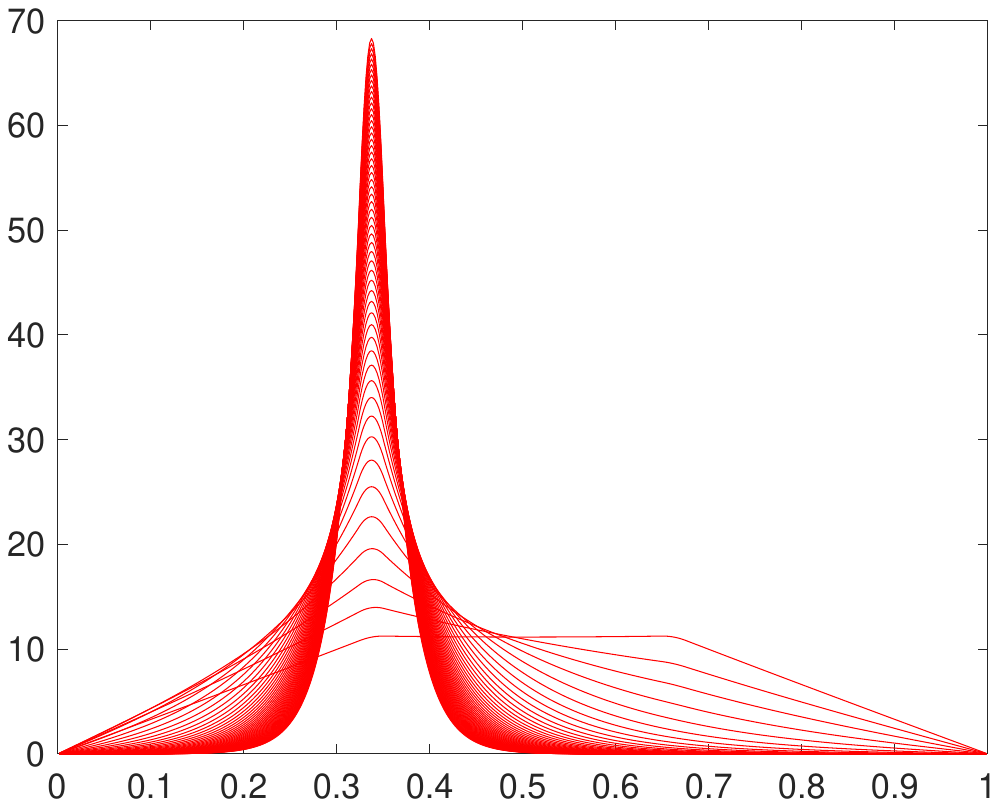}  \put (12,83.5) {\tiny$u(x)$}
		\put (96.5,15) {\tiny$x$}\end{overpic} &
		\begin{overpic}[scale=0.2,trim = 1cm 5cm 1cm 9cm, clip]{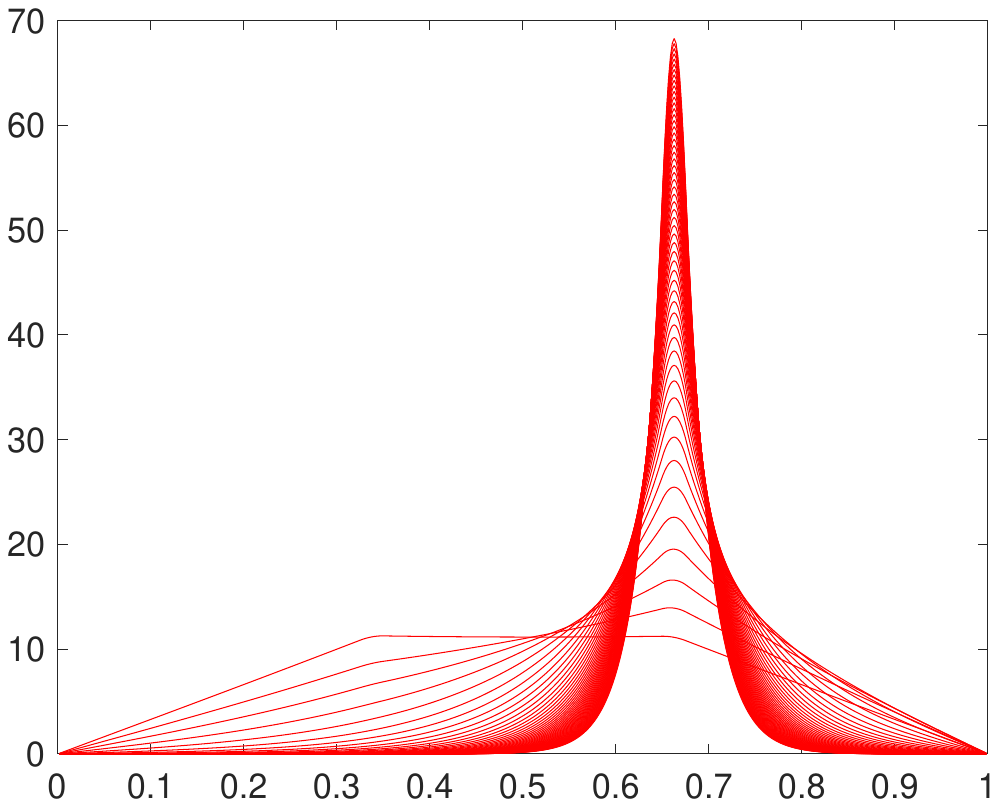}  \put (12,83.5) {\tiny$u(x)$}
		\put (96.5,15) {\tiny$x$}\end{overpic}  \\[-2.5em]
		
		\begin{overpic}[scale=0.2,trim = 1cm 5cm 1cm 5cm, clip]{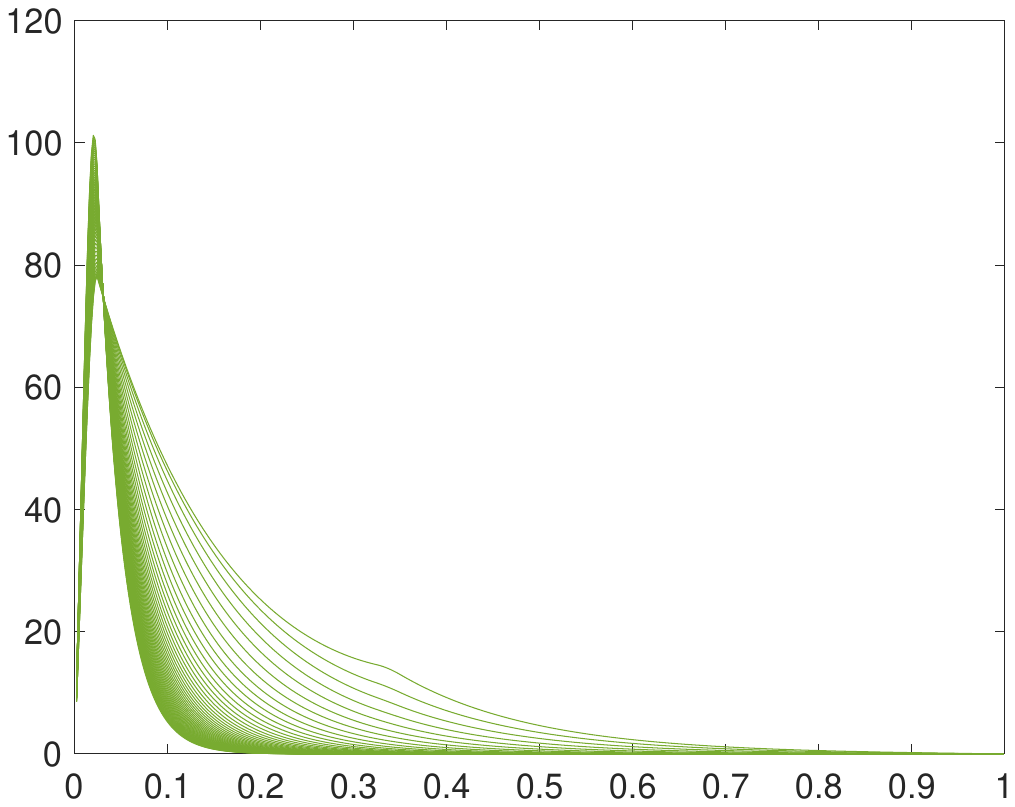}  \put (12,80.5) {\tiny$u(x)$}
		\put (93.5,15) {\tiny$x$}\end{overpic} &
		\begin{overpic}[scale=0.2,trim = 1cm 5cm 1cm 5cm, clip]{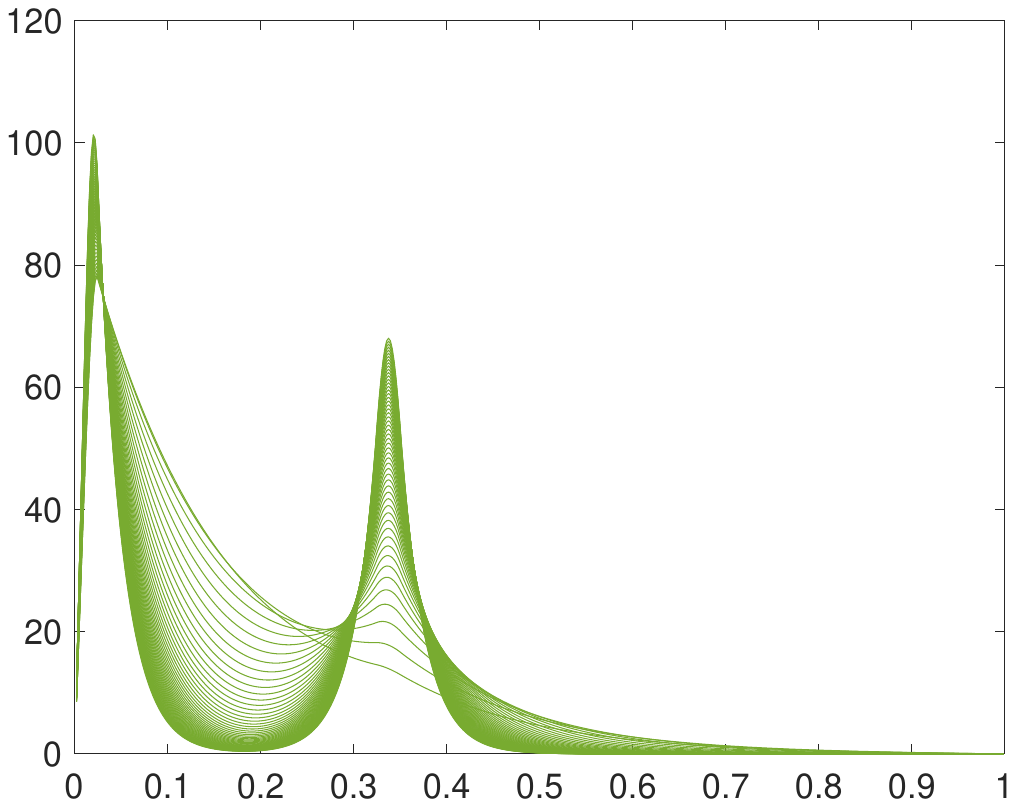}  \put (12,80.5) {\tiny$u(x)$}
		\put (93.5,15) {\tiny$x$}\end{overpic}  &
		\begin{overpic}[scale=0.2,trim = 1cm 5cm 1cm 5cm, clip]{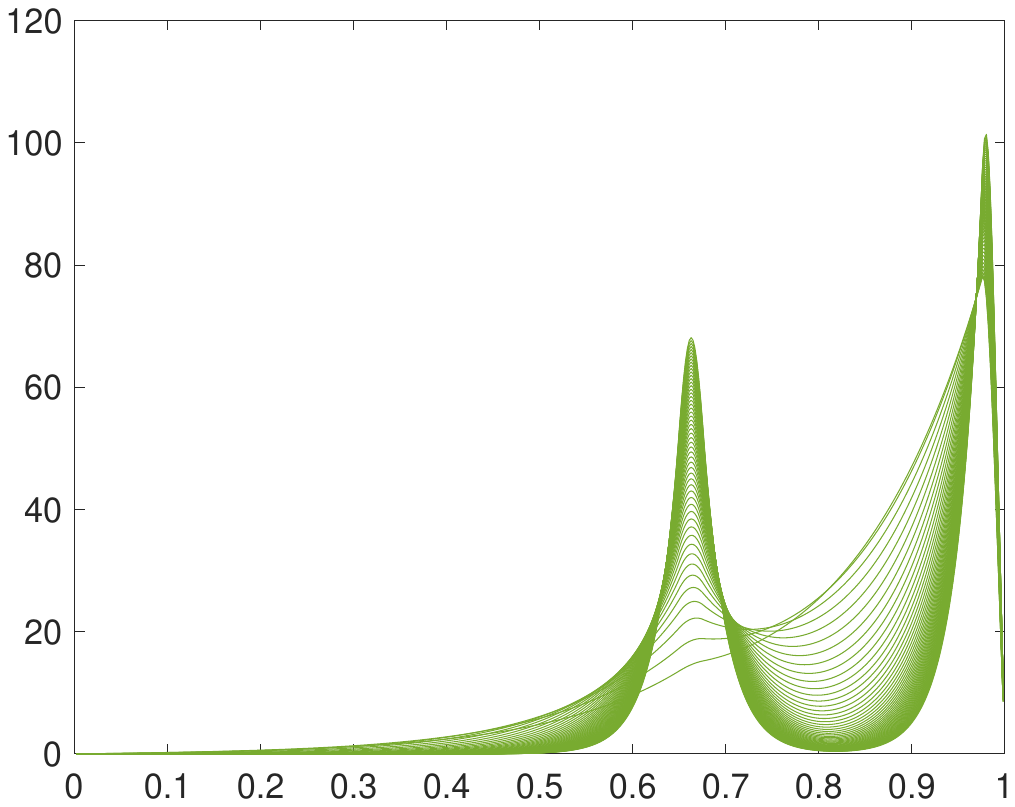}  \put (12,80.5) {\tiny$u(x)$}
		\put (93.5,15) {\tiny$x$}\end{overpic} &
		\begin{overpic}[scale=0.2,trim = 1cm 5cm 1cm 5cm, clip]{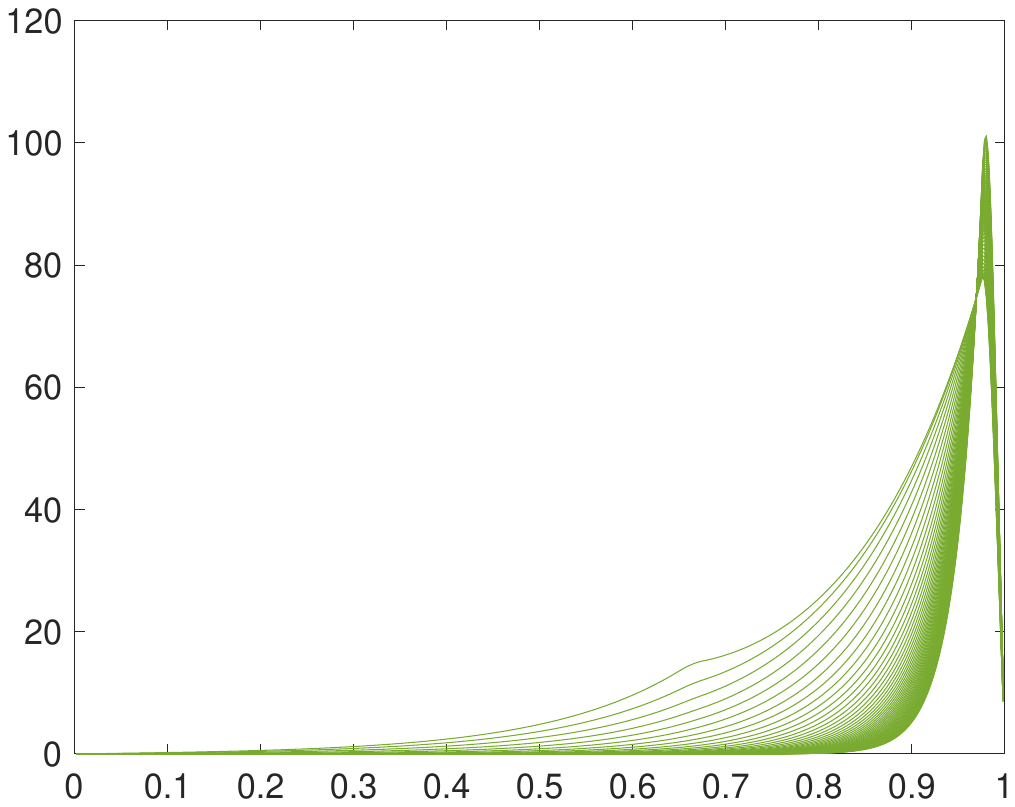}  \put (12,80.5) {\tiny$u(x)$}
		\put (93.5,15) {\tiny$x$}\end{overpic} \\[-2.5em]
		
		\begin{overpic}[scale=0.2,trim = 1cm 5cm 1cm 5cm, clip]{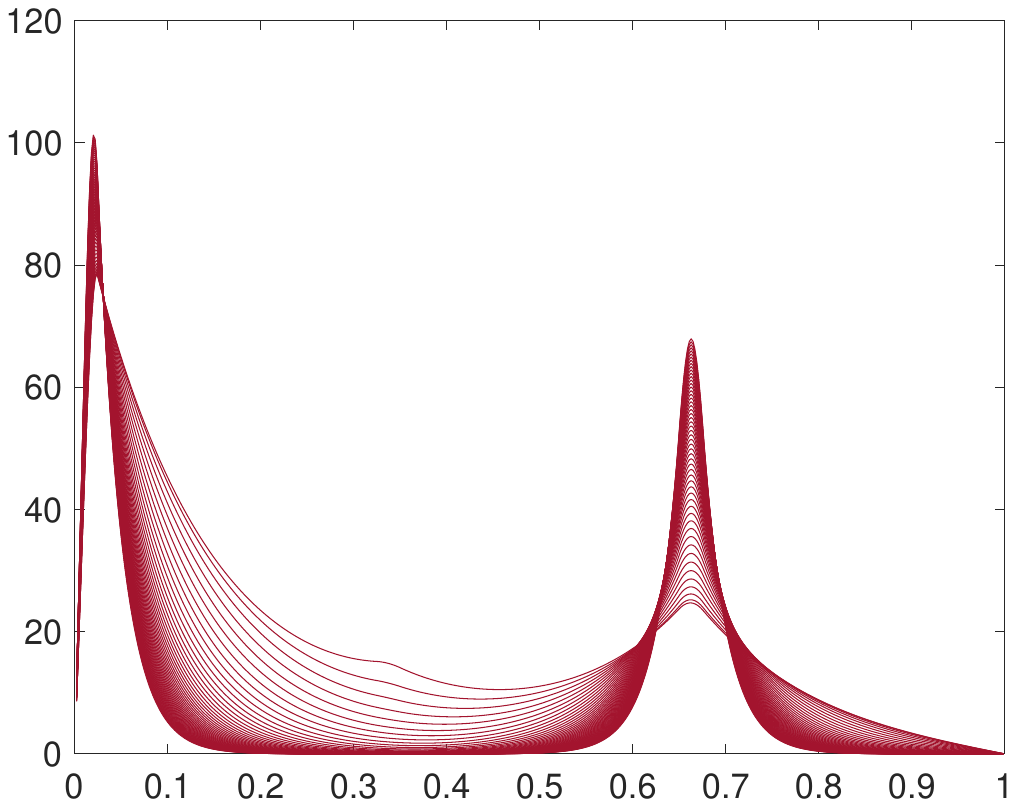}  \put (12,80.5) {\tiny$u(x)$}
		\put (93.5,15) {\tiny$x$}\end{overpic} &
		\begin{overpic}[scale=0.2,trim = 1cm 5cm 1cm 5cm, clip]{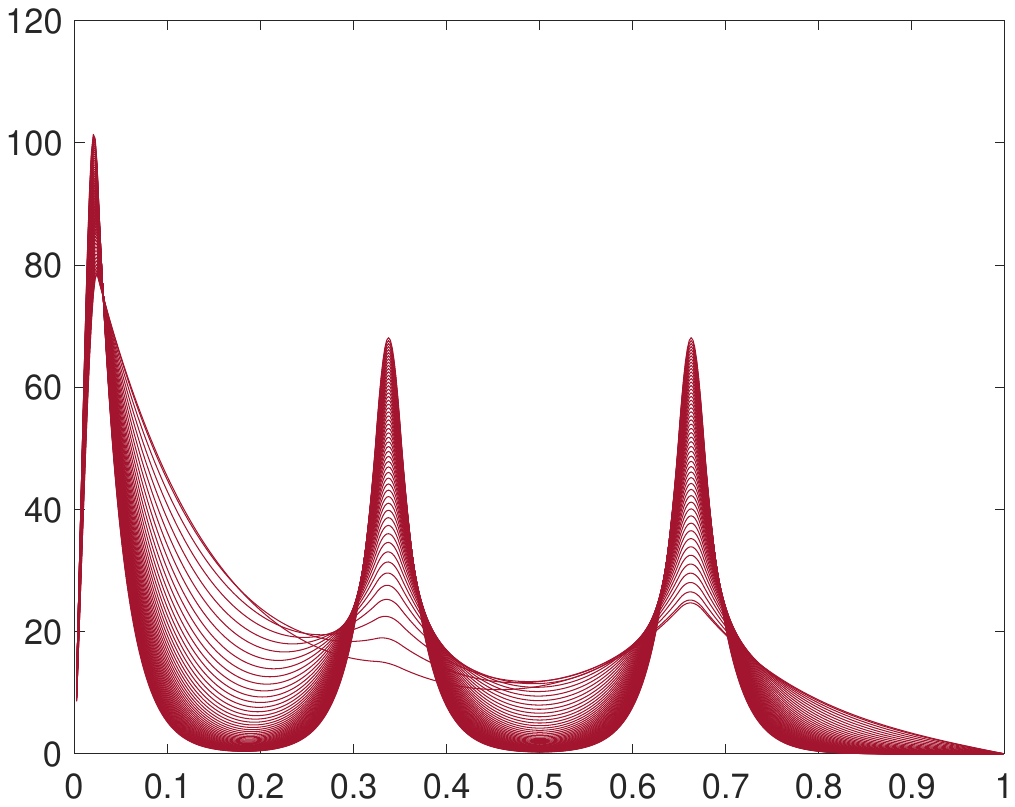}  \put (12,80.5) {\tiny$u(x)$}
		\put (93.5,15) {\tiny$x$}\end{overpic} &
		\begin{overpic}[scale=0.2,trim = 1cm 5cm 1cm 5cm, clip]{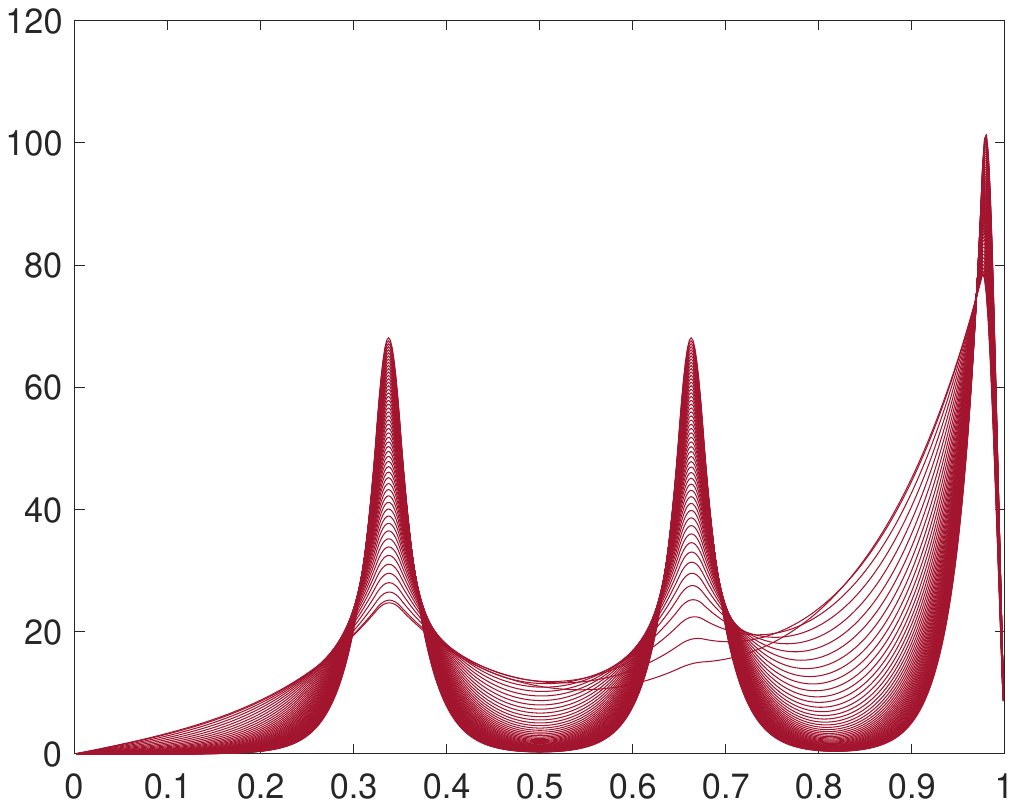}  \put (12,80.5) {\tiny$u(x)$}
		\put (93.5,15) {\tiny$x$}\end{overpic}  &
		\begin{overpic}[scale=0.2,trim = 1cm 5cm 1cm 5cm, clip]{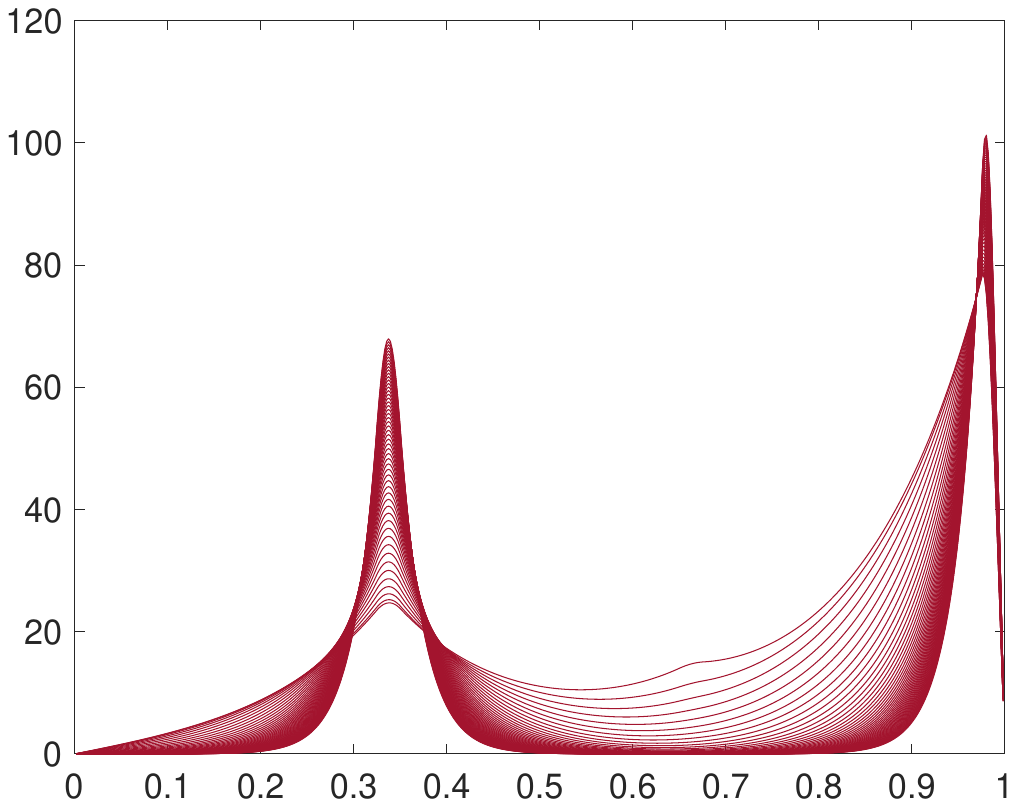}  \put (12,80.5) {\tiny$u(x)$}
		\put (93.5,15) {\tiny$x$}\end{overpic} \\[-2.5em]
		
		\begin{overpic}[scale=0.2,trim = 1cm 5cm 1cm 5cm, clip]{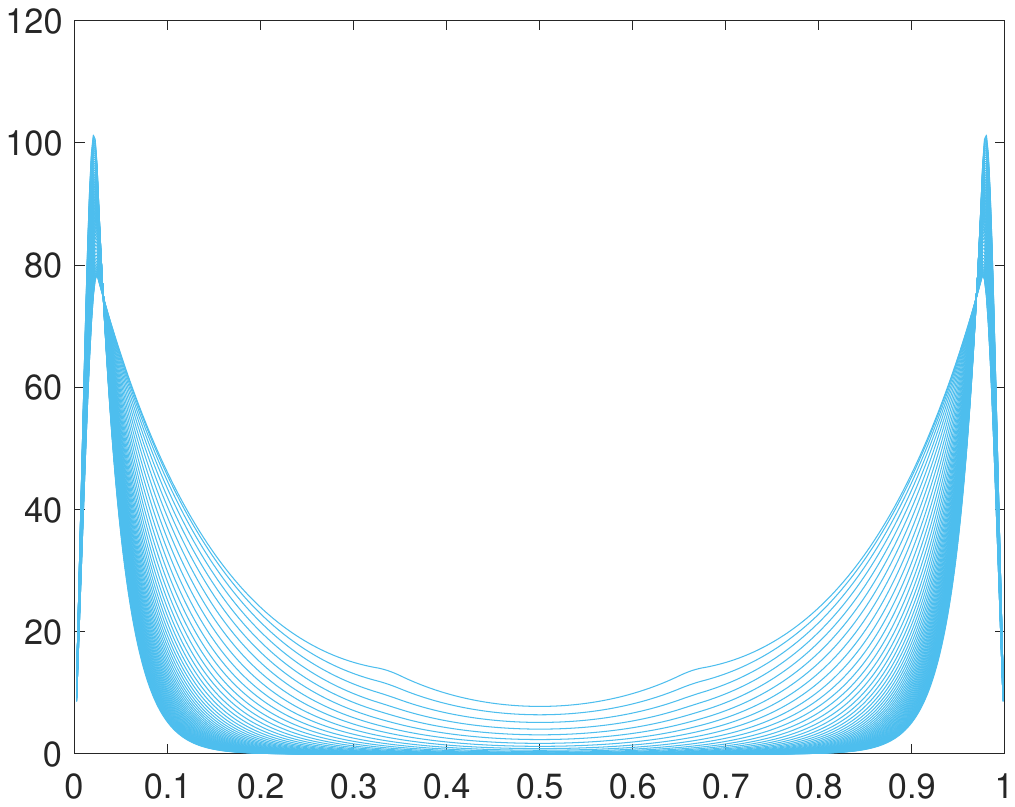}  \put (12,80.5) {\tiny$u(x)$}
		\put (93.5,15) {\tiny$x$}\end{overpic}  &
		\begin{overpic}[scale=0.2,trim = 1cm 5cm 1cm 5cm, clip]{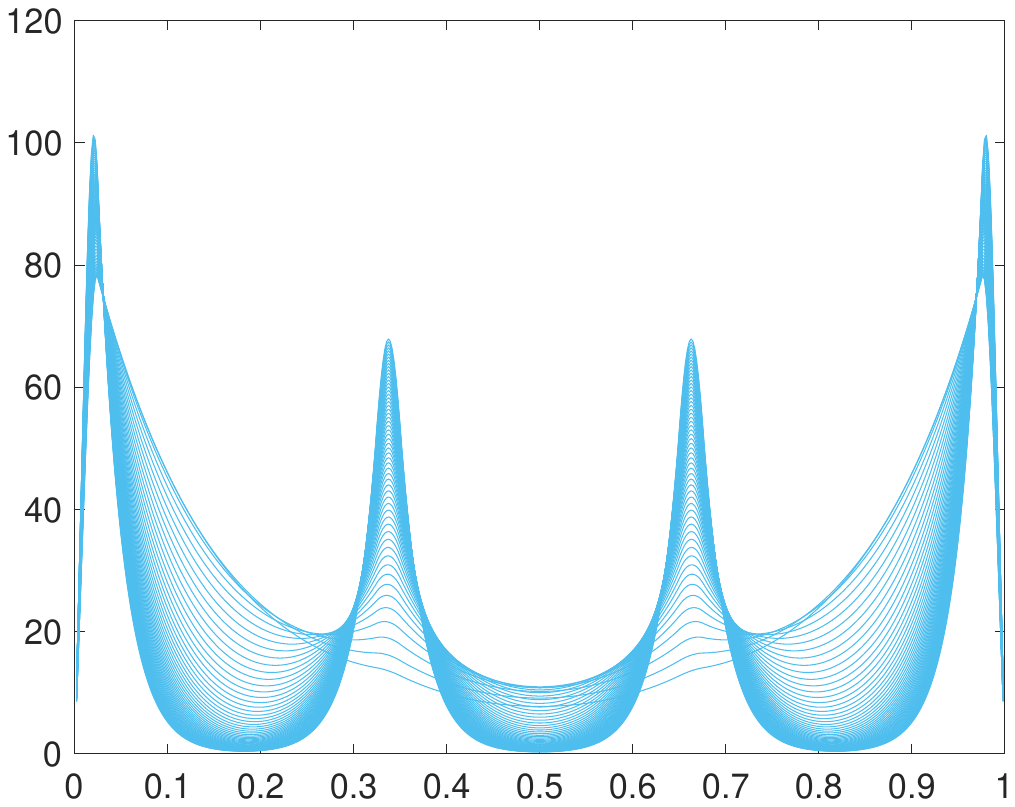}  \put (12,80.5) {\tiny$u(x)$}
		\put (93.5,15) {\tiny$x$}\end{overpic} &
		\begin{overpic}[scale=0.2,trim = 1cm 5cm 1cm 5cm, clip]{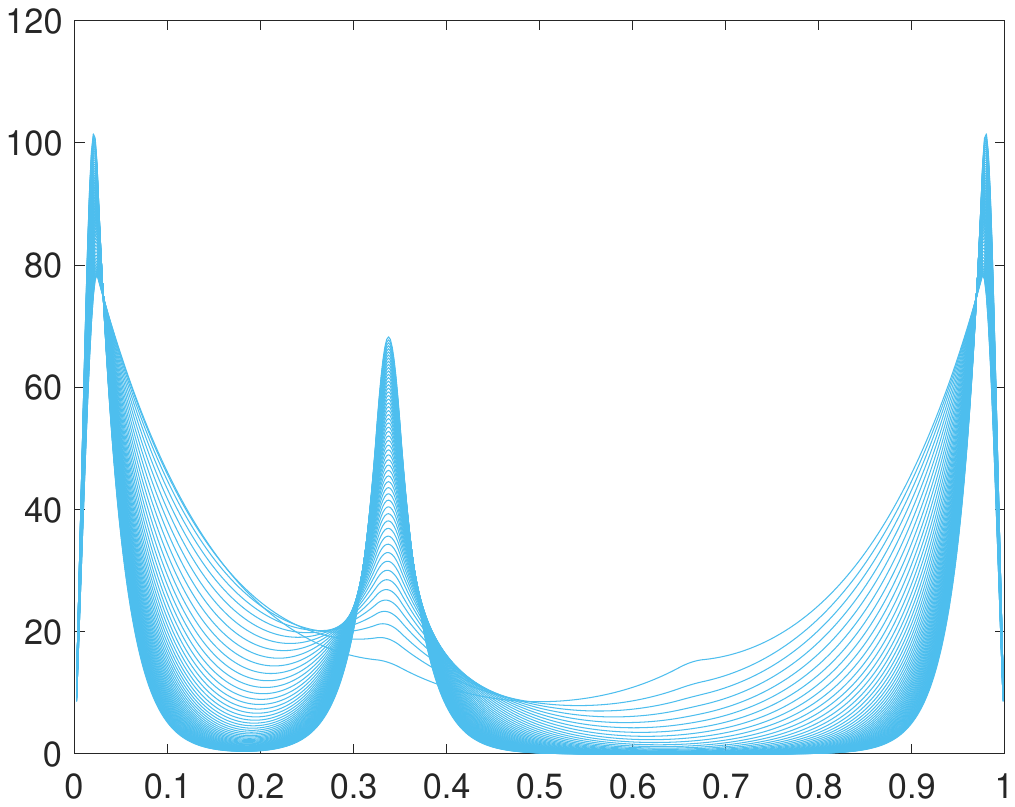}  \put (12,80.5) {\tiny$u(x)$}
		\put (93.5,15) {\tiny$x$}\end{overpic}  &
		\begin{overpic}[scale=0.2,trim = 1cm 5cm 1cm 5cm, clip]{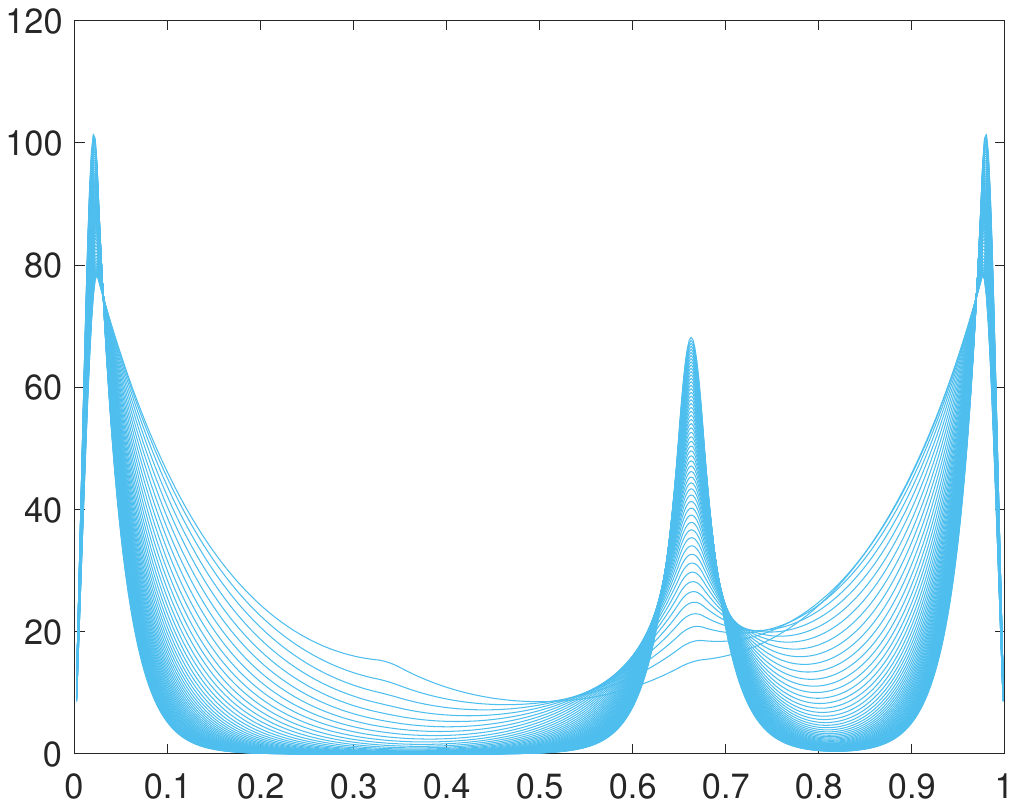}  \put (12,80.5) {\tiny$u(x)$}
		\put (93.5,15) {\tiny$x$}\end{overpic}
	\end{tabular}
	\vspace{-0.4cm}
	\caption{A series of positive solutions along the different branches of the bifurcation diagram of Figure \ref{Fig18}.}
	\label{Fig19}
\end{figure}

\section{Effects of varying $\e\in [0,1]$ in case $\kappa=1$ with $h=0.1$}
\label{sec:6}

In this section, we are going to show the result of some numerical experiments suggesting that, for $a=a_{1,\e}$ with $h=0.1$ and sufficiently negative $\l$, problem \eqref{1.1} exhibits multiplicity of positive solutions not only for $\e=0$, as already discussed in Section \ref{sec:3}, but also for $\e \in (0,1)$ and, in particular, as close to $1$ as desired. Observe that this is in strong contrast with the uniqueness of positive solutions, which holds true for every $\l<\pi^2$ when $a\equiv 1$ (see, e.g., L\'opez-G\'omez, Mu\~{n}oz-Hern\'andez and Zanolin \cite[Proposition 2.1]{LGMHZ}). In addition, the results of our numerical experiments show that the multiplicity in the case $\e \in (0,1)$ can be even higher than for $\e=0$.

\begin{figure}[h!]
	\centering
	\begin{overpic}[scale=0.28,trim = 1cm 5cm 1cm 7cm, clip]{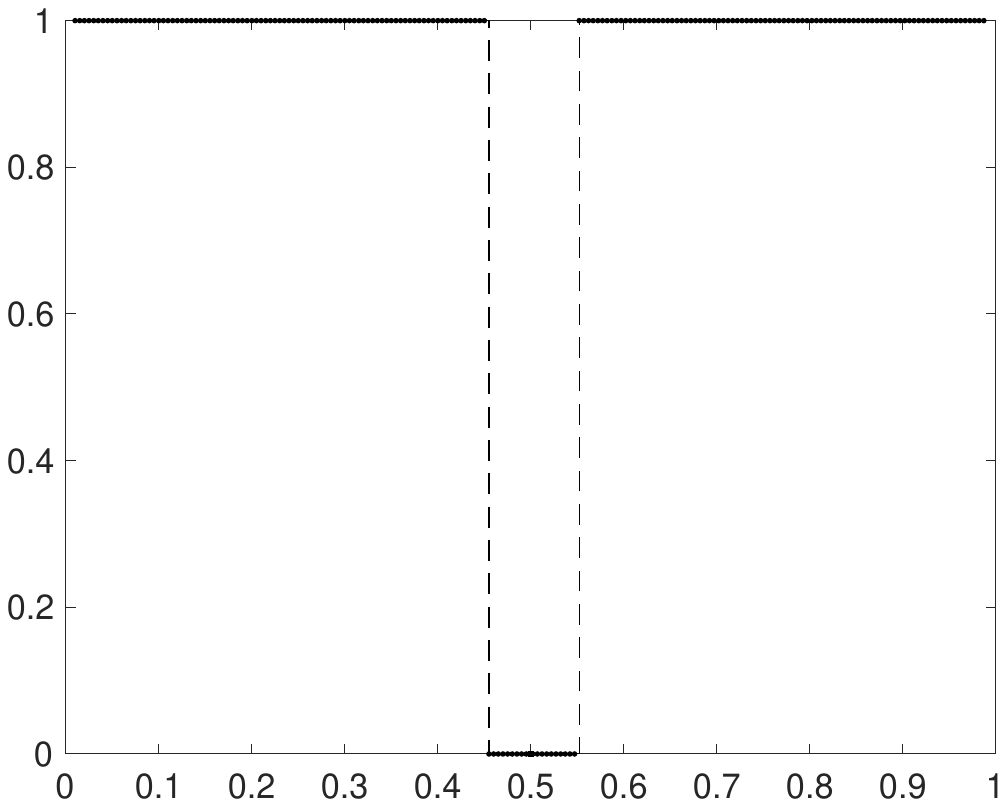} \put (13.5,83.5) {\tiny$a_{1,0}(x)$}
		\put (96.5,15) {\tiny$x$}
	\end{overpic}
		\vspace{-0.8cm}
		\caption{Plot of the weight function $a_{1,0}$ for $h=0.1$.}
		\label{Fig20}
	\end{figure}
	
	Figure \ref{Fig21} illustrates the resulting bifurcation diagram for $\e=0$, which consists of the component $\mathscr{C}_0^+$ bifurcating from $u=0$ at $\l=\pi^2$ and exhibits a secondary bifurcation of subcritical pitchfork type.
	
	\begin{figure}[h!]
		\centering
		\begin{overpic}[scale=0.28,trim = 1cm 5cm 1cm 5cm, clip]{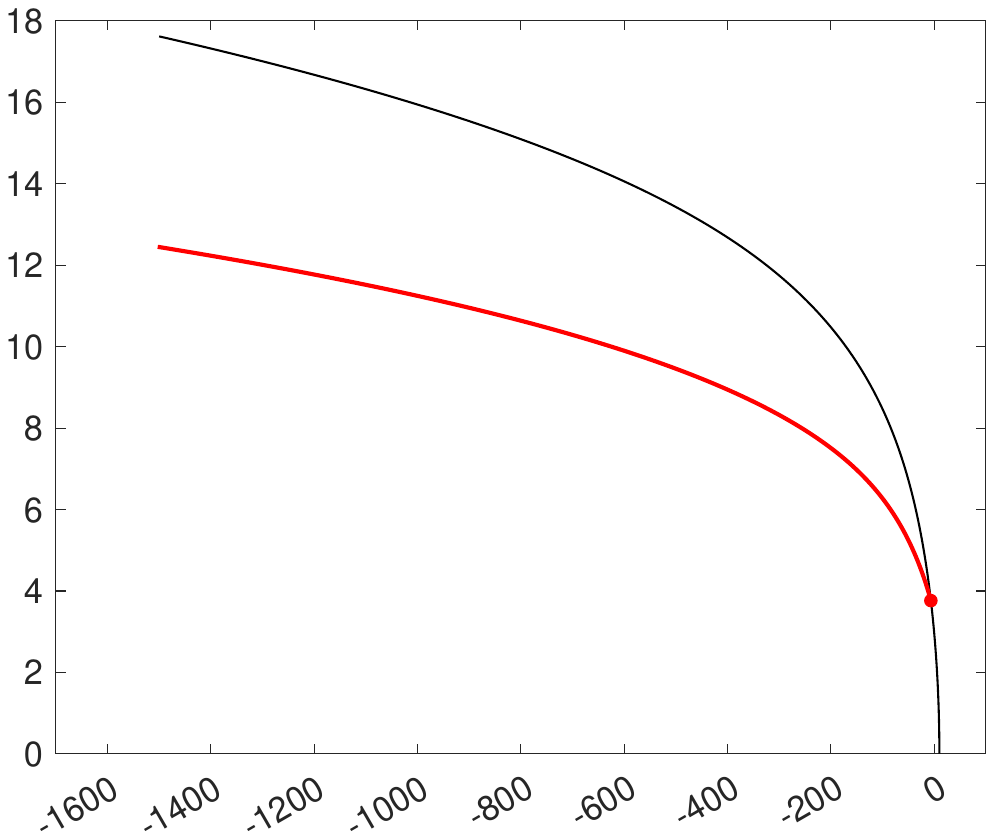}
					\put (12.5,81) {\tiny$\|u\|_2$}
					\put (93,14) {\tiny$\l$}
				\end{overpic}
		\vspace{-0.8cm}
		\caption{Bifurcation diagram for \eqref{1.1} with $a=a_{1,\e}$, $h=0.1$ and $\e=0$.}
		\label{Fig21}
	\end{figure}

	Figure \ref{Fig22} shows a series of positive solutions along each of the three branches of the bifurcation diagram. Symmetric solutions along the primary branch are plotted in black in Figure \ref{Fig22}, for $\l>0$ (upper left, with one peak) and for $\l<0$ (upper right, with two peaks). Asymmetric solutions along the secondary branches are plotted in red and have a single peak. Remember that these secondary branches overlap in the bifurcation diagram of Figure \ref{Fig21} since the $L^2$-norm of these solutions is the same for every $\l$ where they are defined, since they are symmetric about 0.5 from one another. These results are similar to those presented in Section \ref{sec:3}, where the weight function has also one well, i.e. $\k=1$, simply with different values of its amplitude $h$.
	
	\begin{figure}[ht!]
		\centering
		\begin{overpic}[scale=0.28,trim = 1cm 5cm 1cm 8cm, clip]{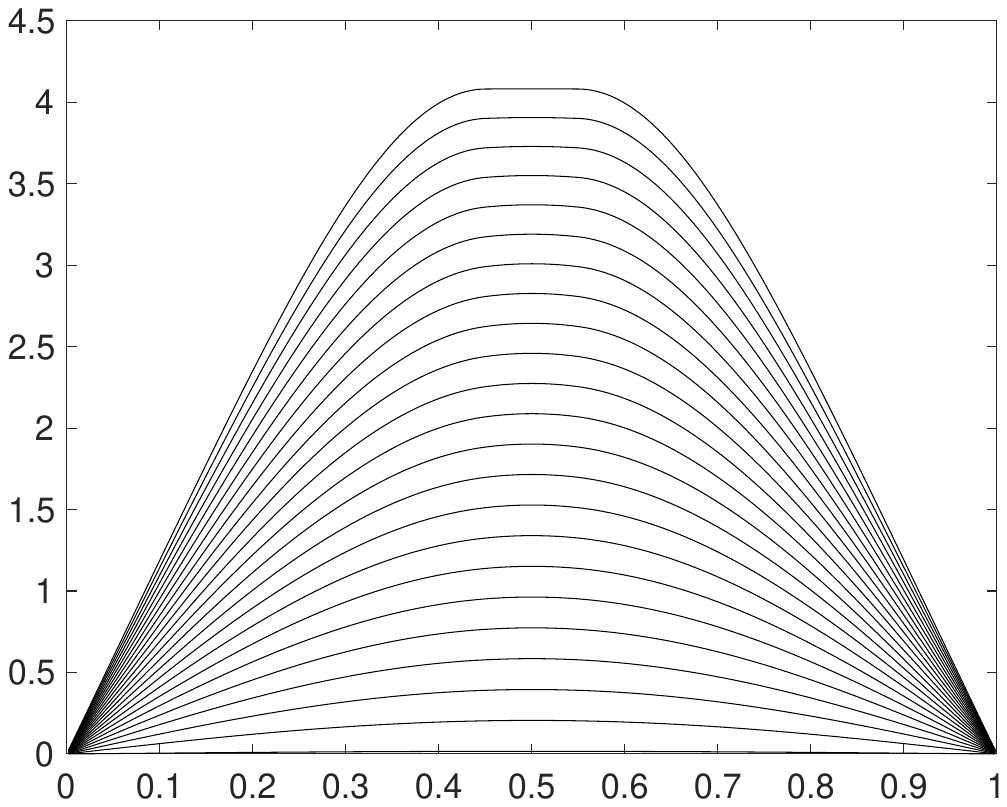} \put (12,83.5) {\tiny$u(x)$}
					\put (96.5,15) {\tiny$x$}\end{overpic} \begin{overpic}[scale=0.28,trim = 1cm 5cm 1cm 8cm, clip]{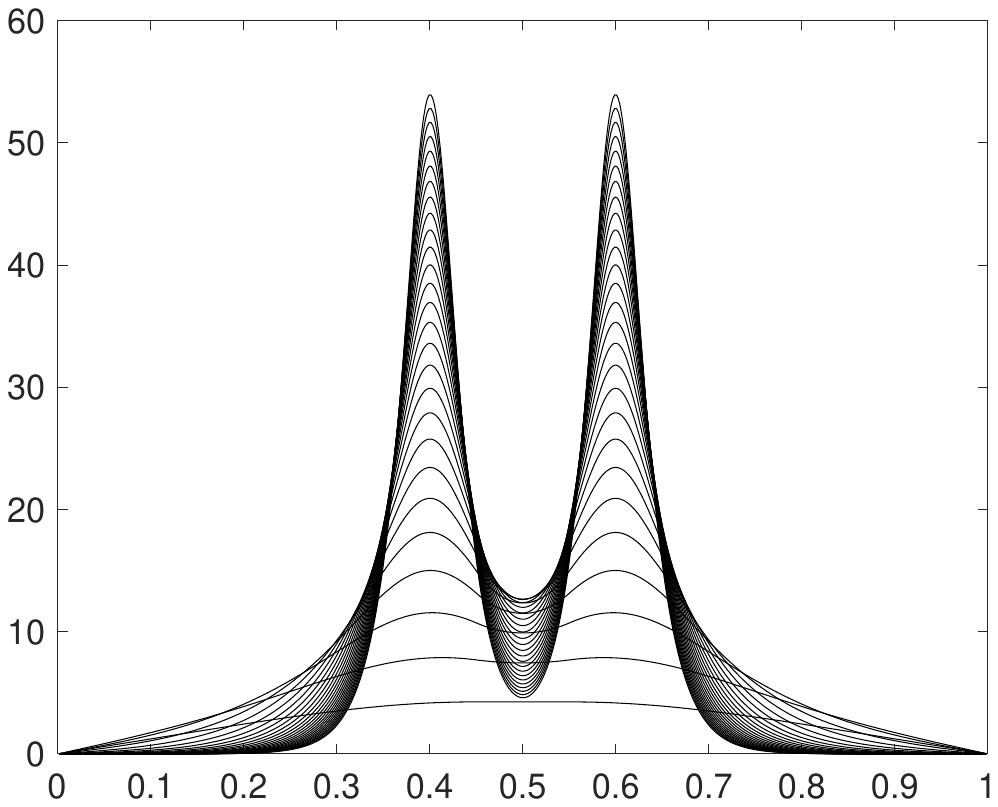} \put (12,83.5) {\tiny$u(x)$}
					\put (96.5,15) {\tiny$x$}\end{overpic} \\[-2.5em]
		\begin{overpic}[scale=0.28,trim = 1cm 5cm 1cm 5cm, clip]{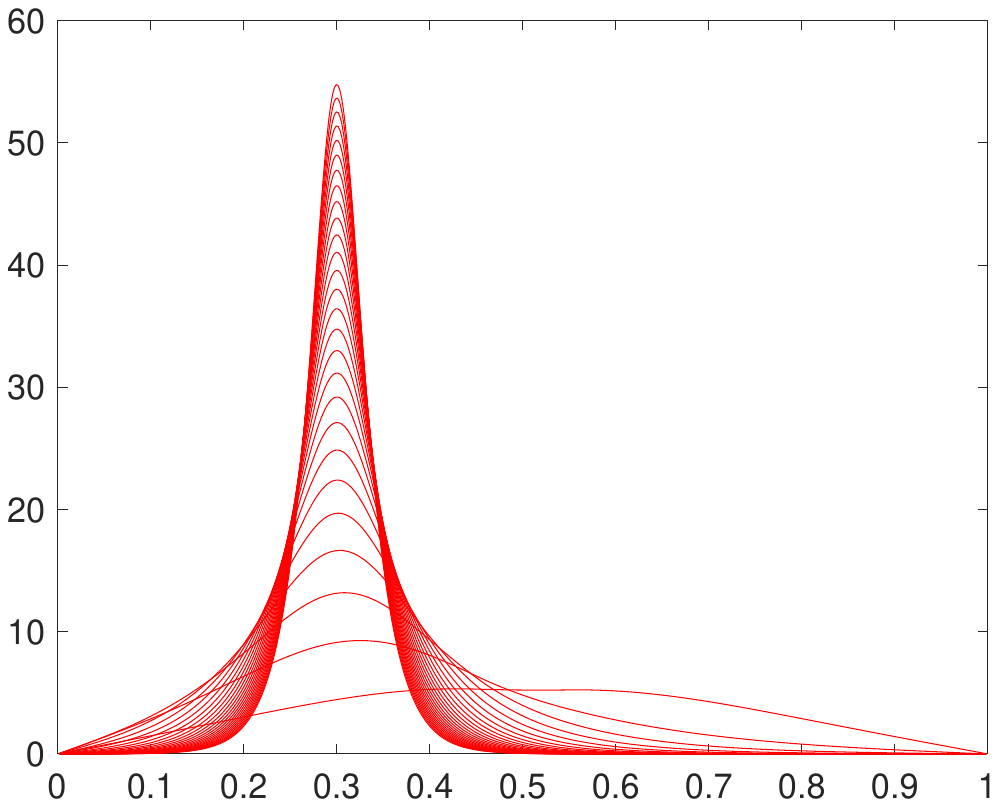} \put (12,80.5) {\tiny$u(x)$}
					\put (93.5,14.5) {\tiny$x$}\end{overpic}
				 \begin{overpic}[scale=0.28,trim = 1cm 5cm 1cm 5cm, clip]{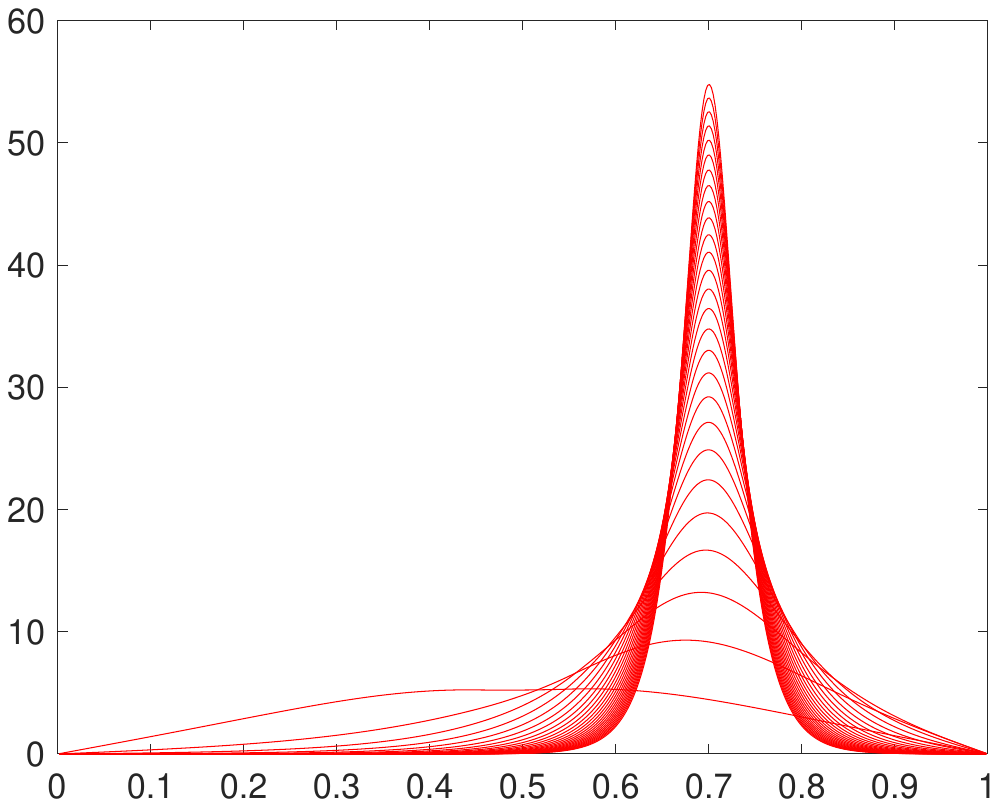} \put (12,80.5) {\tiny$u(x)$}
					\put (93.5,14.5) {\tiny$x$}\end{overpic}
		\vspace{-0.4cm}
		\caption{A series of plots of positive solutions along the three branches corresponding to the bifurcation diagram of Figure \ref{Fig21}. Profiles on the symmetric branch have been plotted for $\l>0$ (upper left) and for $\l<0$ (upper right). Asymmetric profiles lie on the secondary branches.}
		\label{Fig22}
	\end{figure}

Figure \ref{Fig23} shows the bifurcation diagram for $\e=0.3$. As in the case $\e=0$, the bifurcation diagram consists of the component $\mathscr{C}_0^+$, with a secondary bifurcation of subcritical pitchfork type. Nevertheless, an additional component of positive solutions arises: an isola with a subcritical turning point at $\l_t=-1111.65254$. In Section \ref{sec:7}, together with some theoretical and other practical aspects of our numerical simulations, we will explain how this isolated component has been detected. Thus, according to our simulations, in this case problem \eqref{1.1} admits 5 positive solutions for every $\l<\l_t$.

\begin{figure}[ht!]
	\centering
	\begin{overpic}[scale=0.3,trim = 1cm 3cm 1cm 7cm, clip]{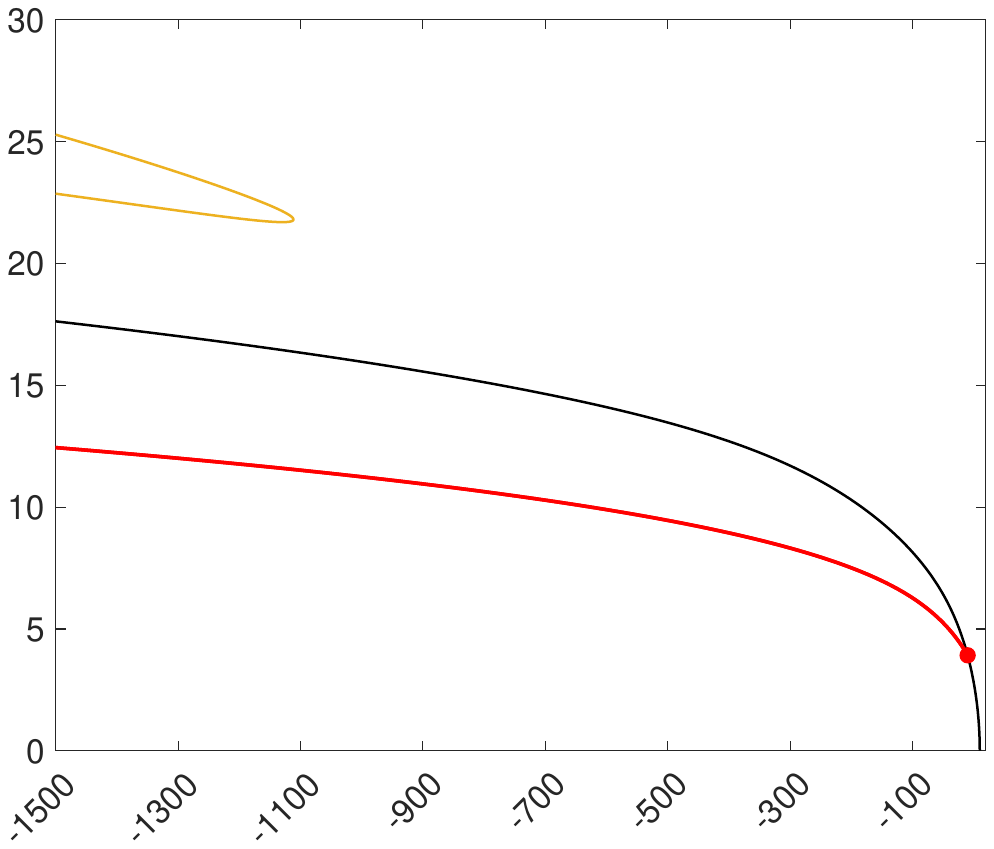}
		\put (12.5,91) {\tiny$\|u\|_2$}
		\put (94,25) {\tiny$\l$}
	\end{overpic}
	\vspace{-1.0cm}
	\caption{Bifurcation diagram for \eqref{1.1} with $a=a_{1,\e}$, $h=0.1$ and $\e=0.3$.}
	\label{Fig23}
\end{figure}

Figure \ref{Fig24} shows  a series of positive solutions along each of the branches of the bifurcation diagram. As above, the profiles plotted in black in Figure \ref{Fig24} correspond to a series of symmetric solutions along the primary branch for $\l>0$ (left) and for $\l<0$ (right), and solutions with a single peak, plotted in red in the second row, belong to each of the two secondary branches, which are overlapped in Figure \ref{Fig23}. The profiles on the isola (plotted in orange, according to the color used in the bifurcation diagram) are symmetric and exhibit a single peak.

\begin{figure}[ht!]
	\centering
	\begin{overpic}[scale=0.28,trim = 1cm 5cm 1cm 7cm, clip]{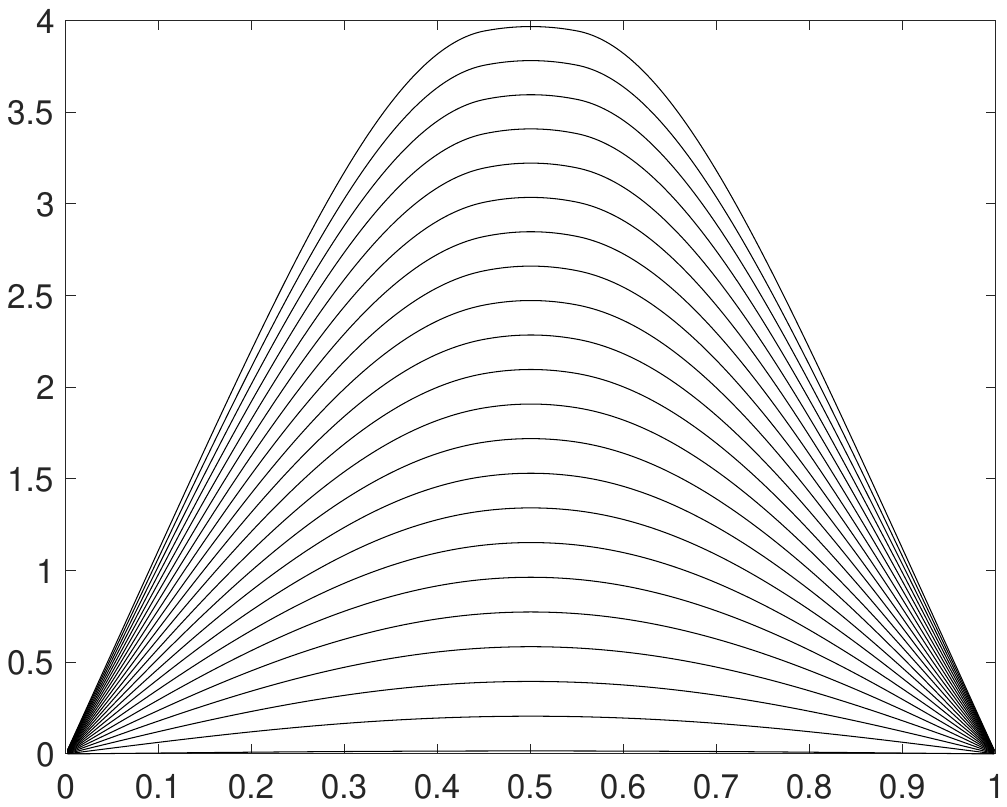} \put (12.3,83.6) {\tiny$u(x)$}
		\put (96.5,15) {\tiny$x$}\end{overpic} \begin{overpic}[scale=0.28,trim = 1cm 5cm 1cm 7cm, clip]{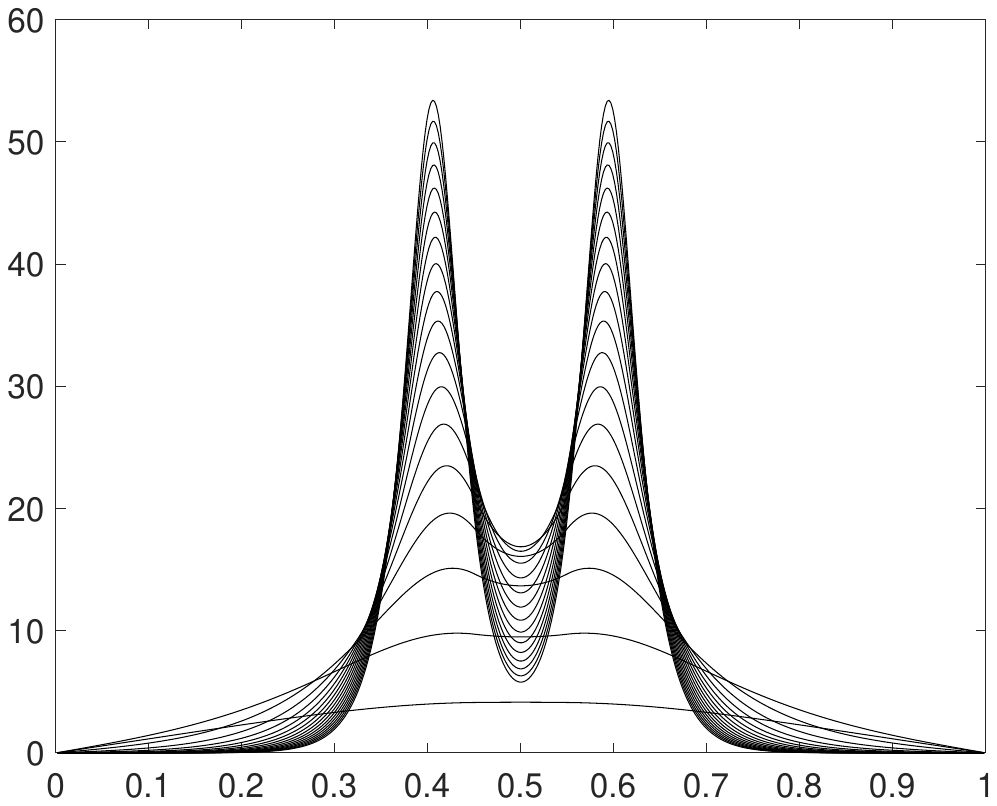} \put (12.3,83.6) {\tiny$u(x)$}
		\put (96.5,15) {\tiny$x$}\end{overpic}\\[-2.5em]
	\begin{overpic}[scale=0.28,trim = 1cm 5cm 1cm 5cm, clip]{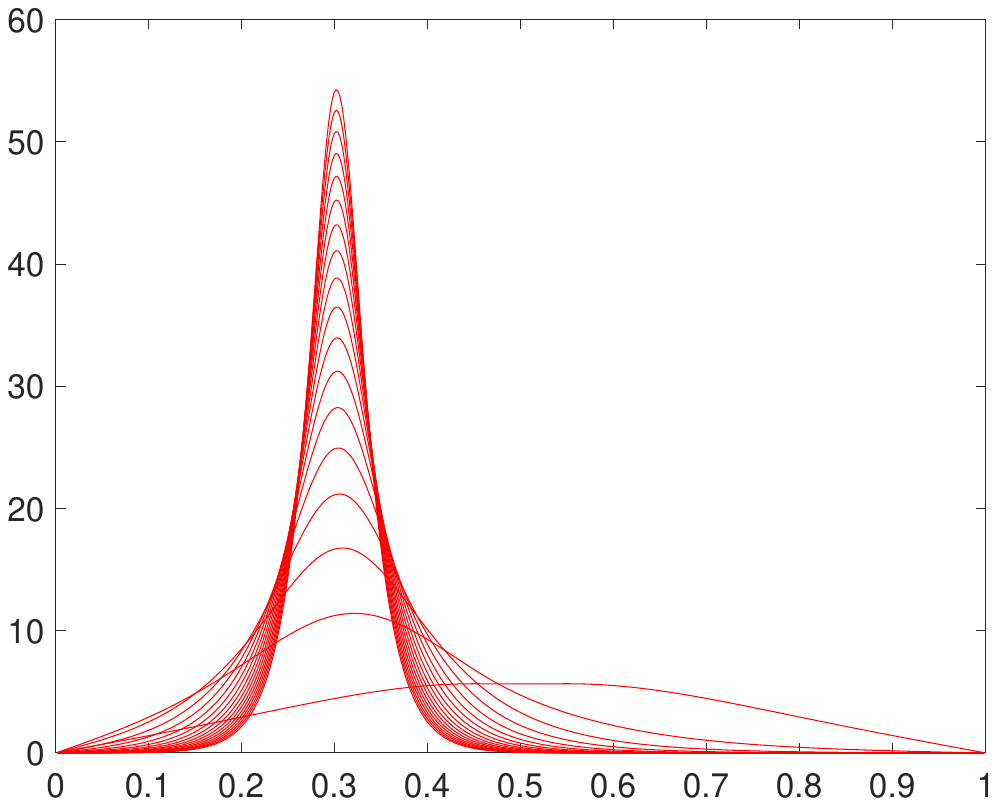} \put (12,80.5) {\tiny$u(x)$}
		\put (93,15) {\tiny$x$}\end{overpic} \begin{overpic}[scale=0.28,trim = 1cm 5cm 1cm 5cm, clip]{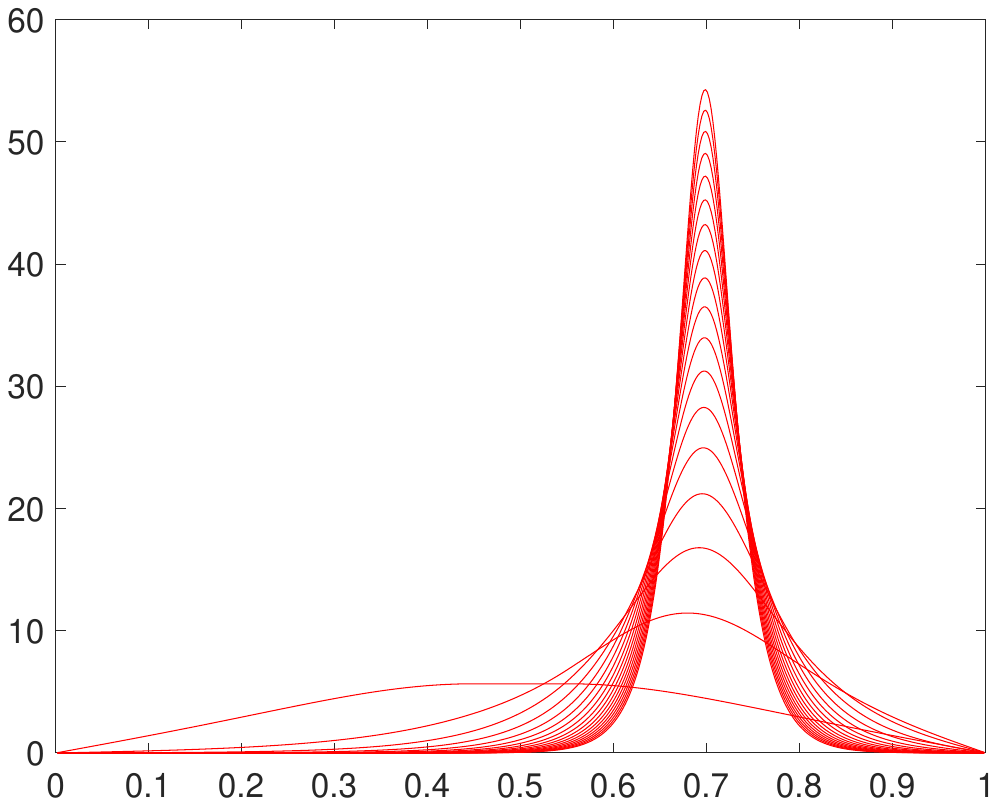} \put (12,80.5) {\tiny$u(x)$}
		\put (93,15) {\tiny$x$}\end{overpic}\\[-2.5em]
	\begin{overpic}[scale=0.28,trim = 1cm 5cm 1cm 5cm, clip]{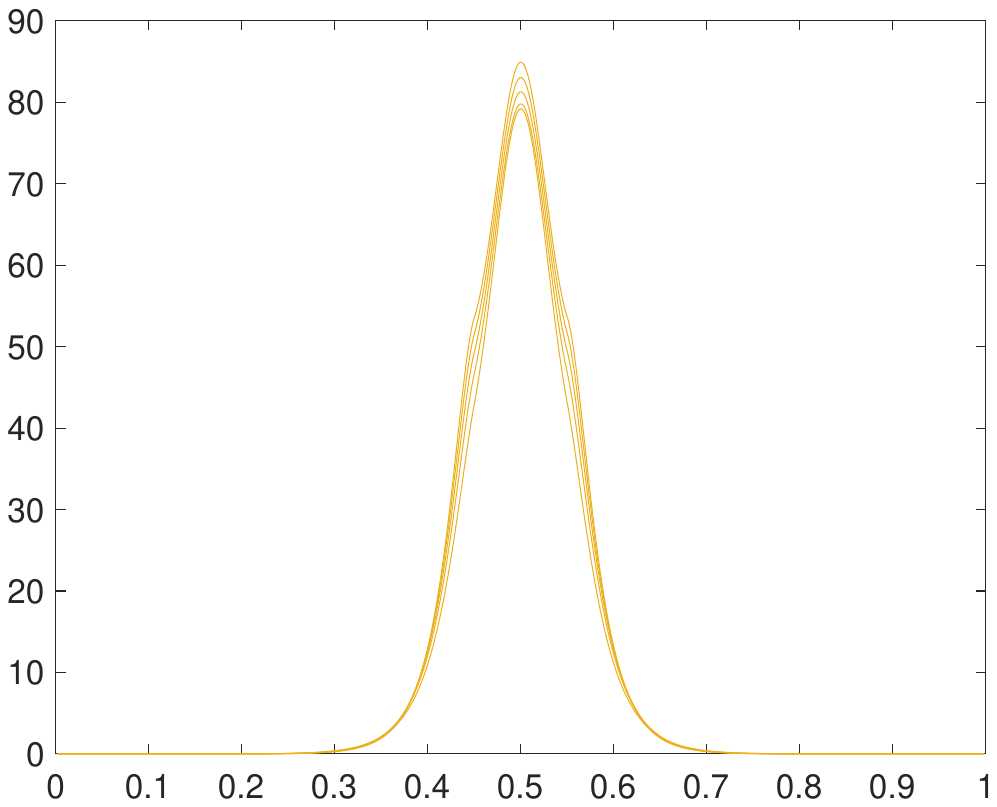} \put (12,80) {\tiny$u(x)$}
		\put (93,15) {\tiny$x$}\end{overpic} \begin{overpic}[scale=0.28,trim = 1cm 5cm 1cm 5cm, clip]{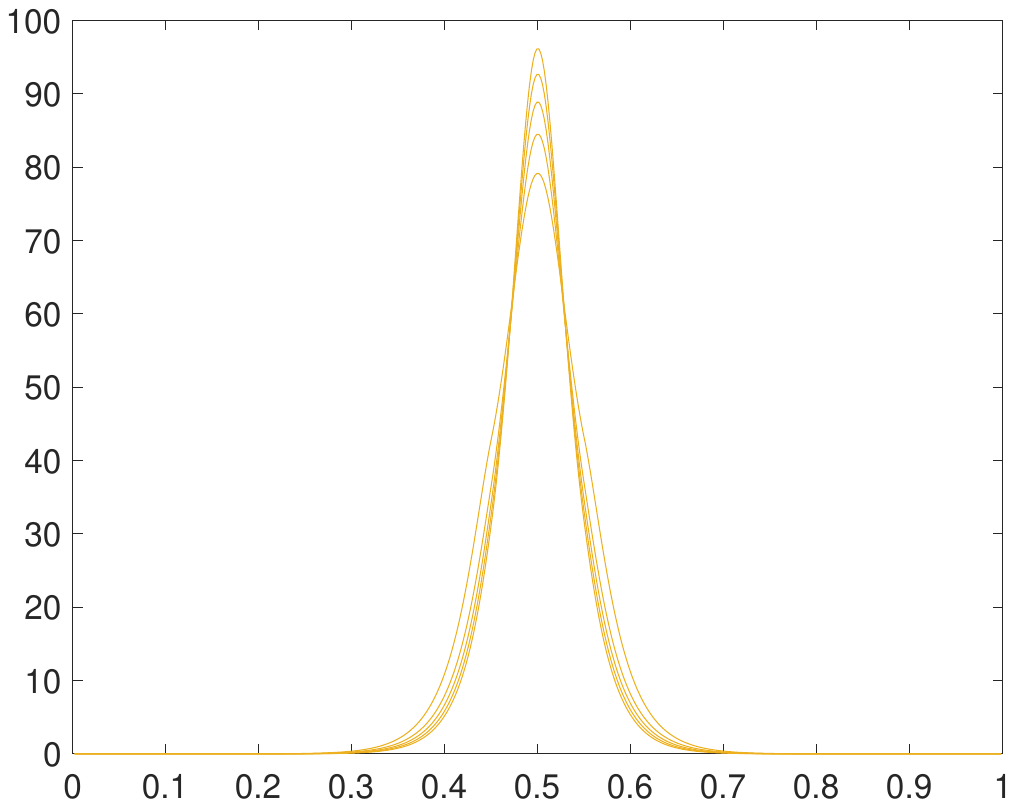} \put (12,80) {\tiny$u(x)$}
		\put (93.5,15) {\tiny$x$}\end{overpic}
	\vspace{-0.4cm}
	\caption{A series of plots of positive solutions along the branches corresponding to the bifurcation diagram of Figure \ref{Fig23}. Symmetric profiles on the principal branch have been plotted for $\l>0$ (upper left) and for $\l<0$ (upper right). Asymmetric profiles lie on the secondary branches. Symmetric profiles on the isola correspond to its upper (bottom left plot) and to its lower (bottom right plot) part.}
	\label{Fig24}
\end{figure}

As $\e$ increases up to $\e=0.5$, the qualitative behavior of the bifurcation diagram and of the profiles of the corresponding solutions does not change significantly. The only difference is that the isola becomes closer and closer to the symmetric branch of $\mathscr{C}_0^+$, as illustrated in Figure \ref{Fig25}.

\begin{figure}[ht!]
	\centering
	\begin{overpic}[scale=0.27,trim = 1cm 3cm 1cm 7cm, clip]{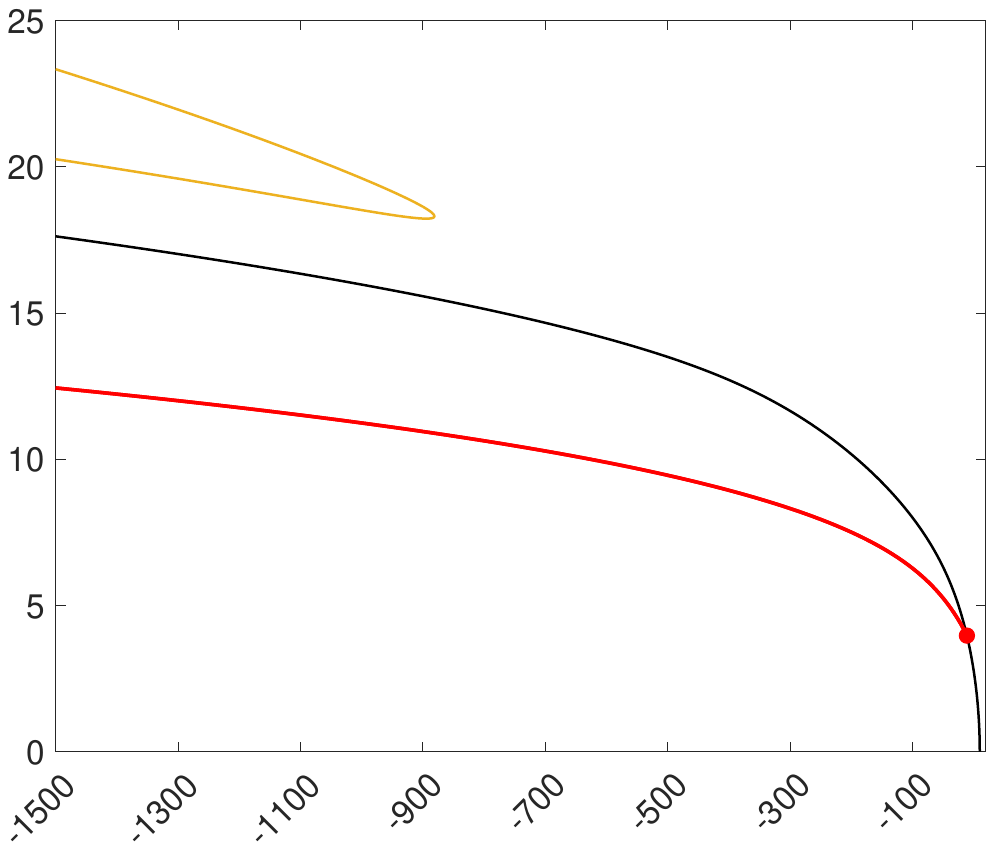}
		\put (12.5,91) {\tiny$\|u\|_2$}
		\put (94,25) {\tiny$\l$}
	\end{overpic} \quad
	\begin{overpic}[scale=0.27,trim = 1cm 3cm 1cm 7cm, clip]{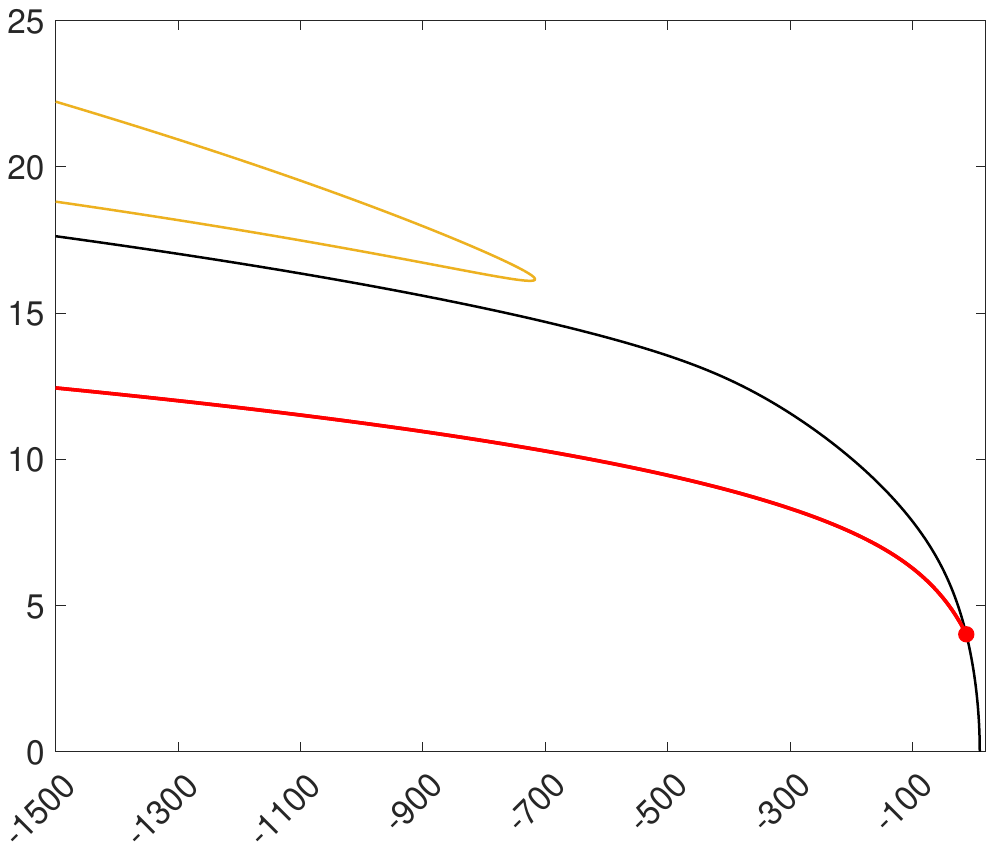}
		\put (12.5,91) {\tiny$\|u\|_2$}
		\put (94,25) {\tiny$\l$}
	\end{overpic} \quad
\begin{overpic}[scale=0.27,trim = 1cm 3cm 1cm 7cm, clip]{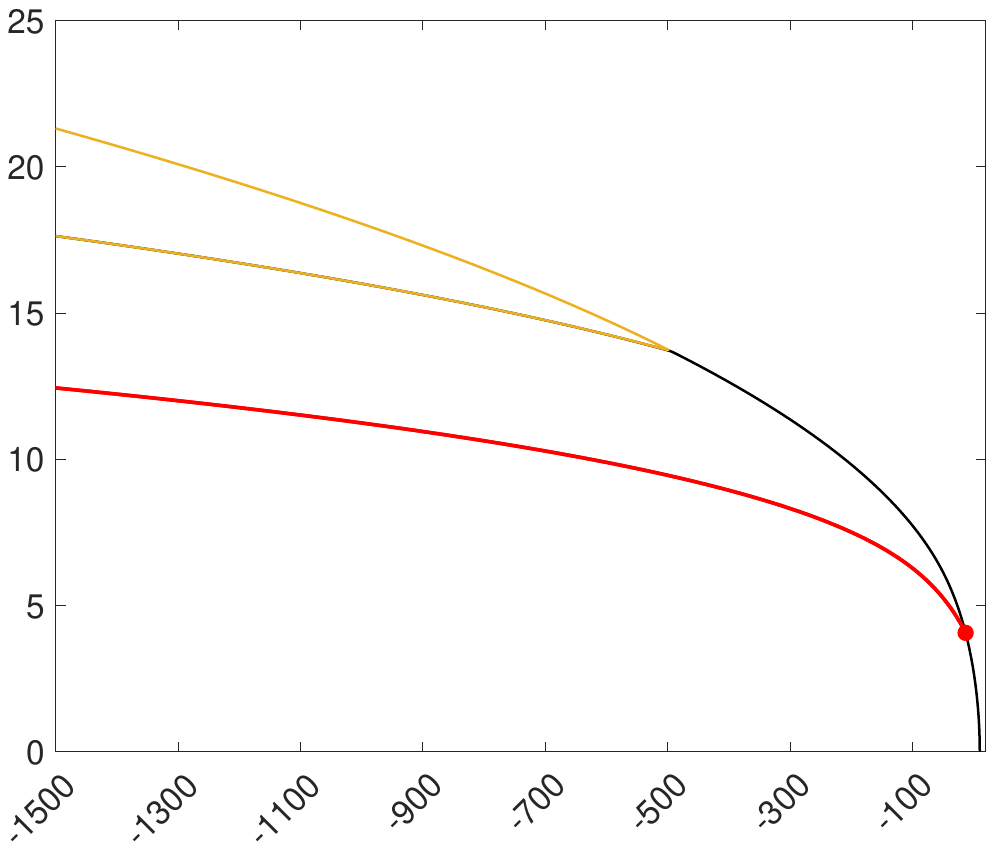}
	\put (12.5,91) {\tiny$\|u\|_2$}
	\put (94,25) {\tiny$\l$}
\end{overpic}
	\vspace{-1cm}
	\caption{Bifurcation diagram for \eqref{1.1} with $a=a_{1,\e}$, $h=0.1$ and $\e=0.38$ (left), $\e=0.44$ (center) and $\e=0.50$ (right).}
	\label{Fig25}
\end{figure}

Figure \ref{Fig26} shows the profiles of the solutions for $\e=0.5$, corresponding to the right diagram of Figure \ref{Fig25}. One can observe how the solutions on the upper part of the isola, which have been represented in the bottom left plot of Figure \ref{Fig26}, become flatter and flatter in $a^{-1}(\e)$ as $\l\da -\infty$.

\begin{figure}[ht!]
	\centering
	\begin{overpic}[scale=0.28,trim = 1cm 5cm 1cm 8cm, clip]{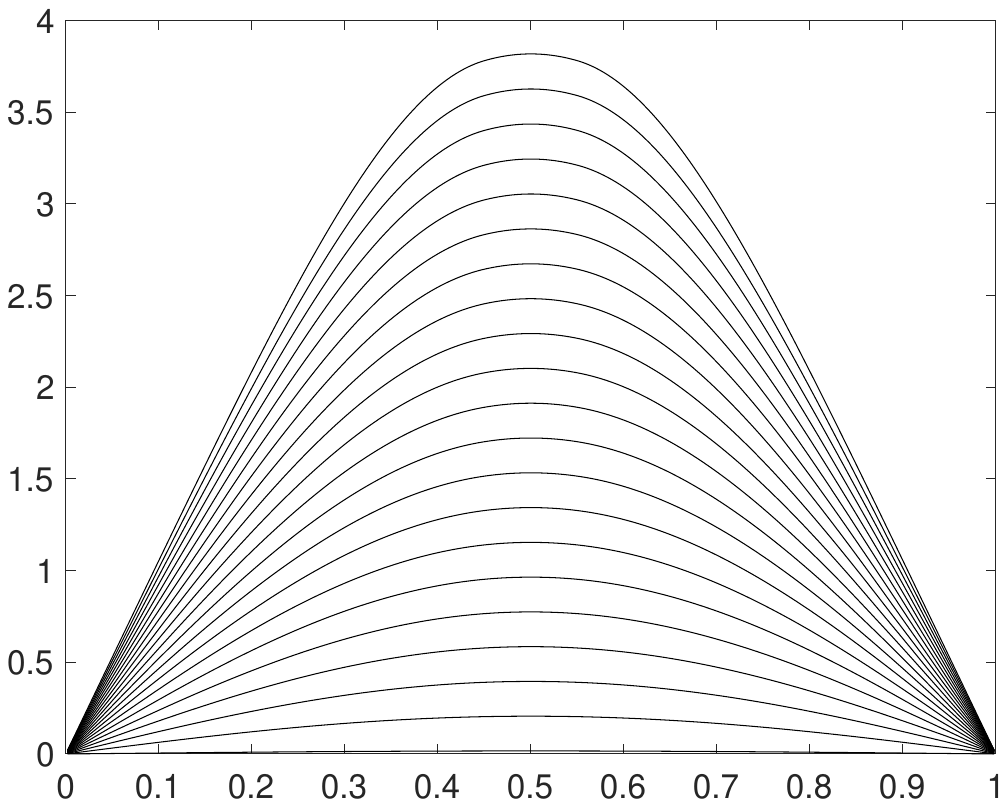} \put (12.3,83.6) {\tiny$u(x)$}
		\put (96.7,15) {\tiny$x$}\end{overpic} \begin{overpic}[scale=0.28,trim = 1cm 5cm 1cm 8cm, clip]{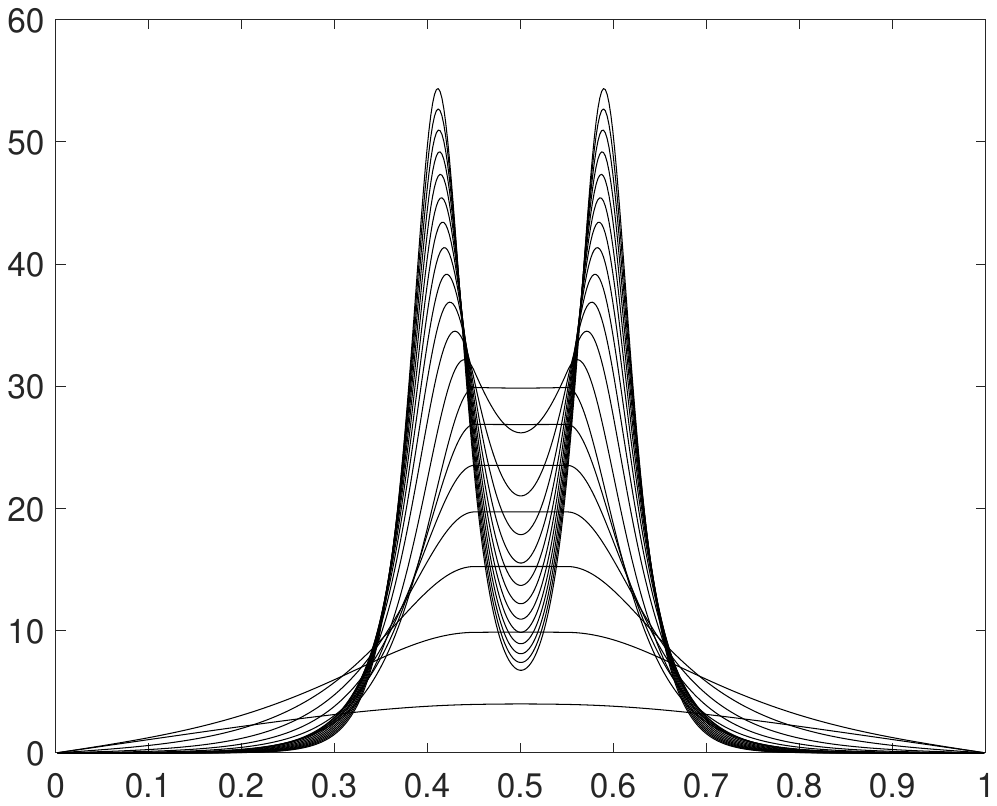} \put (12.3,83.6) {\tiny$u(x)$}
		\put (96.7,15) {\tiny$x$}\end{overpic}\\[-2.5em]
	\begin{overpic}[scale=0.28,trim = 1cm 5cm 1cm 5cm, clip]{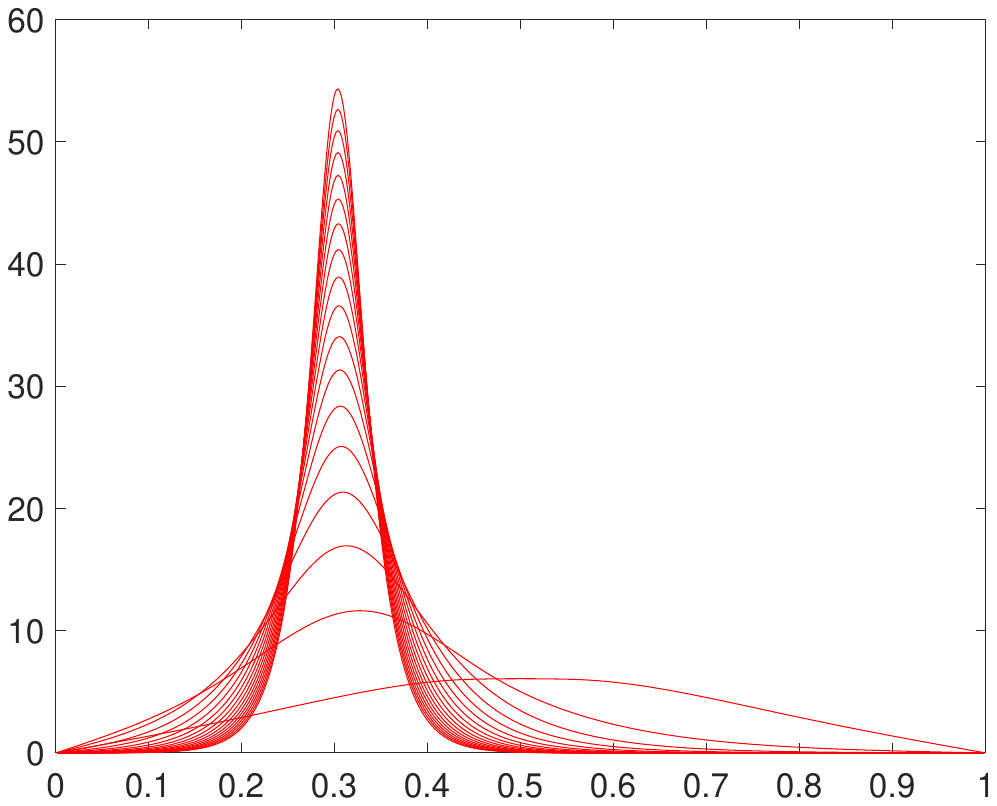} \put (12,80.5) {\tiny$u(x)$}
		\put (93,15) {\tiny$x$}\end{overpic} \begin{overpic}[scale=0.28,trim = 1cm 5cm 1cm 5cm, clip]{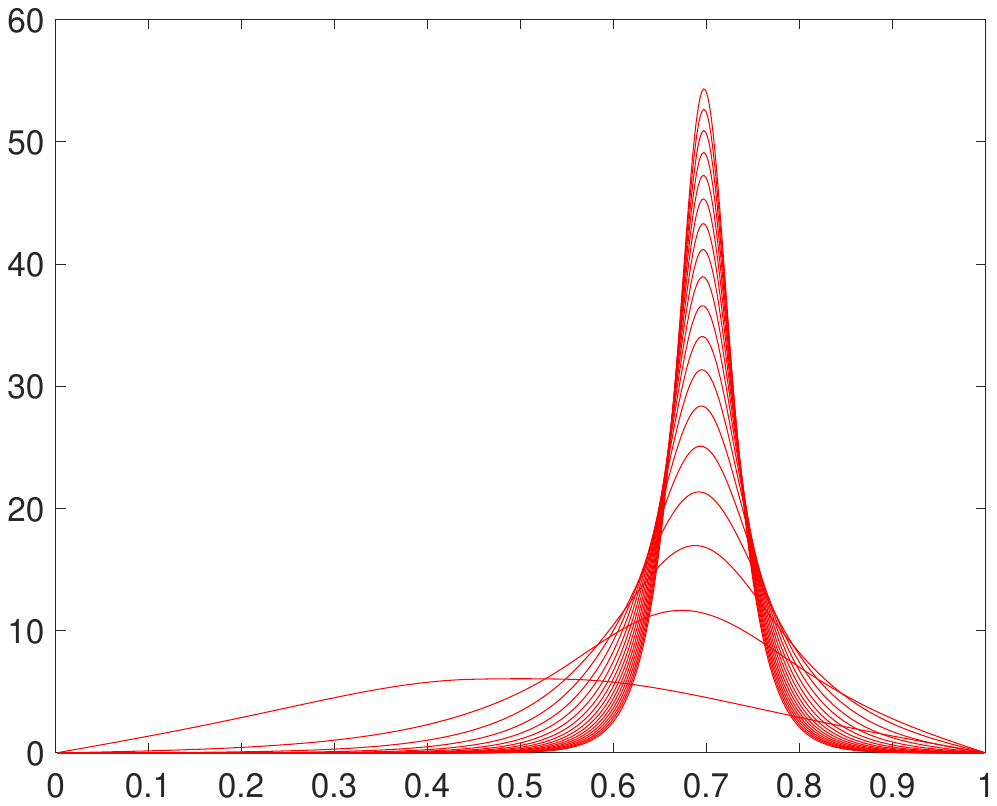} \put (12,80.5) {\tiny$u(x)$}
		\put (93,15) {\tiny$x$}\end{overpic}\\[-2.5em]
	\begin{overpic}[scale=0.28,trim = 1cm 5cm 1cm 5cm, clip]{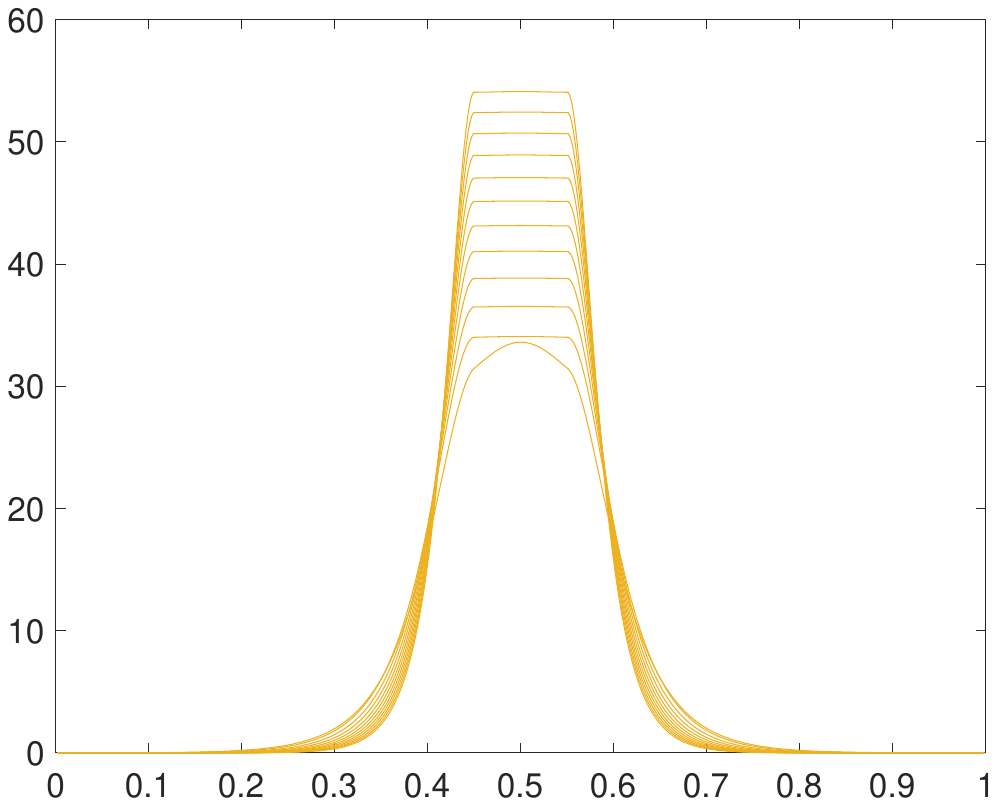} \put (12,80) {\tiny$u(x)$}
		\put (93,15) {\tiny$x$}\end{overpic} \begin{overpic}[scale=0.28,trim = 1cm 5cm 1cm 5cm, clip]{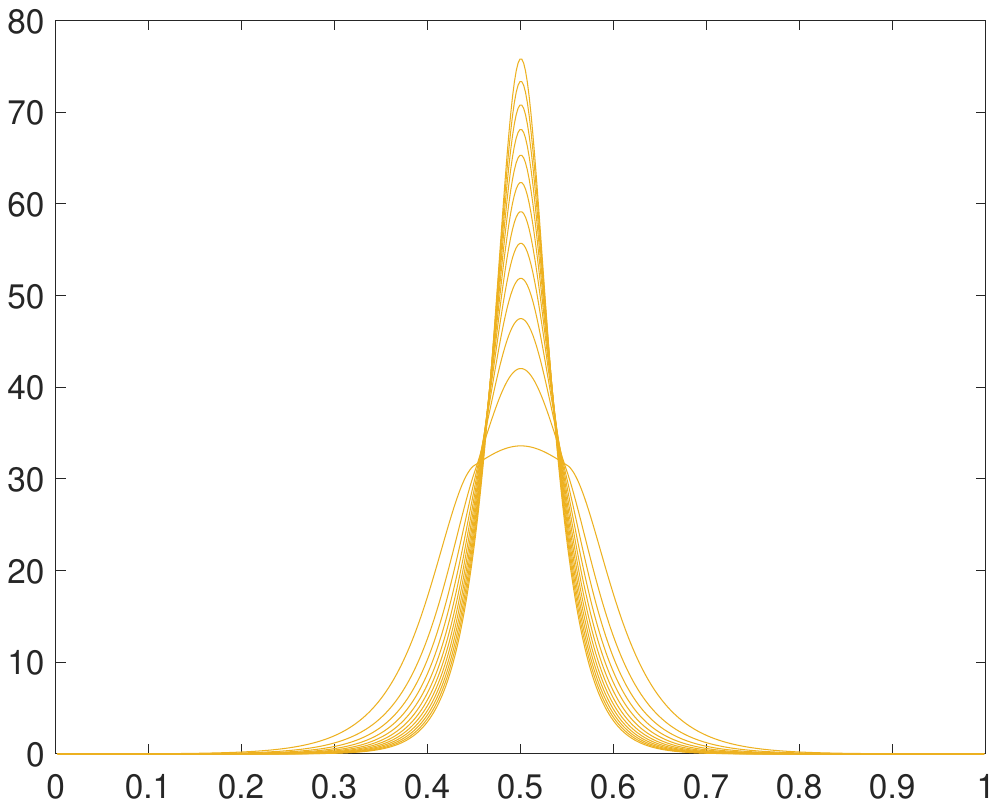} \put (12,80) {\tiny$u(x)$}
		\put (93.5,15) {\tiny$x$}\end{overpic}
	\vspace{-0.7cm}
	\caption{A series of plots of positive solutions along the branches corresponding to the right bifurcation diagram of Figure \ref{Fig25}. Symmetric profiles on the principal branch have been plotted for $\l>0$ (upper left) and for $\l<0$ (upper right). Asymmetric profiles lie on the secondary branches. Symmetric profiles on the isola correspond to its lower (bottom left plot) and to its upper (bottom right plot) part.}
	\label{Fig26}
\end{figure}

When $\e$ increases further and passes a certain critical value $\e^*\in(0.50,0.51)$, the structure of the solution set undergoes a deep change. Although this is not perceptible from the bifurcation diagrams, due to the choice of the discrete $L^2$-norm on the vertical axis (compare the diagram in the right plot of Figure \ref{Fig25}, corresponding to $\e=0.50$ with the diagram of Figure \ref{Fig27}, corresponding to $\e=0.51$), the change becomes apparent from the shape of the profiles of the positive solutions on the symmetric branch of the principal component $\mathscr{C}_0^+$. Indeed, for $\e=0.50$ such profiles have two peaks for all $\l<0$, while for $\e=0.51$ they exhibit only one peak. The opposite behavior, instead, can be observed for the profiles on the isola: the solutions lying therein have one peak for $\e=0.50$, while they have two peaks for $\e=0.51$ (cf., Figures \ref{Fig26} and \ref{Fig28}, corresponding to $\e=0.50$ and $\e=0.51$, respectively).

\begin{figure}[ht!]
	\centering
	\begin{overpic}[scale=0.3,trim = 1cm 6.5cm 1cm 7.5cm, clip]{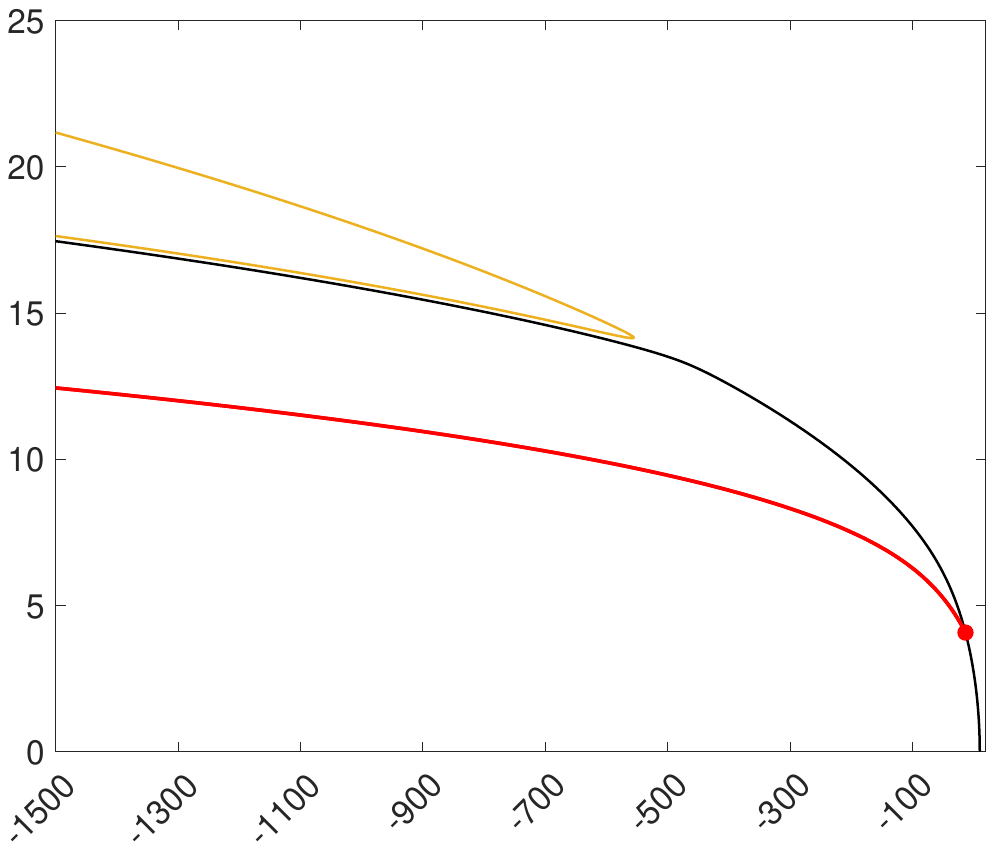}
		\put (12.5,75.5) {\tiny$\|u\|_2$}
		\put (97.5,7.5) {\tiny$\l$}
	\end{overpic}
	\caption{Bifurcation diagram for \eqref{1.1} with $a=a_{1,\e}$, $h=0.1$ and $\e=0.51$.}
	\label{Fig27}
\end{figure}

\begin{figure}[ht!]
	\centering
	\begin{overpic}[scale=0.28,trim = 1cm 5cm 1cm 7cm, clip]{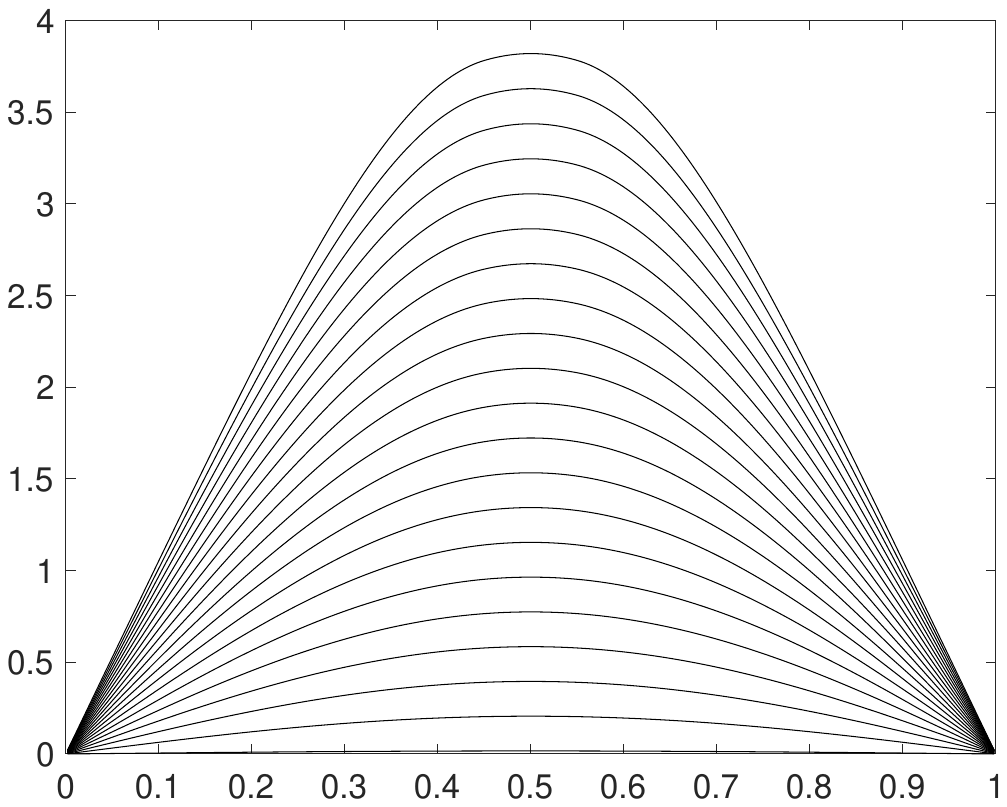} \put (12.3,83.6) {\tiny$u(x)$}
		\put (96.7,15) {\tiny$x$}\end{overpic} \begin{overpic}[scale=0.28,trim = 1cm 5cm 1cm 7cm, clip]{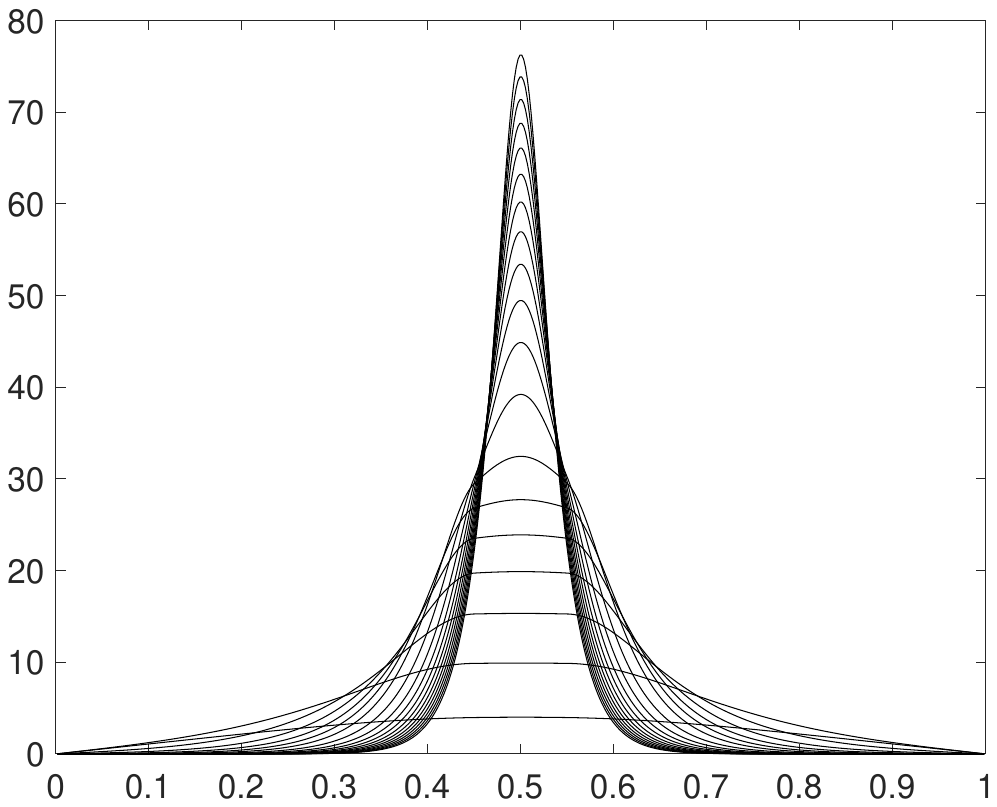} \put (12.3,83.6) {\tiny$u(x)$}
		\put (96.7,15) {\tiny$x$}\end{overpic}\\[-2.5em]
	\begin{overpic}[scale=0.28,trim = 1cm 5cm 1cm 5cm, clip]{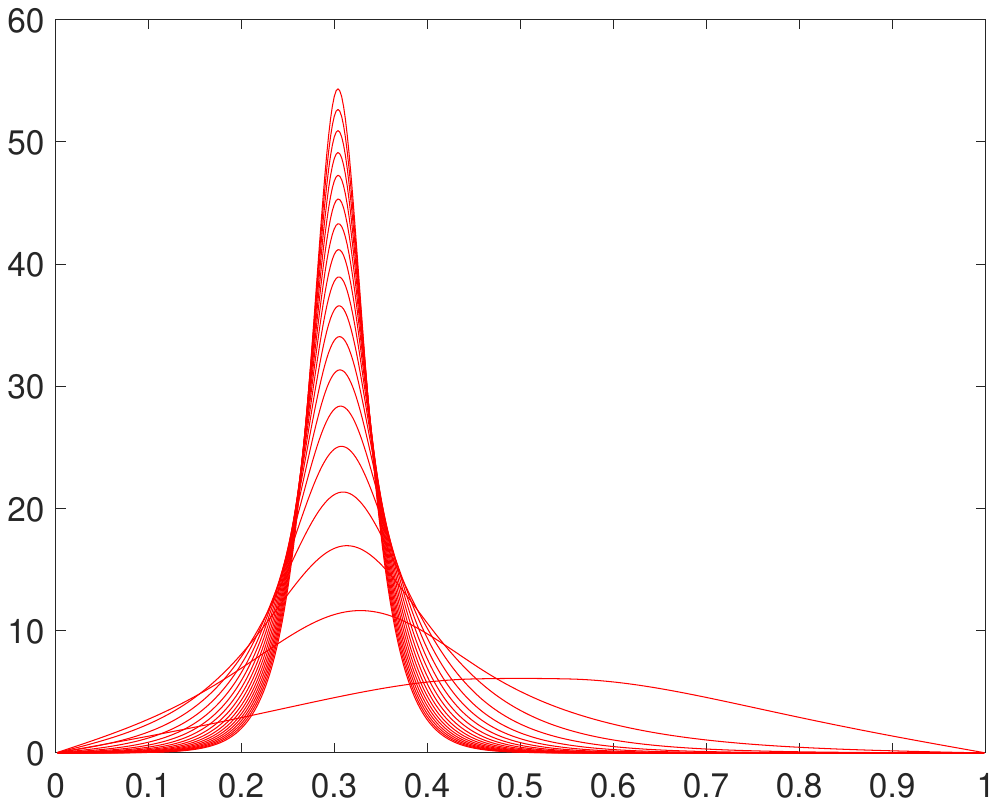} \put (12,80.5) {\tiny$u(x)$}
		\put (93,15) {\tiny$x$}\end{overpic} \begin{overpic}[scale=0.28,trim = 1cm 5cm 1cm 5cm, clip]{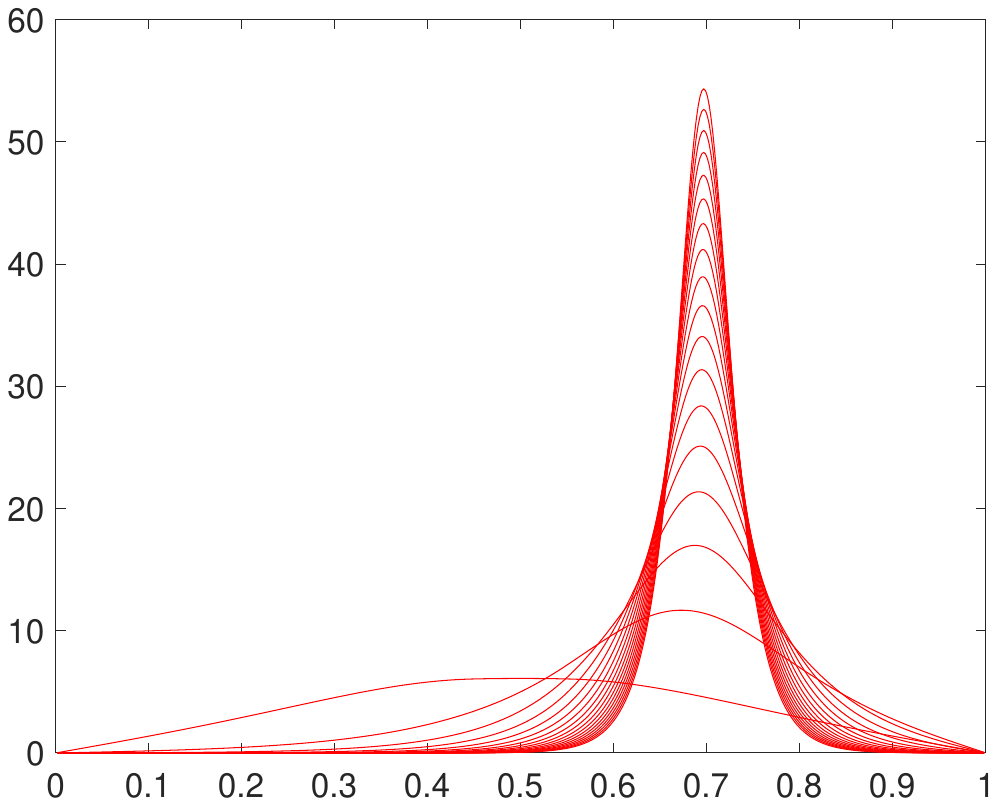} \put (12,80.5) {\tiny$u(x)$}
		\put (93,15) {\tiny$x$}\end{overpic}\\[-2.5em]
	\begin{overpic}[scale=0.28,trim = 1cm 7cm 1cm 5cm, clip]{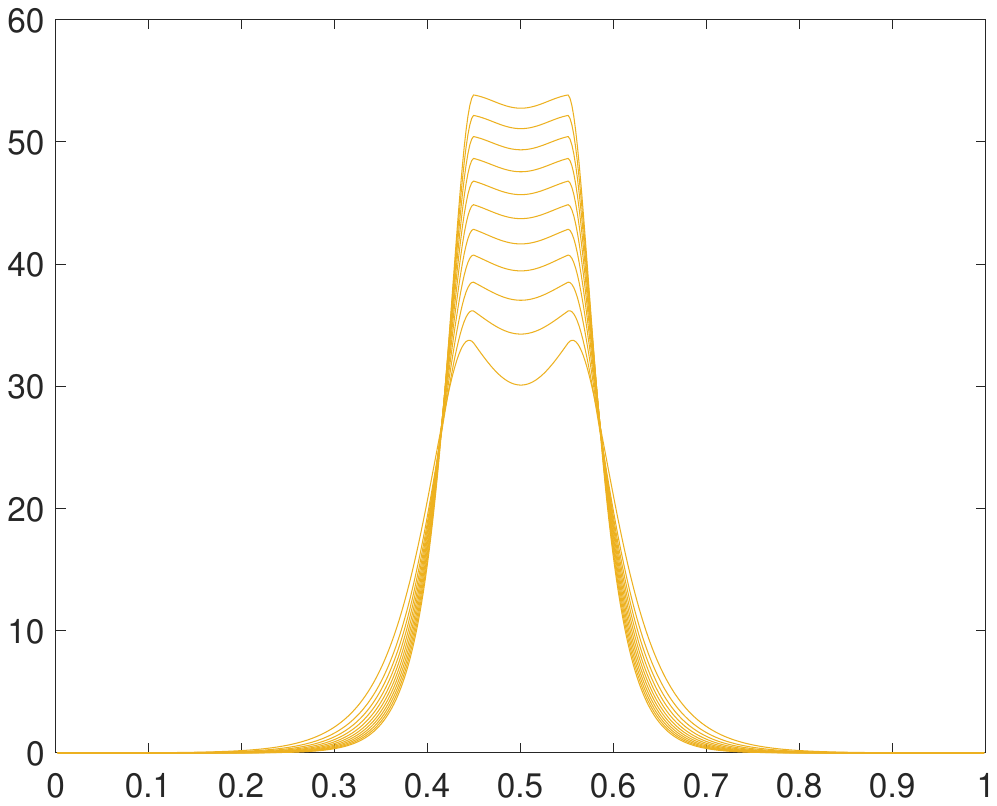} \put (12.5,73.2) {\tiny$u(x)$}
		\put (96.5,5) {\tiny$x$}\end{overpic} \begin{overpic}[scale=0.28,trim = 1cm 7cm 1cm 5cm, clip]{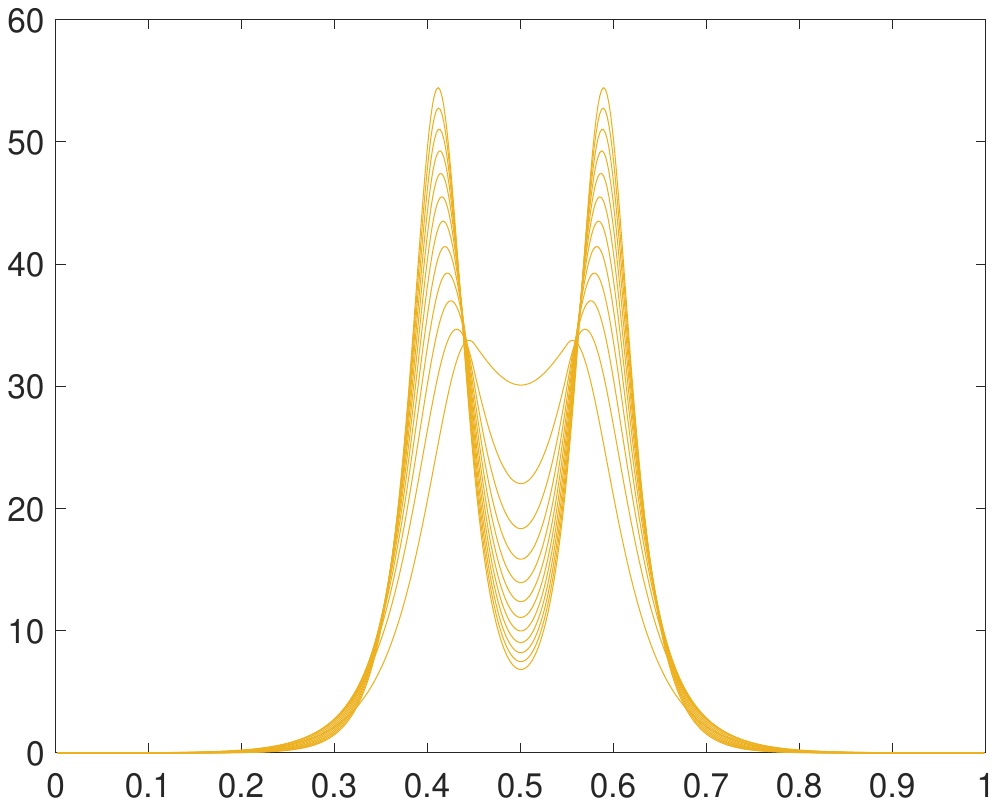} \put (12.5,73.2) {\tiny$u(x)$}
		\put (96.5,5) {\tiny$x$}\end{overpic}
	\caption{A series of plots of positive solutions along the branches corresponding to the bifurcation diagram of Figure \ref{Fig27}. Symmetric profiles on the principal branch have been plotted for $\l>0$ (upper left) and for $\l<0$ (upper right). Asymmetric profiles lie on the secondary branches. Symmetric profiles on the isola correspond to its lower (bottom left plot) and to its upper (bottom right plot) part.}
	\label{Fig28}
\end{figure}

Thus, we observe that a \emph{recombination} of the symmetric solutions on the branches of the bifurcation diagram occurs as $\e$ crosses $\e^*$. According to the simulations we have performed, once such a recombination has occurred, the same structure described for $\e=0.51$ is maintained for all further values of $\e\in(\e^*,1)$. The only difference is that, as $\e$ separates from $\e^*$ and approaches $1$, the isola progressively moves away from the branch of symmetric solutions on $\mathscr{C}_0^+$. This tendencies have been illustrated in Figure \ref{Fig29}, which represents the bifurcation diagrams for $\e=0.6$ (left) $\e=0.7$ (right), and in Table \ref{Tab2}, which collects the values of $\l_t$, the value of $\l$ corresponding to the the turning point on the isola, for a series of $\e\in(0,1)$.

\begin{figure}[ht!]
	\centering
	\begin{overpic}[scale=0.3,trim = 1cm 6.5cm 1cm 7.5cm, clip]{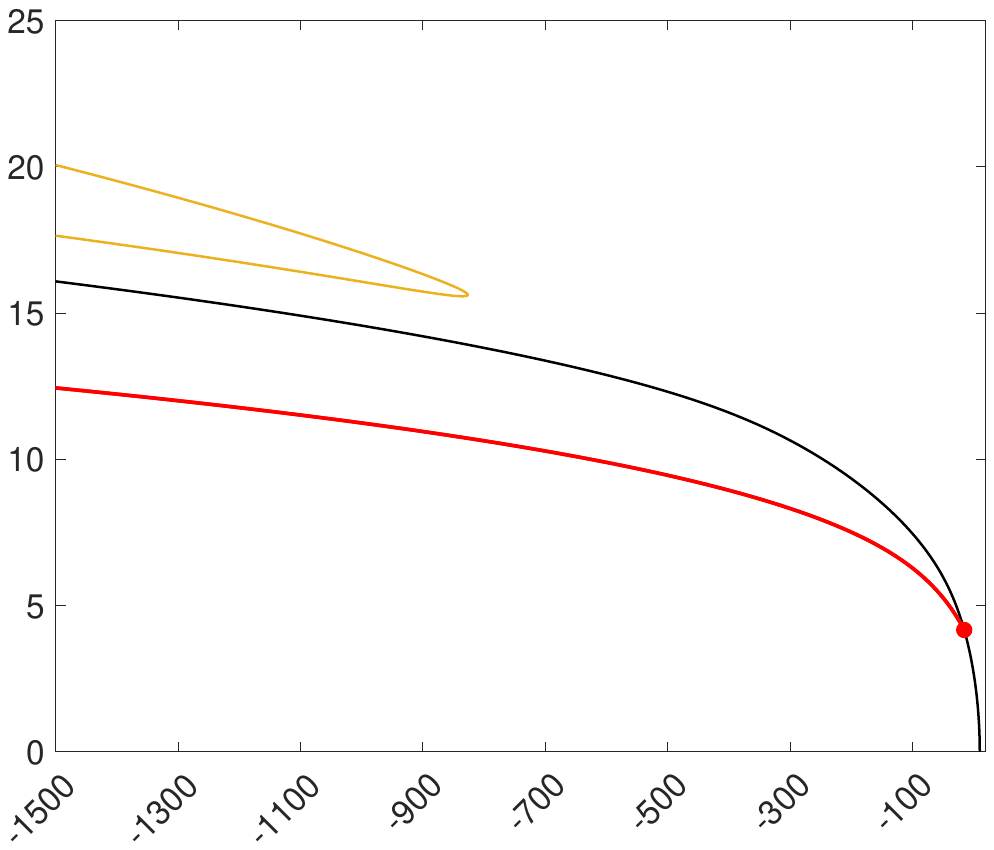}
		\put (12.5,75.5) {\tiny$\|u\|_2$}
		\put (97.5,7.5) {\tiny$\l$}
	\end{overpic} \qquad
\begin{overpic}[scale=0.3,trim = 1cm 6.5cm 1cm 7.5cm, clip]{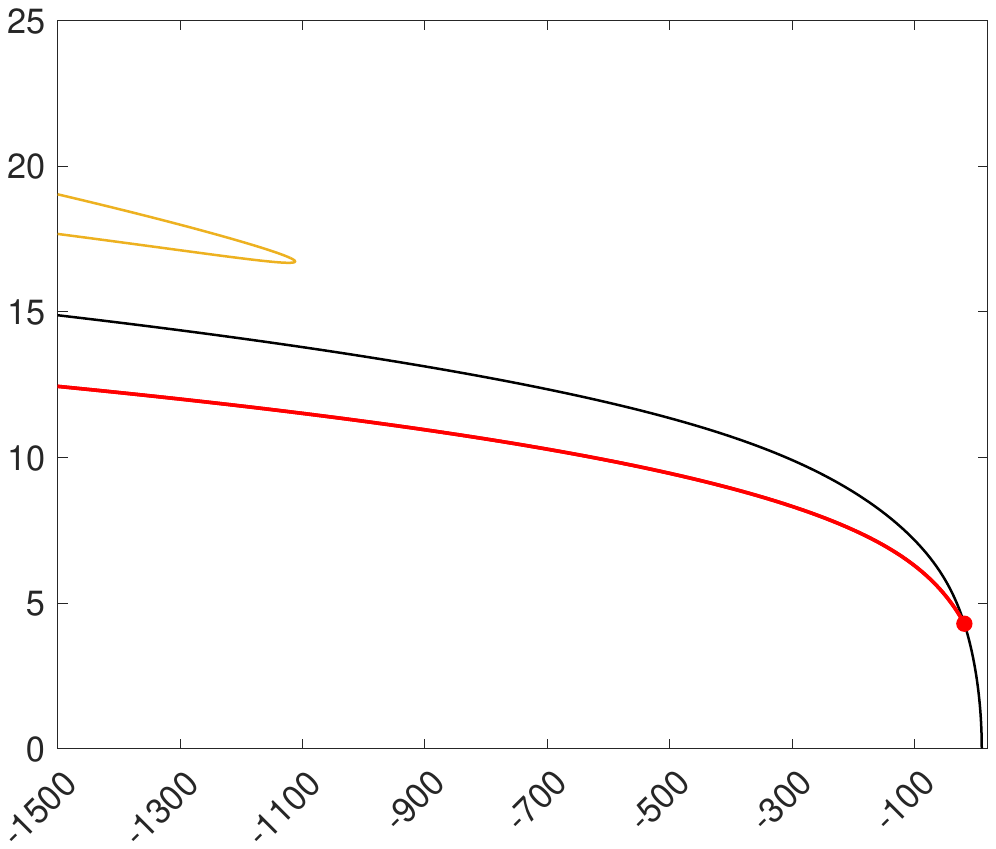}
	\put (12.5,75.5) {\tiny$\|u\|_2$}
	\put (97.5,7.5) {\tiny$\l$}
\end{overpic}
	\caption{Bifurcation diagram for \eqref{1.1} with $a=a_{1,\e}$, $h=0.1$ and $\e=0.60$ (left) and $\e=0.70$ (right).}
	\label{Fig29}
\end{figure}

\begin{table}[h!]
	\caption{Computed values of $\l_{t}$, the value of $\l$ corresponding to the turning point of the isola, for a series of $\varepsilon$.}
	\label{Tab2}
\begin{center}
	
	\begin{tabular}{ |cc|c|cc|c|cc| }
				
		\hline
		
		$\varepsilon$ & $\lambda_t$ & \qquad  & $\varepsilon$ & $\lambda_t$ & \qquad  & $\varepsilon$ & $\lambda_t$ \\
		
		\hline
		
		0.10 & -2104.88508 & \qquad  \qquad &  0.44 & -716.95617  & \qquad \qquad & 0.60 & -826.76479 \\
		
		0.20 & -1476.62785 & \qquad \qquad & 0.48 & -594.68502  & \qquad \qquad & 0.65 & -964.72097 \\
		
		0.30 & -1111.65254 & \qquad \qquad & 0.50 & -499.07238 & \qquad \qquad & 0.70 & -1112.24066 \\
		
		0.38 & -881.29507 & \qquad \qquad & 0.51 & -555.55043 & \qquad \qquad  & 0.80 & -1477.83860 \\
		
		0.42 & -772.29017 & \qquad \qquad & 0.55 & -687.82638 & \qquad \qquad  & 0.90 & -2107.12751 \\
		
		\hline
		
	\end{tabular}
	
\end{center}
\end{table}

Finally, to conclude this section, we present the results of our simulations for some values of $\e\in(0,1)$ which approach the limiting cases. Precisely, Figure \ref{Fig30} shows the bifurcation diagrams for $\e=0.1$ (left) and $\e=0.9$ (right). These diagrams, together with the values collected in Table \ref{Tab2}, leads us to conjecture that, for all $\e\in(0,1)$ there exists $\l_t=\l_t(\e)$ such that problem \eqref{1.1} with $a=a_{1,\e}$ possesses at least 5 positive solutions for $\l<\l_t$ and that
\begin{equation*}
	\lim_{\e\da 0}\l_t(\e)=-\infty, \qquad \lim_{\e\ua 1}\l_t(\e)=-\infty.
\end{equation*}
Up to the best of our knowledge these phenomena have been observed here for the first time, and no analytical tool is available to prove them theoretically.

\begin{figure}[ht!]
	\centering
	\begin{overpic}[scale=0.3,trim = 1cm 6.5cm 1cm 7.5cm, clip]{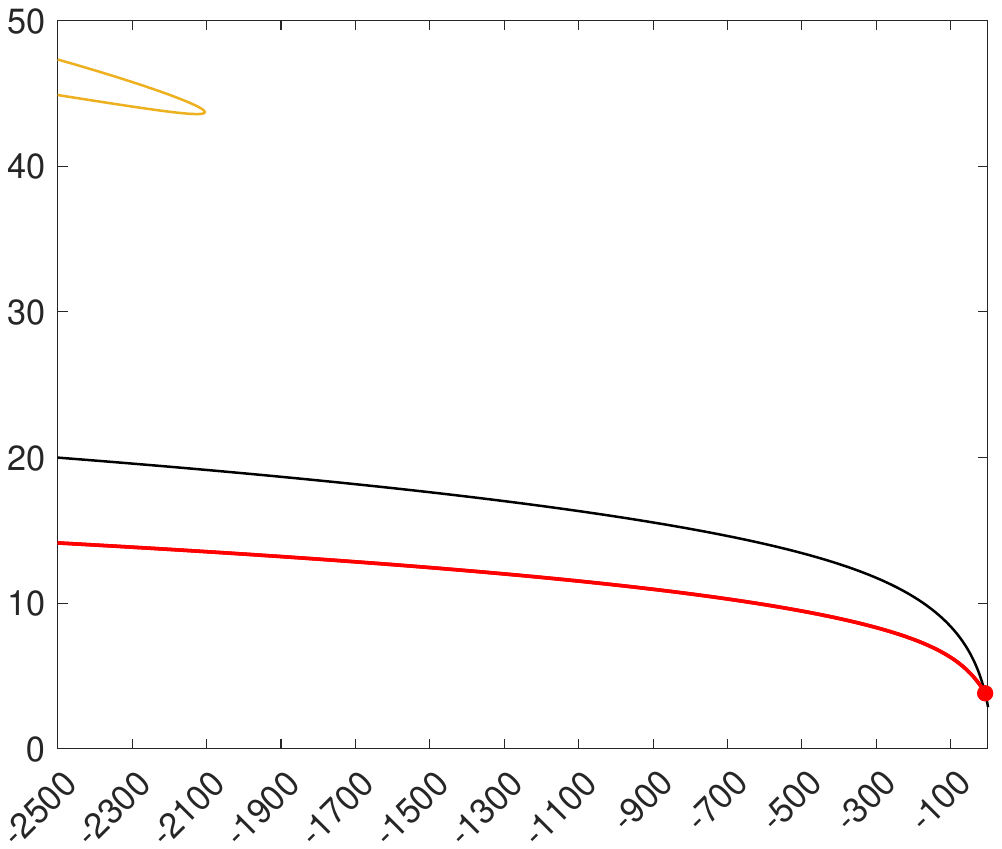}
		\put (12.5,75.5) {\tiny$\|u\|_2$}
		\put (97.5,7.5) {\tiny$\l$}
	\end{overpic} \qquad \begin{overpic}[scale=0.3,trim = 1cm 6.5cm 1cm 7.5cm, clip]{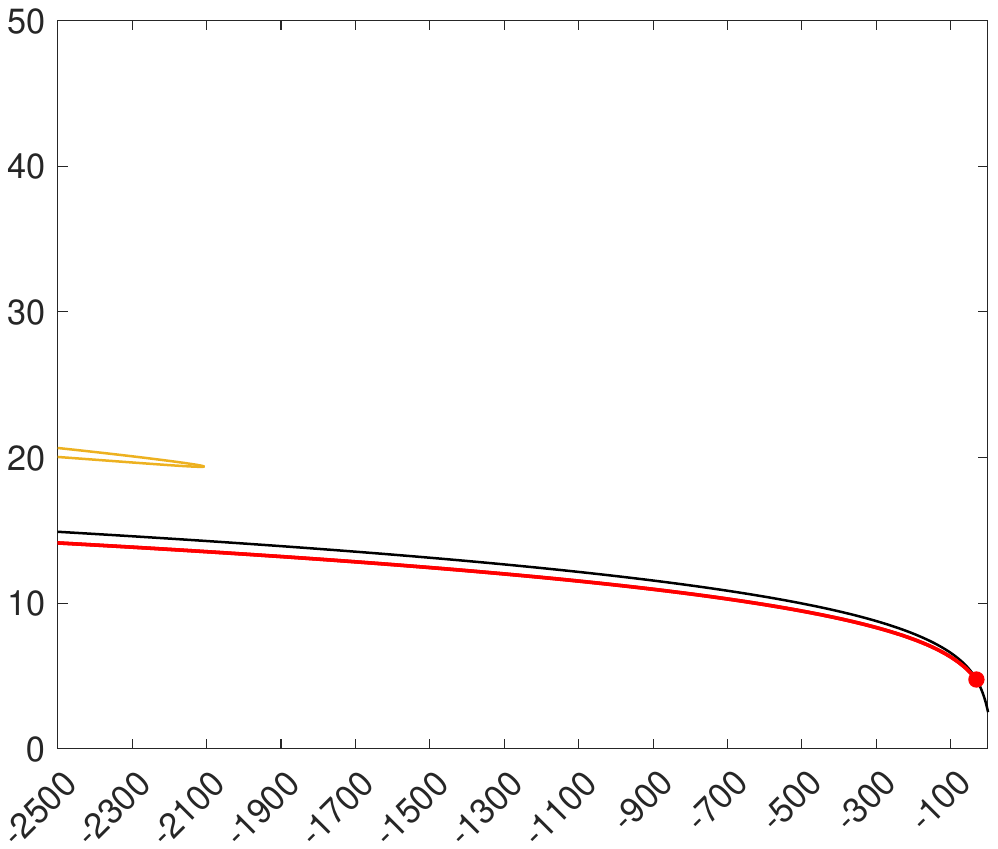}
	\put (12.5,75.5) {\tiny$\|u\|_2$}
	\put (97.5,7.5) {\tiny$\l$}
\end{overpic}
	\caption{Bifurcation diagram for \eqref{1.1} with $a=a_{1,\e}$, $h=0.1$ and $\e=0.10$ (left) and $\e=0.90$ (right).}
	\label{Fig30}
\end{figure}

\section{Technical aspects related to the numerical experiments}
\label{sec:7}

\noindent For an extensive discussion and historical review of the numerical methods used to discretize complex nonlinear bifurcation problems like the one considered in this work, we refer the reader to our previous papers \cite{CLGT,CLGTS} and the references therein. Although pseudo-spectral methods are more efficient to determine the precise location of bifurcation points and bifurcation directions from trivial solutions, and have proven to be very efficient to approximate large compact arcs of positive solutions in a variety of \emph{sublinear and superlinear indefinite problems} as those considered in \cite{GRLGNA,LGEDMM,LGMMJDE,LGMMT,LGMMT2,LGMMR,FLG,FLGN,MMPM}, in this paper we have preferred to discretize the \emph{pure superlinear  problem} \eqref{1.1} by means of a centered finite difference scheme, with a non-uniform mesh, because this requires less computational cost and provide with rather satisfactory results, even in computing solutions with large internal gradients like those arising in Figures \ref{Fig10}, \ref{Fig12}, \ref{Fig19}. The non-uniform mesh had to be specifically adapted in order to compute the positive solutions of \eqref{1.1} with $a=a_{\kappa,0}$ for extremely small $h$'s. However, to compensate for the inconvenience of getting rough approximations of the true bifurcation values from $u=0$, the finite difference scheme requires a much higher number of internal mesh points $N$ compared to the number of collocation points in the pseudo-spectral method. Indeed, the $k$-th eigenvalue ($k\leq N$) of the tridiagonal Toeplitz matrix corresponding to the discretization of the operator $-D^2=-\frac{\mathrm{d}^2}{\mathrm{d}x^2}$ is given by
\begin{equation}
	\label{7.1}
	\tilde{\l}_k=\tilde{\l}_k(N)=2(N+1)^2\left(1+\cos\left(\frac{N+1-k}{N+1}\pi\right)\right)
\end{equation}
(see, e.g., Section 2 of Gover \cite{G}), which converges to the exact value $\l_k=(k\pi)^2$ as $N\to+\infty$. Our numerical simulations confirm this result, as shown in Table \ref{Tab3}, where we indicate, for different values of $N$, the computed value of the bifurcation point of symmetric solutions from the trivial branch, as well as the computational time required to calculate these symmetric solutions up to $\l=-100$ with a fixed Keller--Yang step size equal to $3$ (see the last paragraph of this section for an explanation of these features). The simulations to obtain these values have been performed with a weight function $a_{1,0}$ defined as in Section \ref{sec:3} for $h=0.5$. This table highlights the fact that a compromise is required between precision and speed in order to calculate the bifurcation diagrams.

\begin{table}[h!]
	\caption{Values of the approximate computed value of $\l_1$, and computational time $t$ (in seconds) required to calculate the symmetric positive solutions for the weight $a_{1,0}$ with $h=0.5$ up to $\l=-100$ with a fixed step $\Delta s=3$, for different number of internal mesh points $N$. The approximate values of $\l_1$ agree with formula \eqref{7.1}.}
	\label{Tab3}
	\begin{center}
		\vline
		\begin{tabular}{cccc}
			\hline
			$N$ & Approximate computed $\l_1$ & Computational time (s) \\
			\hline
			100 & 9.868808627128601 & 2.217 \\
			200 & 9.869403719902039 & 2.656 \\
			500 & 9.869571805000305 & 4.296 \\
			800 & 9.869591832160950 & 7.295\\
			1000 & 9.869596123695374 & 10.852\\
			2000 & 9.869602560997009 & 32.766\\
			\hline
		\end{tabular}
		\hspace{-0.2cm}  \vline
	\end{center}
\end{table}

\par
As illustrated by Table \ref{Tab3}, increasing $N$ above $500$ does not substantially improve precision but has the serious handicap of increasing considerably the required computational time. For this reason, we have carried out the majority of the simulations of this work with $N=500$. Naturally, the computational performance is closely related to the hardware used, and, in this paper, calculations have been performed on a 10-core CPU M1 Macbook Pro with 16 Gb of RAM without making use of the GPU.

\par
Nonetheless, it should be emphasized that, when dealing with weight functions $a=a_{\k,\e}$ in front of the nonlinear term of \eqref{1.1} with extremely small values of $h$, an additional difficulty arise. Indeed, in such cases, the discretization of the interval $[0,1]$, i.e.,  the node distribution of the mesh must guarantee the existence of nodes within the subintervals of length $h$ where $a(x)$ vanishes. When using uniform meshes to discretize problem \eqref{1.1} in such extreme cases, a very high number of mesh points should be chosen to comply with the above condition, which would render computations extremely slow or even unaffordable.

\par
In order to overcome this difficulty, a modified Keller-Yang scheme has been implemented to deal with such cases, allowing for computations on  non-uniform meshes. These meshes are still symmetric with respect to $0.5$ but the amount of internal nodes on the vanishing subsets of $a(x)$ has been greatly increased in comparison to other subsets of the domain where such a refinement is unnecessary, keeping the total number of nodes, and thus also computational times, at a reasonable level. We would like to stress the importance of these meshes to be perfectly symmetric; otherwise, unlike the original continuous problem, the discretized one will be non-symmetric, leading to a completely different bifurcation diagram, where bifurcations from simple eigenvalues lead to genuine imperfect bifurcations due to symmetry breaking.
\par
To better illustrate the challenges of such a problem, let us give an example among all the computations performed for this work. For $a=a_{1,0}$ with $h=10^{-5}$, we used a non-uniform mesh containing $201$ nodes concentrated in $[0.4999,0.5001]$ thanks to a mesh size $\Delta x=10^{-6}$ and $490$ nodes in both $[0.01,0.499]$ and $[0.501,0.99]$ thanks to a mesh size $\Delta x=10^{-3}$, which amounts to a total of $1181$ internal nodes in $[0,1]$. Since $h$ is extremely small, this mesh results in only $9$ nodes where $a_{1,0}=0$, but still enough to take into account the effect of this very small subset where the weight vanishes. If the same mesh step $\Delta x=10^{-6}$ had been used for the entire interval $[0,1]$, the resulting uniform mesh would have consisted of $999999$ internal nodes, making the corresponding simulations completely unfeasible. Actually, this is a paradigm of why the existing packages, such as AUTO-07P, are of no utility when dealing with differential equations in the presence of spatial heterogeneities, like in problem \eqref{1.1}. Doedel and Oldeman already emphasized on page 185 of the AUTO-07P manual \cite{DO} that
\begin{quote}
	\emph{ \lq \lq given the non-adaptive spatial discretization, the computational procedure here is not appropriate for PDEs with solutions that rapidly vary in space, and care must be taken to recognize spurious solutions and bifurcations.\rq\rq}
\end{quote}
Indeed, this is  one of the main difficulties encountered in our numerical experiments.

\par
Apart from the convergence of the approximated bifurcation point from trivial solutions,  the local convergence of solution paths as $N\to+\infty$ at regular, turning and all simple bifurcation points was proved, for general Galerkin approximations, by Brezzi, Rappaz and Raviart \cite{BRR1,BRR2,BRR3}, and by L\'{o}pez-G\'{o}mez,  Molina-Meyer and  Villareal \cite{LGMMV} and L\'{o}pez-G\'{o}mez,  Eilbeck, Duncan and  Molina-Meyer in \cite{LGEDMM} for  codimension two singularities. Such results are based on the implicit function theorem (see, e.g., Theorem 3.2 of \cite{LG88}). As a consequence, in these situations, the local structure of the solution sets for the continuous and the discrete models are equivalent provided that a sufficiently fine discretization is performed.

\par
The global continuation solvers used to compute the solution curves of this paper have been built by means of the theory developped by  Allgower and Georg \cite{AG},
Crouzeix and Rappaz \cite{CrRa}, Eilbeck \cite{Ei}, Keller \cite{Ke}, L\'{o}pez-G\'{o}mez \cite{LG88} and L\'{o}pez-G\'{o}mez, Eilbeck, Duncan and Molina-Meyer \cite{LGEDMM}.
\par
In order to compute the components of symmetric solutions throughout this paper we started at a value of $\l$ lower than, but sufficiently close to,  $\pi^2$ and used the initial iterate $u_0=\sin(\pi x)$ for Newton's method, since, for such $\l$'s,  this provides us with a first-order approximation having a similar shape as the positive solutions bifurcating from $u=0$. Our correction algorithm, based on Newton's method, easily converged to a solution of problem \eqref{1.1}. Then, by performing a global continuation based on the algorithm by Keller and Yang \cite{KY}, and detecting whether some secondary bifurcation occurred, we were able to compute the entire solution component containing the symmetric solutions.
\par
As for the global subcritical folds arising in Sections \ref{sec:4} and \ref{sec:5}, which have been referred to as isolas throughout this work, we can ascertain the shape of the solutions lying on such branches from the number of the maxima of these solutions and their approximate position, in the light of Theorem \ref{th2.1} and Remark \ref{re2.2}. Based on this theoretical information, we were able to construct adequate predictions to globally compute these components, as explained in \cite{CLGT}.
\par
Regarding the isolas of Section \ref{sec:6}, instead, they have been detected thanks to the recombination of the solutions that occurs as $\e$ crosses the critical value $\e^*$. Indeed, if one compares the symmetric solutions on $\mathscr{C}_0^+$ for $\e=0.50$ and $\e=0.51$, one observes a sudden change in their profiles, since there are two peaks for $\e=0.50$, while only one peak is present for $\e=0.51$. Thus, assuming a certain underlying continuity, in the latter case, one is lead to look for the solutions with two peaks elsewhere. Precisely, if one starts with a prediction having two peaks, the correction provided by the Newton algorithm allows one to detect a first point on the isola. Similarly, for $\e=0.50$, one can initialize with a profile with one peak, and the Newton algorithm converges towards a solution having such a shape. Then, the isola is obtained by performing a global continuation, as explained above.
\par
The discrete $L^2$-norm used in this paper to represent the solutions in the bifurcation diagrams is defined as
\begin{equation*}
 \| u \|_{2} = \left( \sum_{i=1}^{N}(x_{i}-x_{i-1})  u_i^2 \right) ^{1/2},
\end{equation*}
where
\[
   0=x_0<x_1<\ldots<x_{N}<x_{N+1}=1
\]
are the mesh points and $u_i$, $i\in\{1,\ldots,N\}$, is the approximate value for $u(x_i)$ (recall that in our problem $u(x_0)=u(x_{N+1})=0$). This norm is the discrete analog of the classical norm $L^2(0,1)$. Indeed, as $N\to+\infty$, $\| u \|_{\Delta x}$ converges to $\|u\|_{L^2}$.

\par
Finally, we give some details on the computational parameters used in this paper to obtain the bifurcation diagrams. As a stopping criterion for Newton's method, we considered an iteration to be satisfactory when the norm of our \emph{augmented system}  evaluated at such an iteration, is smaller than $0.0001$. By  \lq\lq augmented system" we mean the system associated with the   Keller and Yang's algorithm \cite{KY} that makes regular the (singular) turning points through a pseudo-arc-length parametrization of fixed step size $\Delta s$ equal to $3$. In order to better approximate bifurcation points, a bisection method with the same tolerance  was implemented, approaching points where the determinant of the Jacobian of the augmented system vanishes.
\vspace{0.2cm} 

\noindent \textbf{Acknowledgements.} The authors thank Professor M. Molina-Meyer (Carlos III University, Madrid, Spain) for some useful observations concerning the presentation of Section \ref{sec:7}.

\end{document}